\documentclass[11pt, reqno]{amsart}
\usepackage{textcmds} 
\usepackage[shortlabels]{enumitem}
\usepackage{amsmath, amssymb, amsfonts, amstext, verbatim, amsthm, mathrsfs}
\usepackage{microtype}
\usepackage[all]{xy}
\usepackage[modulo]{lineno}
\usepackage{aliascnt}
\usepackage{enumitem}
\usepackage{xspace}
\usepackage{amsfonts}
\usepackage{amssymb}
\usepackage[centertags]{amsmath}
\usepackage{amsthm}
\usepackage{rotating}
\usepackage[margin=3.25cm]{geometry}
\usepackage{dsfont}
\usepackage{bm}
\usepackage{subfigure}
\usepackage{amsmath}
\usepackage{array}
\usepackage[all]{xy}
\usepackage{euscript}
\usepackage[T1]{fontenc}
\usepackage{mathbbol}
\usepackage{nccmath}
\usepackage{marginnote}
\usepackage{stmaryrd}
\usepackage{youngtab}
\usepackage{tikz}
\usepackage{blkarray}
\usepackage{chngcntr}
 \usepackage[abs]{overpic}

\usepackage{tikz}

\usepackage{graphics,graphicx}  

\let\oldtocsection=\tocsection
\let\oldtocsubsection=\tocsubsection

\renewcommand{\tocsection}[2]{\hspace{0em}\oldtocsection{#1}{#2}}
\renewcommand{\tocsubsection}[2]{\hspace{1em}\oldtocsubsection{#1}{#2}}

\usepackage[colorlinks=true,linkcolor=black,citecolor=blue,urlcolor=blue,citebordercolor={0 0 1},urlbordercolor={0 0 1},linkbordercolor={0 0 1}]{hyperref} 

\tikzset{node distance=3cm, auto}

\makeatletter
\def\@secnumfont{\bfseries}
\def\section{\@startsection{section}{1}%
  \z@{.7\linespacing\@plus\linespacing}{.5\linespacing}%
  {\normalfont\Large\bfseries}}
\def\subsection{\@startsection{subsection}{2}%
  \z@{.75\linespacing\@plus.7\linespacing}{-.5em}%
  {\normalfont\large\bfseries}}

  \def\subsubsection{\@startsection{subsubsection}{3}%
  \z@{.75\linespacing\@plus.7\linespacing}{-.5em}%
  {\normalfont\bfseries}}
\makeatother

\newtheorem{thm}{Theorem}[subsection]
\newtheorem{lemma}[thm]{Lemma}
\newtheorem{prop}[thm]{Proposition}
\newtheorem{cor}[thm]{Corollary}
\newtheorem{question}[thm]{Question}

\newtheorem{definition}[thm]{Definition}

\theoremstyle{remark}
\numberwithin{equation}{subsection} 
 \numberwithin{figure}{section}

\newtheoremstyle{customremark}
{3pt}
{3pt}
{}
{}
{\bfseries}
{.}
{.5em}
{}
\theoremstyle{customremark}
\newtheorem{rmk_no_diamond}[thm]{Remark}
\newenvironment{rmk}{\begin{rmk_no_diamond} } {\hfill$\er$ \end{rmk_no_diamond}}
\newtheorem{example_no_diamond}[thm]{Example}
\newenvironment{example}{\begin{example_no_diamond} } {\hfill$\er$ \end{example_no_diamond}}

\newcommand{\SO}{{\rm SO}}

\newcommand{\se}{{\stackrel{s}\hookrightarrow}}

\newcommand{\Ii}{\mathcal{I}}

\newcommand{\Vol}{{\rm Vol}}
\newcommand{\vol}{{\rm Vol}}
\newcommand{\Per}{{\rm Per}}

\newenvironment{itemlist}
   { \begin{list} {$\bullet$}
         { \setlength{\topsep}{.5ex}  \setlength{\itemsep}{.5ex} \setlength{\leftmargin}{2.5ex} } }
   { \end{list} }
\newcommand{\SL}{{\rm SL}}
\newcommand{\Aff}{{\rm Aff}}
\newcommand{\bbm}{{\bf m}}
\newcommand{\bw}{{\bf w}}

\newcommand{\TE}{{\widetilde E}}

\newcommand{\Tm}{{\widetilde m}}

\newcommand{\intt}{{\rm int}}
\newcommand{\NI}{{\noindent}}

\newcommand{\Cc}{{\mathcal C}}

\newcommand{\Mm}{{\mathcal M}}
\newcommand{\Hh}{{\mathcal H}}
\newcommand{\Ss}{{\mathcal S}}
\newcommand{\bE}{{\mathbb E}}

\newcommand{\ov}{\overline}
\newcommand{\al}{{\alpha}}

\newcommand{\be}{{\beta}}

\newcommand{\Om}{{\Omega}}
\newcommand{\om}{{\omega}}
\newcommand{\de}{{\delta}}

\newcommand{\io}{{\iota}}
\newcommand{\ka}{{\kappa}}
\newcommand{\la}{{\lambda}}
\newcommand{\La}{{\Lambda}}
\newcommand{\si}{{\sigma}}

\newcommand{\less} {{\smallsetminus}}
\newcommand{\p}{{\partial}}
\newcommand{\MS}{{\medskip}}

\newcommand{\er}{{\Diamond}}

\newcommand{\Z}{\mathbb{Z}}

\newcommand{\R}{\mathbb{R}}
\newcommand{\Q}{\mathbb{Q}}
\newcommand{\C}{\mathbb{C}}
\newcommand{\CP}{\mathbb{CP}}

\newcommand{\eps}{\varepsilon}

\newcommand{\op}[1]{{\operatorname{#1}}}

\newcommand{\ind}{\op{ind}}

\newcommand{\oCP}{{\overline{{\C}P}}\!\,}
\newcommand{\btm}{\widetilde{\bm{m}}}
\newcommand{\bB}{{\bf{B}}}

\makeatletter
\newcommand{\dashover}[2][\mathop]{#1{\mathpalette\df@over{{\dashfill}{#2}}}}
\newcommand{\fillover}[2][\mathop]{#1{\mathpalette\df@over{{\solidfill}{#2}}}}
\newcommand{\df@over}[2]{\df@@over#1#2}
\newcommand\df@@over[3]{%
  \vbox{
    \offinterlineskip
    \ialign{##\cr
      #2{#1}\cr
      \noalign{\kern1pt}
      $\m@th#1#3$\cr
    }
  }%
}
\newcommand{\dashfill}[1]{%
  \kern-.5pt
  \xleaders\hbox{\kern.5pt\vrule height.4pt width \dash@width{#1}\kern.5pt}\hfill
  \kern-.5pt
}
\newcommand{\dash@width}[1]{%
  \ifx#1\displaystyle
    2pt
  \else
    \ifx#1\textstyle
      1.5pt
    \else
      \ifx#1\scriptstyle
        1.25pt
      \else
        \ifx#1\scriptscriptstyle
          1pt
        \fi
      \fi
    \fi
  \fi
}
\newcommand{\solidfill}[1]{\leaders\hrule\hfill}

\newcommand{\oset}[3][0ex]{%
  \mathrel{\mathop{#3}\limits^{
    \vbox to#1{\kern-2\ex@
    \hbox{$\scriptstyle#2$}\vss}}}}

\makeatother

\date{\today}

\title{Curvy points, the perimeter, and the complexity of 
convex toric domains}

\author{Dan Cristofaro-Gardiner}
\address{Mathematics Department, University of Maryland, College Park}
\email{dcristof@umd.edu}
\thanks{DCG thanks the NSF for their support under agreements DMS-2227372 and DMS-2238091}
\author{Nicki Magill}
\address{Mathematics Department, UC Berkeley}
\email{nmagill@berkeley.edu}
\thanks{NM thanks the NSF for their support under the agreement DMS-2402169}
\author{Dusa McDuff}
\address{Mathematics Department, Barnard College}
\email{dusa@math.columbia.edu}

\begin{document}

\begin{abstract}

We study the related notions of curvature and perimeter for toric boundaries and their implications for symplectic packing problems in dimension $4$; a natural setting for this is a generalized version of convex toric domain which we also study, where there are no conditions on the moment polytope at all aside from convexity.

We show that the subleading asymptotics of the ECH and elementary ECH capacities recover the perimeter of such domains in their liminf, without any genericity required, and hence the perimeter is an obstruction to a full filling.  As an application, we give the first examples of the failure of packing stability by open subsets of compact manifolds with smooth boundary or with no boundary at all; this has implications for long-term super-recurrence.   We also show that a single smooth point of positive curvature on the toric boundary obstructs the existence of an infinite staircase, and we build on this to completely classify smooth (generalized) convex toric domains which have an infinite staircase.  We also extend a number of theorems to generalized convex toric domains, in particular the  ``concave to convex\rq\rq\, embedding theorem and the ``accumulation point theorem\rq\rq.   A curvy point forces ``infinite complexity\rq\rq; we raise the question of whether an infinitely complex domain can ever have an infinite staircase and we give examples with infinite staircases and arbitrarily high finite complexity.
\end{abstract}

\maketitle

\tableofcontents

\section{Introduction}

Let $\Om\subset \R^2_{\ge 0}$ be a compact convex region with boundary $\p \Om$, and $X_\Om: = \Phi^{-1}(\Om)$  the corresponding four-dimensional symplectic domain,
where $$
\Phi:\C^2\to 
\R^2_{\ge 0},\qquad (\zeta_1, \zeta_2)\mapsto (\pi |\zeta_1|^2 , \pi |\zeta_2|^2)
$$ is the moment map.  The symplectic geometry of these domains has been of considerable interest (see e.g. \cite{CG, cghind, AADT, Hutchq, Ruel, Helem, Usher, BHM, Choi, CG2, AADT, JinL, ball}).

A basic observation is that if $\Om$ lies entirely off the axes, then up to symplectomorphism $X_\Om$ depends only on  $\Om$ up to affine equivalence, i.e. integral affine transformations $\vec{v} \mapsto\vec a + A \vec{v}$, where the matrix $A$ is integral.  (Points on the axes are special  since their preimage under the moment map is a point or circle.)  It is therefore natural to study properties of $\Om$ that are preserved under this equivalence.  One goal of the present work is to study the implications of two preserved notions --- the (affine) perimeter and the existence of a positively curved point on $\partial \Om$  ---  for symplectic embedding problems.  
We give several applications of this point of view.

At the same time, we are also interested in generalizing existing theory in the following sense. 
 Previous work has often required that $\Om$ contain a neighborhood of the origin (in which case $X_\Om$ has been called a {\bf convex toric domain}) or that it be a rational convex polytope.  We make neither of these assumptions, requiring only that $\Om\subset \R^2_{\ge 0}$ be compact and convex.  We could call these {\bf generalized convex toric domains}, though since all of our theorems in this paper will be valid in this more general setting, we will usually continue to call them convex toric domains for simplicity.

Let us now summarize our main results.

\subsection{Curvy points and the perimeter}

As a starting point for explaining our results,  
let us begin with the following question originating in dynamics.  

Let $(M,\omega)$ be a symplectic manifold of finite volume, $\Psi$ a Hamiltonian diffeomorphism, and fix an open subset $U \subset M$.  In this situation, ``Poincare recurrence\rq\rq\, guarantees that $\Psi^i(U)$ must intersect $U$ nontrivially for some $1 \le i \le \lfloor \frac{ \vol(M)}{\vol(U)} \rfloor$.  It is a longstanding problem, see \cite{polterovichschlenk},
to better understand for what kind of open sets this bound on $i$ can be improved in the ``critical case\rq\rq\, when the volume of $U$ actually divides the volume of $M$.   To make this precise, 
let us say that {\bf long term super-recurrence holds} (which we will sometimes just call super-recurrence\footnote{The definition in \cite{polterovichschlenk} is slightly different from this.  Strictly speaking, there one only requires a sequence of arbitrarily small scalings improving this Poincare recurrence property.  Anything satisfying our condition satisfies the condition in \cite{polterovichschlenk} as well.} 
for short) for an open subset $U \subset M$ if whenever $U' \subset M$ is such that $U'$ is symplectomorphic to a scaling of $U$ and $\vol(U')$ properly divides $\vol(M)$, 
\[ \Psi^k(U') \cap U' \ne \lbrace \emptyset \rbrace,\]
for some $1 \le k \le \frac{\vol(M)}{\vol(U')} - 1.$


It is useful to view super-recurrence through the lens of symplectic packing problems.  Indeed, a closely related notion is that of \lq\lq packing stability\rq\rq.  Recall that one (possibly disconnected) symplectic manifold $(X_1,\omega_1)$ {\bf fully fills} another $(X_2,\omega_2)$  if there is a symplectic embedding of $c \cdot X_1$ into $X_2$ whenever  $\vol(c \cdot X_1) < \vol(X_2)$.  Let $\sqcup$ denote the disjoint union.  We say that {\bf packing stability} holds for $(X_1,\omega_1)$ into $(X_2,\omega_2)$ if $\sqcup^d_{i=1} (X_1,\omega_1)$ fully fills $(X_2,\omega_2)$ for all sufficiently large $d$.    A wide reaching conjecture by Schlenk~\cite{schl}, asserts that packing stability holds 
when $X_1$ is any bounded domain in $\mathbb{R}^{2n}.$

For example, when $M = \mathbb{C}P^2$ and $U$ is an open ball, it follows from the 
work of Biran \cite{biran} that packing stability holds, 
so that  long-term super-recurrence does not occur.  On the other hand, the recent work \cite{cghind} produced open manifolds $M$ such that super-recurrence holds for every $U \subset M$ with smooth boundary. Thus Schlenk's conjecture fails for these manifolds.  However, the manifolds in \cite{cghind} have quite complicated boundaries, so one would like to better understand the situation in the closed case or the case with smooth boundaries; for example, one would like to know whether or not Schlenk's conjecture can fail in this case.

Our first theorem gives natural examples answering this question, via a new kind of packing phenomenon.  For a (generalized) convex toric domain $X_{\Omega}$, let $\Per(X_{\Omega})$ denote the $\SL_2(\mathbb{Z})$-perimeter of $\Omega$, i.e. the affine length of the boundary, see \S\ref{ss:length}.  For a closed symplectic manifold $(X,\omega)$ let $\Per(X) = c_1(\omega) \cdot [\omega]$.\footnote
{This interpretation is justified since, when $M$ is a toric manifold, $c_1(\omega) \cdot [\omega]$ is the affine length of the boundary of its moment polytope. See also Remark~\ref{rmk:curvature}.} Further, we write $X\se Y$ if there is a symplectic embedding of $X$ into $Y$ where $X,Y$ are symplectic manifolds of the same dimension.

\begin{thm}
\label{thm:packing} Let
$X_{\Omega_1}, \ldots, X_{\Omega_n}$ be generalized convex toric domains, and let $X$ be either a generalized convex toric domain or $\mathbb{C}P^2$.  Assume that there exists a full filling 
\[ X_{\Omega_1} \sqcup \ldots \sqcup X_{\Omega_n} \se X\]

Then
\[ \sum_{i=1}^n \Per(X_{\Omega_i}) \ge \Per(X).\]
\end{thm}

Theorem~\ref{thm:packing} is a consequence of a refined version of an ECH ``Weyl law\rq\rq\, that we will introduce in \S\ref{sec:without}.
We expect that Theorem~\ref{thm:packing} applies to many other closed symplectic $4$-manifolds, but we have focused on the case of $\mathbb{C}P^2$ for simplicity. The proof of Theorem~\ref{thm:packing} is found in Section~\ref{ss:fullFilling}.  The novelty of Theorem~\ref{thm:packing} is the very general setting in which it holds; it generalizes \cite[Cor. 1.13]{Ruel} and \cite[Cor. 2]{Wormleighton}.  For a discussion of analogues of Theorem~\ref{thm:packing} in the concave case, see Remark~\ref{rem:wormleighton}.

Theorem~\ref{thm:packing} has the following implication for super-recurrence and packing stability.   Let us say that a convex toric domain $X_{\Omega}$ has {\bf zero perimeter} if its boundary contains no line segments of rational slope.
As an example, $X_{\Omega}$ for $\Omega = \lbrace (x-1)^2 + (y-1)^2 \le 1/2 \rbrace$ has zero perimeter.
 
\begin{cor}
\label{cor:super}
A finite number of zero perimeter domains $X_{\Omega_1}, \ldots, X_{\Omega_n}$ can never fully fill a ball or $\mathbb{C}P^2$.  In particular, long term super-recurrence occurs for any open zero perimeter domain in a four-dimensional ball or in $\mathbb{C}P^2$.   
\end{cor}

For a previous case of super-recurrence (though not necessarily a case of the failure of packing stability) via a different kind of phenomenon, see \cite{traynorexper}.   
The problem of super-recurrence is not directly addressed in \cite{traynorexper}, but \cite[Thm. 1.1]{traynorexper} 
gives some examples where packing by Hamiltonian images of a set must be given by (particularly simple) ``simplex packings", and it is not hard to find examples where these simplex packings can not fill the full volume; the ideas behind \cite[Thm. 1.1]{traynorexper} build on work of Sikorav \cite{sikorav}.


 
\begin{rmk}
\label{rmk:example}
In contrast to Corollary~\ref{cor:super}, there certainly exist finite collections of zero perimeter domains filling an arbitrarily large proportion of volume; one can even take these domains to be rescaled copies of a single domain.  For example, one can take $\Omega$ to be a square off the axes in $\mathbb{R}^2_{\ge 0}$ with edges of irrational slope and fill at least any ratio $r < 1$ of the area of the part of the moment polytope of the ball away from the axes by a finite number of translates of rescaled copies of the square.  One can similarly fill the ball or $\mathbb{C}P^2$ by infinitely many zero perimeter domains.
 \end{rmk}

The simplest class of zero perimeter domains are ones with {\bf curvy boundary}, i.e. $X_{\Omega}$ where the boundary of $\Omega$ is smooth with positive curvature. 
  It turns out that the notion of curvy boundary is also related to a seemingly quite different kind of problem that has attracted considerable interest.  Given symplectic manifolds $X,Y$, we write $X \,\se\,Y$ if  $X$ embeds symplectically in $Y$. Recall the {\bf ellipsoid embedding function} of a closed symplectic $4$-manifold 
\begin{align}\label{eq:ellcap}
 c_M(a) := \text{inf} \lbrace \lambda \hspace{1 mm} | \hspace{1 mm} E(1,a) \se \lambda \cdot M \rbrace,
 \end{align} 
where $E(1,a)$ denotes the ellipsoid $\bigl\{(\zeta_1,\zeta_2)\in \C^2 \ \big| \ \pi\bigl( |\zeta_1|^2+\frac{|\zeta_2|^2}a \bigr) \le 1\bigr\}.$
Much work has gone into understanding this function \cite{ball,CG2,AADT,MMW,BHM,Usher}.   
In particular, while it is continuous, it is known that  the function $c_M$ can have infinitely many nonsmooth points on a compact interval.  In this case we say that $(M,\omega)$ {\bf has an infinite staircase} and a main question in the area is to classify for which $M$ an infinite staircase can occur, for natural families of $M$.
It turns out that consideration of curvature allows us to make considerable progress on this.
 Let us say that $\Omega$ has a {\bf curvy point} if there is a $p \in \partial \Omega$ such that $\partial \Omega$ is smooth in a neighborhood of $p$, with positive curvature.

\begin{thm}
\label{thm:curvy}
Let $X_{\Omega}$ be a convex toric domain such that $\Omega$ has a curvy point.  Then $\Omega$ does not have an infinite staircase.
\end{thm}

A more precise version of Theorem~\ref{thm:curvy} is proved in Proposition~\ref{prop:round}.
By combining the above theorem with 
a generalized ``accumulation point theorem\rq\rq\, (stated in Theorem~\ref{thm:acc} below), we can give a classification result for the following natural class of domains.  Let us say that a convex toric domain $X_{\Omega} \subset \mathbb{R}^4$ is {\bf smooth} if its boundary\footnote
{
By this we mean the $3$-dimensional boundary of the manifold $X_\Om$,  not the boundary of the region $\Om$.}
 $\p X_{\Omega} $ is
smooth.
 For example, an irrational ellipsoid is a smooth convex toric domain, and \cite[Qu.1.4]{CG} asks if it has an infinite staircase.  Many special cases of this question were previously answered in \cite{salinger} but the question in full generality has remained open. We prove the following result in Section \ref{ss:curvy}.

\begin{thm}
\label{thm:class}
Let $X_\Om$ be a smooth convex toric domain.  Then $X_\Om$ has an infinite staircase if and only if $X_\Om$ is a ball, a scaling of an ellipsoid $E(1,2)$, or a scaling of an ellipsoid $E(1,3/2)$.   
\end{thm}

\subsection{Convex toric domains without restrictions}
\label{sec:without}

The proofs of the theorems stated above require the extension of the standard theory of convex toric domains to our more general setting.  We now state the corresponding results.

We note first of all that the definition of the
 associated {\bf weight expansion}\footnote{The weight expansion is discussed in detail in Section \ref{sec:weight}.} of positive real numbers $(b; b_1, \ldots, b_j,\ldots)$  extends  without difficulty to our case.  When the weight expansion is finite we say that $\Omega$ has {\bf finite type}, but since such  $\Om$ are polygons with sides of rational slope it is important to consider the case when there are infinitely many $b_j$. 
The properties of this weight expansion are explored in \S\ref{sec:weight}, while \S\ref{sec:ECH} 
extends the basic technical tools to our more general situation.  
As we mentioned above we also do not want to demand that $\Omega$ includes a neighborhood of the origin;  otherwise, for example, every domain would have  perimeter of positive length.

The first result we state here allows us to study embeddings into our (generalized) toric domains from two different perspectives, both of which are used in our paper; it is a generalization of the ``concave into convex\rq\rq\, theorem of \cite{CG}. The proof of this result is given in Section~\ref{ss:proofmain}. 

Recall that a {\bf concave toric domain} is a toric domain corresponding to a region $\Om'$ that lies between some interval $[0,s]$ on the $x$-axis and the graph of a continuous convex function $f: [0,s]\to \R$ that strictly decreases from $f(0)$ to $f(s)=0$.
For example, an ellipsoid  $E(1,s)$ is both concave and convex, and this is the main concave domain of interest to us in the present work.  

\begin{thm}
\label{thm:main}
Let $X_{\Omega_1}$ be a concave and $X_{\Omega_2}$ a convex toric domain.  Then the following are equivalent:
\begin{itemize}
\item[{\rm (i)}] There is a symplectic embedding $\intt(X_{\Omega_1}) \to \intt(X_{\Omega_2})$.
\item[{\rm (ii)}]  There is a symplectic embedding 
\[ \bigsqcup_i \intt(B(a_i)) \sqcup  \bigsqcup_j\intt(B(b_j)) \to \intt(B(b)),\]
where the $(a_i)$ are the weights of $\Omega_1$ and the $(b;(b_j))$ are the weights of $\Omega_2$.
\item[{\rm (iii)}]  Each ECH capacity $c_k$ satisfies $c_k(\intt(X_{\Omega_1})) \le c_k( \intt(X_{\Omega_2}))$. 
\end{itemize}
\end{thm}

The main new point in this theorem is that $\Omega_2$ is not required to touch the axes.  The arguments in \cite{CG} do not suffice for this, because they use a uniqueness theorem for star-shaped domains that are standard near the boundary and the $X_{\Omega_2}$ in our more general case need not be star-shaped nor even have boundary diffeomorphic to $S^3$.  
Theorem~\ref{thm:main} is proved in \S\ref{sec:ECH}.

Next we state a theorem that builds on Theorem~\ref{thm:main} (ii), extending
 the ``accumulation point theorem\rq\rq\,  from \cite{AADT}  
 to our generalized setting.  The accumulation point theorem is the key result
 that has been used to explore the existence of infinite staircases, and we now state a version of it.

Let $\Vol(X_\Om)$ denote the volume of $X_{\Om}$, normalized to be twice the area of $\Omega$, and let $\Per(\Om)$ denote the affine length of its perimeter as in \S\ref{ss:length}. We write $c_{\Om}$ for
 the ellipsoidal capacity function for $X_\Om$ that was defined in \eqref{eq:ellcap},
and define the
{\bf volume constraint}  $V_\Om(z)$ to be the number $\mu$ such that $\Vol(E(1,z)) = \Vol(\mu X_\Om)$; this is the lower bound on $c_{\Om}$ coming from the classical volume obstruction.  

 \begin{thm}
\label{thm:acc}
 Let $\Om$ be convex and suppose that  $c_\Om$ has infinitely many  nonsmooth points $z_k$. Then:
 \begin{itemize}\item[{\rm (i)}]  The 
 sequence $(z_k)_{k\ge 1}$ converges to the point $a^{\Om}_0$ that is
the unique solution $\ge 1$ of the equation $z^2 - z\Bigl(\frac{\Per(\Om)^2}{\Vol(\Om)} - 2\Bigr) + 1 = 0.$
\item[{\rm (ii)}]  
Further the point $a_0^\Om$ is unobstructed; i.e.  $
c_\Om(a^{\Om}_0) = V_\Om(a^{\Om}_0).
$
\end{itemize}
\end{thm}

Here, the main novelty is that $\Om$ is not required to be ``finite type\rq\rq\,.  It was observed in \cite[Rem. 4.11]{AADT} that the arguments in \cite{AADT} do not suffice to handle the case of infinite weight expansion, and the question of whether or not one can get around this was raised.  Perhaps somewhat surprisingly, our result shows that the theory of \cite{AADT} continues to hold for all convex toric domains, without any restrictions on $\Om$ at all.   This is used in the proof of Theorem~\ref{thm:class}, and we can also use it to rule out infinite staircases for further classes of domains.  Here is one example illustrating a characteristic way to apply Theorem~\ref{thm:acc}.

\begin{example}
Suppose that $\p\Omega$ consists entirely of lines of irrational slope.  Then by Theorem~\ref{thm:acc}, $X_{\Omega}$ does not have an infinite staircase since in this case, $\Per(\Omega) = 0$, so the equation in Theorem~\ref{thm:acc} 
has no  real roots $\ge 1$.  
In fact,  there can be no staircase when $(\Per)^2/\Vol < 4$. See Proposition~\ref{prop:noStairGromov} for a discussion of some further obstructions.
\end{example}

\begin{rmk}\label{rmk:curvature} {\bf (The  perimeter in the closed case.)}\; 
It was observed in \cite{AADT} that the accumulation point theorem holds for an important class of closed manifolds as well.  Namely, if $X$ is a rational symplectic $4$-manifold, i.e. a blowup of $\mathbb{C}P^2$, then the symplectic form is encoded in a (finite)  {\bf blowup vector} $(b; b_1, \ldots, b_n)$ (here, $b$ is the size of the line class and the $b_j$ are the sizes of the blowups), and then \cite{AADT} noted that the arguments to prove the accumulation point theorem hold verbatim to establish the same result, with $\Per = 3b - \sum b_i,$ provided that $\Per\ge 0$. The same argument shows that there is no staircase if $\Per < 0$.
A new observation we make here is that the formula for $\Per$ in fact has a natural geometric interpretation in the closed case, as does the condition of zero or negative perimeter.  Namely, for such $X$, we can write
\[ \Per(X) = c_1(\omega) \cdot [\omega].\]
The quantity $c_1 \cdot [\omega]$ is in turn one of the classical topological invariants of symplectic $4$-manifolds.  By ``Blair's formula\rq\rq\, \cite{blair}, it also has a natural interpretation (up to a universal constant) as the {\bf total scalar curvature}, i.e. the integral of the Hermitian curvature of any compatible metric.  We can therefore rule out infinite staircases for rational symplectic manifolds with nonpositive total curvature: see Corollary~\ref{cor:curvature}.
\end{rmk}

Another kind of generalization --- this time moving from the generic to the non-generic case --- is used to prove Theorem~\ref{thm:packing}.  Let $c_k$ denote either the ECH capacities or the elementary ECH capacities; see \S\ref{ss:elem}.

Recall that by the ``ECH Weyl Law\rq\rq, the $c_k$ detect the volume via their leading order asymptotics.  That is, if we define $e_k := c_k - \sqrt{2k \Vol}$, then the $e_k$ are $o(k^{1/2})$.  Much recent activity has gone into understanding the subleading asymptotics of the $c_k$, i.e. the asymptotics of the $e_k$ \cite{Ruel,cghind,Edtmair}.  For convex toric domains, we prove the following refinement in Section \ref{sec:subECH}:

\begin{thm}
\label{thm:convexper}
Let $X_{\Omega}$ be any convex toric domain.  Then
\[ \liminf_k e_k(X_{\Omega}) = - \frac{\Per(\Omega)}{2}.\]
\end{thm} 

The main novelty of Theorem~\ref{thm:convexper} is that no genericity is required of $\Om$.  Indeed, for generic convex toric domains (with some further hypotheses) it was shown by Hutchings that Theorem~\ref{thm:convexper} holds, and in fact the $e_k$ have a well-defined limit.  However, it has been well-known that for  convex toric domains such as the $4$-ball, the $e_k$ do not
have a limit; Theorem~\ref{thm:convexper} illustrates that even when the $e_k$ do not have a well-defined limit, one can still extract meaningful information from them.  It is also important that we prove Theorem~\ref{thm:convexper} for elementary ECH capacities as well; this is what allows us to access the closed manifold $\mathbb{C}P^2$ in Theorem~\ref{thm:packing}, since the ECH capacities of $\mathbb{C}P^2$ are still not known.

\begin{rmk}
\label{rmk:concave}
Theorem~\ref{thm:convexper} does not hold for disjoint unions.  For example, the disjoint union of two $B(1)$ has the same ECH capacities as an $E(1,2)$; an $E(1,2)$ has $\Per = 4$, while the disjoint union of two $B(1)$ has $\Per = 6$.  Similarly, Theorem~\ref{thm:convexper} does not hold for concave toric domains, because it is not hard to produce examples of concave toric domains with the same ECH capacities but different perimeters\footnote{When in addition one assumes that the domains have symplectomorphic interiors, recent work of Hutchings \cite{zoominar} shows that in fact the domains are the same (up to a reflection), at least for certain convex toric domains; as explained to us by Hutchings, it is natural to conjecture that the same holds for concave domains.}.
On the other hand, as we will see in Lemma \ref{lem:sublead}, it is true that the liminf of the disjoint union is bounded from below by the sum of the liminfs, which is used to study disjoint unions in Theorem~\ref{thm:packing}.
\end{rmk}

As another illustration of Theorem~\ref{thm:convexper}, we explain a new kind of embedding phenomenon related to the accumulation point discussed above in connection with Theorem~\ref{thm:acc}.  All previous theorems about the accumulation point concern obstructing infinite staircases.  Here is a different kind of result proved in Section \ref{ss:fullFilling}:

\begin{cor}
\label{cor:new}
Let $X_{\Omega}$ be a convex toric domain and let $a_0$ denote its accumulation point.  Then 
\[ c_\Om(a) > V_\Om(a)\]
whenever $a < a_0$ is irrational.  In particular, the set of obstructed $a \in [1, a_0]$ has full measure.  
\end{cor}

In contrast, as we show in Corollary~\ref{cor:accum1} every sufficiently large $a$ is unobstructed; in other words, eventually the only embedding obstruction is the volume constraint.
 We also note that Corollary~\ref{cor:new} is in some sense optimal: there certainly do sometimes exist unobstructed $a < a_0$ that are rational and, in addition, $a_0$ itself can sometimes be both irrational and unobstructed.  For example, when $X_{\Omega}$ is a ball, it follows from \cite{ball} that $a_0 = \tau^4$ is unobstructed and the ratios $\frac{g^2_{n+1}}{g^2_n} < a_0$ of squares of odd-index Fibonacci numbers are unobstructed. Notice also that Corollary~\ref{cor:new} implies that if there are infinitely many unobstructed $a< a_0$ then there has to be a staircase.

\subsection{Curvy points, complexity, and more staircases}

In view of Theorem~\ref{thm:curvy}, one might further speculate about what is implied by the existence of a curvy point. 
It is not hard to see that a curvy point forces an infinite weight expansion, i.e. the domain is not of  finite type.  One could speculate that this is in fact the only relevance of curvy points to the staircase question; in other words, one could ask:

\begin{question}
\label{que:infstar}
Is there any $\Om$ with an infinite weight expansion that has an infinite staircase?
\end{question}

In fact, since all previous known examples of infinite staircases occurred for domains with weight expansions with no more than $4$
entries,  one might also conjecture that the size of the weight expansion is in fact quite small whenever there is an infinite staircase.

Concerning this latter conjecture, we can indeed give counter examples ruling this out.   
To get a more precise statement, it is helpful to define the following, see \eqref{eq:cutdefn}: define 
the {\bf cut-length of $\Om$} to be the minimum, over all integral affine transformations $A$, of the
number  of cuts $b_{A,j}$ needed to express $A(\Om)$ in the form
$ \Om(b_A, (b_{A,j})_{j\ge 1})$. This is finite exactly when  $\Om$ is rational (i.e. has rational normals) and is a measure of 
 the complexity of $\Om$.

\begin{thm}\label{thm:inftystair}
There is a sequence of rational domains $\Om_n, n\ge 1,$ of increasing cut-length that do support staircases.
\end{thm}

The proof is given in \S\ref{sec:stair}.  
The regions $\Om_n$ are rational with only seven sides. However, the normals to the sides get increasingly complicated as $n$ increases.  We suspect that one could find many more examples, in which $\Om$ could have an arbitrarily large number of sides; however that is not our emphasis here, and even in our relatively simple examples the constructions and calculations, which are based on those in \cite{MM,MMW}, are quite complicated.

 As for Question~\ref{que:infstar}, the answer remains unknown.  We do show that if it is the case that a region with an infinite staircase must have finite weight length, then this must be for a subtle reason.  Namely, 
 in 
 Proposition~\ref{prop:obs} we show that irrational ellipsoids $E(1,\al)$, which by 
 the arguments here (or, in special cases, the arguments in \cite{salinger})
 are known not to have staircases, 
 do support 
 \lq\lq ghost stairs\rq\rq;  that is, there are  infinitely many obstructive classes  that have no effect on the capacity function because they are \lq\lq overshadowed\rq\rq\, by another larger obstructive class. 
 This shows for example that the proof of Theorem~\ref{thm:curvy}, which goes by showing that in this case there can only be finitely many obstructive classes, does not extend to the general  case.

\begin{rmk}\rm To establish the above results, we use 
two rather different general approaches.  We either argue geometrically 
analyzing the particular curves that obstruct embeddings, or argue using properties of the ECH capacities.  These methods seem to have different advantages and our theorems hopefully illustrate this.  In particular, the only proof we know of Theorem~\ref{thm:curvy}, Theorem~\ref{thm:acc} and Theorem~\ref{thm:class} takes the first approach. On the other hand, Theorem~\ref{thm:convexper} and Theorem~\ref{thm:packing} are proved using ECH or elementary ECH capacities,  in particular their subleading asymptotics.  It would be interesting to further explore the relationship between these two methods. For example, one might investigate whether the ghost obstructions seen in  Proposition~\ref{prop:obs} are also given by appropriate ECH capacities.
\end{rmk}

\subsection{Further questions}

We conclude with several other open questions that our work raises but that we do not address here.

Let $D$ be a bounded domain in $\mathbb{R}^{2n},$ let $X$ be a finite volume symplectic $2n$-manifold, and define the $k^{th}$ packing number $p_{k,D}(X)$ to be the proportion of the volume of $X$ that can be filled by $k$ disjoint symplectically embedded copies of an appropriate scaling of $D$.  Our Corollary ~\ref{cor:super} gives many examples in $\R^4$ of pairs $(D,X)$ where $p_{k,D}(X) < 1$ for all $k$.  However, our obstructions give no information about the  
{\em asymptotic packing number}
\[p_D(X) := \liminf_{k \to \infty} p_{k,D}(X).\]
As we explained in Remark~\ref{rmk:example}, there do exist zero perimeter domains $D$ with $p_D(X) = 1$, but we do not know whether this occurs for all zero perimeter domains.  
In fact, 
in the first version of this paper, we asked the following:

\begin{question}
Must $p_D(X) = 1$ for any such pair $(D,X)$?
\end{question}

This question has subsequently been answered, in the negative, in \cite{cghi} via some four-dimensional examples, but the question of asymptotic symplectic packing still remains very much open.  When $D$ is a ball, \cite[Rem. 1.5.G]{McPolt} shows that this question has an affirmative answer.

There are also many open questions about infinite staircases.  For example, we now know by Theorem~\ref{thm:acc} that the accumulation point theorem holds under very general hypotheses, and we would like to study it further. 
\begin{question}
\label{que:accum}
If $X_\Om$ has an infinite staircase, must the accumulation point $a_0$ be irrational?
\end{question}

Heuristically, one expects Question~\ref{que:accum} to have an affirmative answer, since the accumulation point can be thought of as the germ of the infinite staircase and a rational number does not seem to contain enough information.  Moreover, we now have a plethora of infinite staircases, see e.g. \cite{AADT, MPW}, and all of them have irrational accumulation points.

Here are some other questions about infinite staircases.
 In all known examples, the infinite staircases contain infinitely many visible peaks given locally by a straight line through the origin followed a horizontal line.  Such obstructions are 
 determined by ``perfect\rq\rq\, classes, see Remark~\ref{rmk:perfectObs}; is this a general phenomenon?  One would also like to know 
whether  any kind of classification is possible, for example for the class of (generalized) convex domains.  Our Theorem~\ref{thm:curvy} shows that in attempting such a classification, one can essentially restrict to domains with piecewise linear boundary.

In a different direction, one can further study the subleading asymptotics of ECH capacities.  Now that we have a wide class of examples where the limit is not defined, but the liminf contains geometrically interesting information, one can attempt to build on this.  One potentially fruitful direction to study involves the difference between the liminf and the limsup.   For example, let $U$ be a bounded star-shaped domain in $\mathbb{R}^4$, with smooth boundary.  Then it seems possible that the difference between the liminf and the limsup measures information about the dynamics on the boundary, for example one can ask:

\begin{question}
\label{que:ivrii}
Is there a relationship between $\limsup_k e_k(U) - \liminf_k e_k(U)$ and the measure of the set of periodic Reeb orbits on $\partial U$?
\end{question}   

To start, one can speculate that when the measure of the set of periodic Reeb orbits is zero, the $\limsup$ and the $\liminf$ should be equal; this would be an analogue of a celebrated result of Ivrii for the Laplace spectrum \cite{ivrii} and has been proved for various toric domains in \cite[Cor. 1, Thm. 3]{Wormleighton} and \cite[Thm. 1.10]{Ruel}.  The restriction that $D$ is star-shaped can also presumably be relaxed; for example, one could demand only that $D$ is a smooth compact Liouville domain with finite ECH capacities, or one could study the ECH spectral invariants on a three-dimensional contact manifold directly without reference to a filling.  In the case of toric domains, additional speculation about conditions under which the $e_k$ have a well-defined limit appears in \cite{Wormleighton}.  Another interesting question for toric domains is to understand if $\liminf_k e_k(X)$ has any natural interpretation in the concave case; as we explained in Remark~\ref{rmk:concave}, the ECH capacities do not always determine the perimeter of a concave toric domain.

\subsection{Organization}
In Section~\ref{sec:weight}, we define the cutting algorithm used to determine the weight length and introduce various notions of length associated with convex toric domains. In Section~\ref{sec:ECH}, we prove Theorem~\ref{thm:main}, which characterizes when a concave region embeds into a convex region, and give formulas to compute the ECH and elementary ECH capacities of a generalized convex toric domain. In Section~\ref{sec:acc}, we prove the accumulation point theorem (Theorem~\ref{thm:acc}). In Section~\ref{sec:subECH}, we compute the subleading asymptotics of the ECH capacities of convex toric domains in the proof of Theorem~\ref{thm:convexper} and apply this computation to applications about full fillings with the proofs of Theorem~\ref{thm:packing}, Corollary~\ref{cor:new}, and Corollary~\ref{cor:super}. In Section~\ref{sec:nostair}, we prove Theorems~\ref{thm:curvy} and~\ref{thm:class}, showing that certain convex toric domains do not admit infinite staircases, and explore the phenomenon of ghost stairs in irrational ellipsoids in Section~\ref{ss:ghost}. In Section~\ref{sec:stair}, we prove Theorem~\ref{thm:inftystair}, which provides infinitely many new examples of domains with increasing complexity that have infinite staircases. 

\vspace{3 mm}

\NI {\bf Acknowledgements}
We thank Tara Holm, Michael Hutchings, Matthew Salinger, and Morgan Weiler for very helpful discussions. Additionally, we thank Michael Hutchings and Richard Hind for very useful comments on a draft of our work.  We also thank Michael Hutchings and Rohil Prasad for inspiring discussions related to Question~\ref{que:ivrii}, and the referee for very detailed and helpful comments.  

\section{Weight decompositions, symmetries and length measurements}\label{sec:weight}

The first subsection reviews the cutting algorithm, describes some of its subtleties, and in equation \eqref{eq:cutdefn}
defines an associated measure of complexity: the {\bf cut length}. In \S\ref{ss:Crem}, after a  brief discussion of the realization problem,
we discuss the properties of the Cremona transform, 
showing that it preserves  the ECH capacities and defining an associated complexity measure: the {\bf Cremona length}; see Remark~\ref{rmk:Crem}.  Finally in \S\ref{ss:length} we explain two different ways of measuring the length of planar curves, namely affine length and 
length with respect to a toric domain, and in Lemma~\ref{lem:needed} prove a technical result about the latter  measurement that is used in \S\ref{sec:subECH}.

\subsection{The cutting algorithm}\label{ss:cut}

The  cutting algorithm, which is adapted from \cite{CG},
assigns to a (generalized) convex toric domain $\Omega$ a collection $\Om(b; (b_j))$, where $b$ is a positive real number and the $b_j$ form  a nonincreasing sequence of real numbers. To describe it, suppose first that $\Om$  contains a neighborhood of the origin in $\R^2_{\ge 0},$ and define  $\p^+\Om$ to be the closure of $\p \Om\cap \R^2_{>0}$.
  In this case, there is a unique $b>0$ such that the line $x+y=b$ is tangent to $\p^+\Om$, and 
\begin{align}\label{eq:twoOm}
\intt(\Om) = \intt (T(b)) \less (\Om_1 \sqcup \Om_2)
\end{align}
where $T(b)$ is the standard triangle with vertices $(0,0), (b,0), (0,b)$
and the closed regions $\Om_1, \Om_2$ have vertices at $v_1: = (0,b), v_2: = (b,0)$  respectively.
For $i=1,2$ there is a unique (orientation preserving) affine transformation that takes the corner of $\Om_i$ at $v_i$
to the corner of $\R^2_{\ge 0}$ at the origin and takes the two edges emanating from $v_i$ to segments on the axes. The image $\Om_i'$ of $\Om_i$ is then a concave region  with vertices at $(0,0), (0,\ell_1), (\ell_2,0)$, and we decompose it into a union of balls as in \cite{CG}: We begin by a cut of size $a$, where the line $x+y=a$  is tangent to $\p^+ \Om_i'$, and decomposes $\Om'_i$ into three regions, the triangle $T(a)$ and two concave regions $R_1,R_2$ (one or both may be empty), each with a Delzant corner\footnote
{A Delzant (or smooth) corner is one that is affine equivalent to the corner of $\R^2_{\ge 0}$ at $(0,0)$; equivalently, the two primitive integral normals to its edges form a matrix of determinant $\pm 1$.} 
 on the appropriate axis.  
We then move each of these corners to $(0,0)$ by an affine transformation, and repeat the process. 
This gives a sequence of cuts that are best described by a graph as in 
Figure~\ref{fig:1}.   (For further properties of this graph, see Remark~\ref{rmk:cut} below.)
After decomposing both $\Om_1$ and $\Om_2$ in this way, we define
 the sequence
$(b_j)_{j\ge 1}$ to consist of the sizes of all the cuts, listed in nonincreasing order.

In the general case, when $\Om$ does not contain a neighborhood of the origin, we translate $\Om$ in the positive quadrant to a region $\Om'$ whose boundary intersects both the $x$- and the $y$-axis, and then choose $b$ so that the line $x+y=b$ is also tangent to $\p\Om'$.  Then 
$$
 \Om' = T(b) \less  \intt\,(\Om_0\sqcup \Om_1 \sqcup \Om_2)
$$
where  $\Om_1,\Om_2$ are as before and $\Om_0$ is a concave region with Delzant corner at the origin. 
We  cut up  each of these three regions $\Om_0, \Om_1, \Om_2$ as before, and again define $(b_j)$ to be the
union of the sizes of the cuts listed in nonincreasing order.

\begin{figure}\label{fig:1}
\advance\leftskip-3.5cm
\vspace{-1in}
\includegraphics{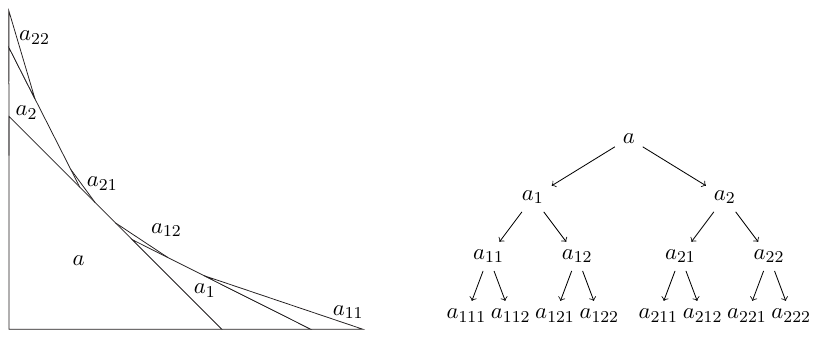}\vspace{-7.5in}
\caption{This illustrates the cutting algorithm for a concave region $\Om$. The first triangle  has size $a$; the second cuts   have sizes $a_1,a_2$ where $a_1+a_2 \le a$, the third set of cuts  have normal vectors $(1,3), (2,3), (3,2), (3,1)$ and  sizes $a_{11}, a_{12}, a_{21}, a_{22}$, 
where $a_{11}+a_{12} \le a_1$,  $a_{21}+a_{22} \le a_2$, and $a_{12} + a_{21} \le a - (a_1+a_2)$.
After three cuts, there are four concave regions $R_{i_1i_2}$ given by the closures of the  components of $\Om'\less \bigl(T(a)\cup T(a_{1})\cup T(a_2)\bigr)$, and so on.}
\end{figure}

\begin{rmk}\label{rmk:cut}\rm {\bf (Comments on the cutting procedure)} 
(i)\ Let $\Om'$ be a concave region with corner at $(0,0)$.
After the $k$th stage of the cutting procedure described above
we have $2^{k}$ concave regions  $R_I$ (some possibly empty), with disjoint interiors, and boundaries on $\p \Om'$,
indexed by $I\in \Ii$, where $\Ii$ is the set of all finite tuples  $I: =  (i_1,\dots, i_{k})$ with $i_j\in \{1,2\}$.
We then cut off  two standard triangles $T_{I1}, T_{I2}$ in each nonempty $R_I$ of sizes $a_{I1}, a_{I2}\ge 0$ where $a_{I1}+ a_{I2}\le a_I$; see Fig.~\ref{fig:1}.  This cut is tangent to $\p \Om'$ at some point $p_I$ of its outer edge.  Note the following
\begin{itemlist}

 \item[{\rm (a)}]
  If any $a_{Ii}=0$, that branch of the tree simply stops.  Geometrically, this means that the 
  point $p_I$ at which the cut meets  $\p \Om'$ is an endpoint of the arc $R_I\cap \p\Om'$.
  Similarly, if $a_{I1} + a_{I2} = a_I$ then the point at which these cuts meet lies in $\p\Om'$ and there are no further cuts centered at this vertex; moreover this vertex is a nonsmooth point of $\p\Om'$.

 \item[{\rm (b)}]  For each rational number $ r> 0$  there is a cut (possibly trivial, i.e. of size $a_I=0$, and hence unseen) whose outer edge has slope $-r$.

 \item[{\rm (c)}]  The boundary of the concave region $\Om'$ has Delzant corners if and only if $\Om$ has finite weight expansion and no two (nontrivial) cuts have the same endpoint.

 \item[{\rm (d)}] If $\Om'$ is {\bf rational} (that is, if its boundary is a finite union of line segments of rational slopes),
 then after a finite number of cuts, $R_I$ is a standard triangle, and the process stops completely. In all other cases, there is at least one infinite chain of cuts.  Note that, if $\p \Om'$ contains a line segment $S$ of rational slope $p/q$, then every cut whose outer edge  has slope $> p/q$ (resp. $< p/q)$ lies entirely to the left (resp. to the right) of $S$. It follows that there is a cut  whose outer edge  contains $S$.
   \end{itemlist}
   \MS
   
   \NI (ii)
 Above we have described an algorithm that cuts up a concave region $\Om'$  into standard triangles. However, given the $(b_j)$ 
  there is no canonical way to make these cuts,
 even if  we specify the intersection of $\Om'$ with the axes. 
 Thus, in general there are   many different concave regions with the same weight sequence $(b_j)$ and the same intersections with the axes.  This phenomenon is even more  apparent when we are given the weight sequence $(b; (b_j))$ of a convex region. 
 In particular we cannot always  construct a convex region by making the first  three cuts at different vertices, and then cyclically  making a cut at each edge.  For example, it is only possible to construct a convex region $\Om$ with weights $(2;1,1,1,1/2)$ if at least two of
  of the first three cuts are placed along the same edge. Indeed if we put three cuts of size $1$ at different vertices of $T(2)$ then no further cuts are possible since  $\Om$  must meet each of the boundary edges of $T(2)$.
\end{rmk}

Although the geometric structure of $X_\Om$ depends significantly on whether or not $\Om$ intersects the axes, the properties that are relevant to the considerations in this paper (such as the weight decomposition $(b; (b_j))$ and the ECH capacities) do not change when $\Om$ is translated off the axes. Moreover, if $\Om\subset \R^2_{>0}$, then $X_\Om$ is symplectomorphic to $X_{A\Om}$ where $A$ is any integral affine transformation such that $A\Om\subset \R^2_{>0}$.  We saw above that each such region $A\Om$ is the translate of a unique region of the form  $\Om(b; (b_j)_{j=1}^n)$, and we set ${\rm length}(A\Om): = n$. 
We  then define the {\bf cut length} of a region $\Om\subset \R^2_{>0}$
as follows:
\begin{align}\label{eq:cutdefn}
{\rm Cut}(\Om): = \min_{A\in \SL(2,\Z)}  {\rm length}(A\Om). 
\end{align}

If $\Om$ contains the origin, then we define its cut length to be that of any of its translates $\Om'$ in $\R^2_{>0}$.
Clearly, this length is finite only if $\Om$ is rational in the sense of Remark~\ref{rmk:cut}~(i) (d). Note that the cut length of a region $\Om$ may bear little relation to the number $n$ of cuts used to present  $\Om$  as $\Om(b;(b_j))$: for example there is no bound on the number of cuts needed to present the image $AT(1)$ of the standard triangle, as $A$ ranges over $\SL(2,\Z)$, while all these regions have cut length $0$.  We discuss a (possibly different) measure of complexity of $\Om$ in Remark~\ref{rmk:Crem}.

For general polytopes $\Om$, it is not
clear 
how to calculate the minimum in \eqref{eq:cutdefn}.  The following result shows that we can estimate this in terms of
the order of the singularities of its vertices. Here, if $\Omega$ is a rational polygon with vertices $\{v_1,\hdots,v_k\}$, we  denote the outward-pointing primitive normal vector to the edge connecting $v_i$ to $v_{i+1}$ by $n_i$ (where the indices are taken $\mod k$),
and then define the {\bf singularity order of the vertex $v_i$} to be $|\det(n_{i-1},n_{i})|.$

Let $\{F_k\}_{k\geq 0}$ denote the Fibonacci sequence where $F_0=F_1=1.$

\begin{lemma} \label{lem:cutBound}
    Let $\Omega$ be a convex rational polygon with cut length $k$. Then, the order of singularity of any vertex of $\Omega$ is at most $8F_k^2.$
\end{lemma}
\begin{proof}
Since the order of any vertex is preserved under integral affine transformations, 
we can assume    without loss of generality that  the cutting procedure of $\Omega$ achieves the cut length.   As described in the cutting algorithm, the weight sequence is computed by considering the weight sequences of the three regions, $\Omega_0,\Omega_1,\Omega_2$. 

    The regions $\Omega_0':=\Omega_0,$ $\Omega_1':=A_1(\Omega_1)$, and $\Omega_2':=A_2(\Omega_2)$ are concave regions that can be translated to contain a neighborhood of the origin where \[A_1=\begin{pmatrix}
        0 & 1 \\ -1 & -1
    \end{pmatrix}^{-1} \quad \text{and} \quad A_2=\begin{pmatrix}
        -1 & -1 \\ 1 & 0 
    \end{pmatrix}^{-1}.\] 
 For each of these concave regions $\Omega_\bullet'$ for $\bullet=0,1,2$, we follow the process outlined in Remark \ref{rmk:cut} organizing the cuts in a tree. Let $(a_j^i,b_j^i)$ for $1 \leq j \leq 2^i$ denote the normal vectors to the cuts on the $i$th layer in the tree describing the cuts of $\Omega_\bullet'$. Define $C_i$ to be the maximum element of $\{|a_j^i|,|b_j^i|\}_{1 \leq j \leq 2^i}$. On the $i$th layer of the tree, the normal vector to each cut is the sum of a normal vector on the $(i-1)$st layer with the normal vector on some lower level. Hence, we have that  $C_{i+1} \leq C_i +C_{i-1},$ which implies that $C_i \leq F_i$ as $C_1=C_0=1.$ By assumption, the cut length of $\Omega$ is $k$, so each $\Omega_\bullet'$ can have at most $k$ layers of the tree. We can conclude that the normal vectors to the cuts of $\Omega_\bullet'$ have entries at most $F_k.$

    By the definition of $A_1,A_2,$ the absolute value of the entries of the normal vectors to the cuts in $\Omega$ are at most $2F_k.$
    Hence   the order of any singularity in $\Omega$ is at most $8F_k^2$ as claimed. 
\end{proof}

The following corollary is an immediate consequence of Lemma~\ref{lem:cutBound}.
\begin{cor} \label{cor:cutLength}
    Let $\{\Omega_n\}_{n\geq 0}$ be a sequence of rational convex polygons, and define $o_n$ to be the maximum of
    the  singularity orders of the vertices in $\Omega_n.$ If the sequence $\{o_n\}_{n \geq 0}$ is unbounded, then the sequence of cut lengths $\{{\rm Cut}(\Om_n)\}$ is unbounded.  
\end{cor}

\subsection{The realization problem and the Cremona action}\label{ss:Crem}

We saw above that each convex domain $\Om$ has the form $\Om(b; (b_j))$, where the
 parameters $(b; (b_j))$ are uniquely determined by $\Om$.  However, the assignation $\Om \to \Om(b; (b_j))$ is neither injective nor surjective. 
 
 To see that
 different domains $\Om$ may give rise to the same  tuple $(b; (b_j))$, consider the case  $(b; (b_j)) = (3;1,1)$.
Then $\Om$ might either be the triangle $T(3)$ with two corners of size $1$ cut off, or it might be $T(3)$ with the corner at $(0,3)$ removed by cutting along the line $y = 1+x$.  (A related  point is made in Remark~\ref{rmk:cut}~(ii).)

At the same time, the question as to which tuples $(b; (b_j))$ do define convex domains also has subtleties.  
The most obvious necessary conditions are:
\begin{align}\label{eq:cutcondit}
\sum b_j^2 < b^2, \quad \sum b_j < 3b, \quad b_1+b_2\le b,
\end{align}
where the last condition is needed in order to fit in the first two triangles. 
However these conditions are not sufficient.  For one, we also need the union of balls $\sqcup_j B(b_j)$ to embed symplectically in $X_{T(b)}$ (or equivalently into $\C P^2(b)$) in such way that their interiors $\intt B(b_j)$ are disjoint. The obstructions to such an embedding are  given by the set of exceptional divisors $E$ in blowups of $\C P^2$, and are made fully explicit in  Karshon--Kessler~\cite{KK}.
A more subtle point is that  even if the balls $\sqcup_j \intt B(b_j)$ do embed in  $B(b) = X_{T(b)}$ there is no guarantee that they can be embedded via the cutting procedure described  in Remark~\ref{rmk:cut}.  For example, since the ball $B(b)$ may be fully filled by four balls of size $b/2$, one might wonder if there is a toric domain corresponding to the tuple $(b; (b/2 - \eps)^{\times 4})$. However, unless $\eps\ge b/6$ so that one can put three of the four cuts along one edge,  it is straightforward to check that no such domain exists.

The {\bf Cremona group} acts on tuples of the form $(b;(b_j))$ (where the $(b_j)$ are not necessarily decreasing or positive) by composing permutations of the $b_j$ with 
 the following  
 transformation of order $2$:
\begin{align}\label{eq:Cremon}
\Cc:  (b;b_1, \ldots) \to (b+d;b_1+d,b_2+d,b_3+d,b_4, \ldots), \quad d = b - b_1 - b_2 - b_3.
 \end{align}

\begin{definition} \label{def:Cremona}  We say that the tuple $(b; (b_j))$  with all $b_j>0$ is {\bf ordered} if the 
$(b_j)$ are nonincreasing. If in addition 
$d\le 0$, we define the {\bf Cremona move} ${\rm Cr}$ to be given by the composite of $\Cc$ with the permutation that restores the order; while if $d>0$, we define   
${\rm Cr}$  to be the identity.
An (ordered) tuple is said to be {\bf reduced} if $d\ge 0$, that is if 
${\rm Cr}(b; (b_j)) = (b; (b_j))$.
\end{definition}

If a realization
 $\Om(b; (b_j))$ of $(b; (b_j))$ puts the first three cuts at different corners of $T(b)$, then it is easy to check that there is an affine transformation that 
 takes $\Om\bigl({\rm Cr}(b; b_1,b_2,b_3)\bigr)$ onto  $\Om(b; b_1,b_2,b_3)$; see Figure~\ref{fig:3} for an illustration. However, it is not clear what happens when two of the first three cuts are put along the same edge. (If the first three cuts are all along the first edge then $d\ge 0$ and the tuple is reduced.)
Notice also that, because a given tuple might have several,  or no, realizations as a toric domain, we cannot   
 in general  interpret the Cremona move as an action on toric domains.  As illustration, consider the following example.

\begin{example} (i)
Consider the tuple $(5;2,2,2,2,2)$.  Since $3\cdot 2>5$, this cannot arise from a convex toric domain.  On the other hand, ${\rm Cr}(5;2,2,2,2,2) = (4;2,2,1,1,1)$ can be realized, as can ${\rm Cr}(4;2,2,1,1,1) = (3;1,1,1,1)$.   
\MS

\NI (ii) 
In contrast, Karshon--Kessler show in \cite{KK} that a finite union of closed balls $\sqcup_{j=1}^k B(b_j)$ embeds symplectically into $\C P^2(b)$ exactly if the tuple
$(b;(b_j)_{j=1}^k)$ is Cremona equivalent through positive tuples to a reduced tuple, i.e. one in which $b\ge b_1+b_2+b_3$. 
\MS

\NI
(iii) The two situations considered here are somewhat different: the triangles we cut out of $T(b)$  intersect along boundary segments (though they have disjoint interiors), while 
 the closed balls $B_j$ in $\C P^2$ are always disjoint. Thus in the Cremona reduction process for toric domains we can ignore zeroes and the number of cuts can change,
 while that does not happen in the process considered in \cite{KK}.
\end{example}

We now show that even though the Cremona action is not fully geometric, it does preserve ECH capacities in the following sense.   Consider tuples of the form $(a_0; a_1, \ldots, a_n)$; 
the prototypical example of such a tuple is the negative weight sequence $(b; (b_j))$  of a convex toric domain.  To such a tuple associate its {\em ECH capacities} 
\[ c_{ECH}(B(a_0)) - c_{ECH}(B(a_1)) - \ldots - c_{ECH}(B(a_n))\] 
where the sequence subtraction operation refers to \cite{Hutchingsblog}.   More precisely, 
we define
\begin{align} &c_k(a_0; a_1, \ldots, a_n) =\\ \notag 
&\qquad\quad \min \big\{ d a_0 - d_1 a_1 - \ldots- d_n a_n\ \big|\  d^2 + 3d - d_1^2 - d_1 - \ldots - d_n^2 - d_n \ge 2k \big\},
\end{align}
where $d$ and the $d_i$ are nonnegative integers.  (This generalizes the formula for 
$ c_{k}(B(a_0))$ given in Lemma~\ref{lem:Perball}.)

\vspace{2 mm}
\begin{lemma}\label{lem:CremECH}
Assume that $(a_0; a_1, \ldots, a_n)$ are positive real numbers satisfying the {\bf admissibility conditions}
\[ a_0^2 > a_1^2 + \ldots + a_n^2, \quad \quad c_{ECH}( B(a_1) \cup \ldots \cup B(a_n) ) \le c_{ECH}( B(a_0) ).\]
Then the ECH capacities of $(a_0; a_1, \ldots, a_n)$ and 
${\rm Cr}(a_0; a_1, \ldots, a_n)$ are the same.
\end{lemma}

We note that the admissibility conditions simply assert that there is no obstruction from either volume or ECH to embedding the balls $B(a_i)$ into the ball $B(a_0)$; this holds automatically when the $(a_0;a_1,\ldots,a_n)$ arise from a convex toric domain.

\begin{cor}  If two convex domains have negative weight expansions in the same Cremona orbit, then they have identical ECH capacities, and therefore by Theorem~\ref{thm:main}, identical ellipsoid embedding functions. 
\end{cor}

\begin{rmk}\label{rmk:Crem} \rm  Lemma~\ref{lem:CremECH} suggests that we  measure the complexity of a domain $\Om$ not by its cut length as described in \eqref{eq:cutdefn}  but by its {\bf Cremona length}
$cr(\Omega)$ defined as follows. Let us write  $\Om\sim \Om'$  if  $\Om(b;b_i)$ and $\Om'(b';b'_j)$ are in the same Cremona orbit, i.e. one can get from one of these sequences to the other by a series of Cremona transforms and reordering. 
Then we define
\begin{align}\label{eq:Cremleng}
 cr(\Omega) := \min_{\Omega' \sim \Omega} \lbrace | \Omega'(b; b_1, \ldots, b_m)|  \rbrace,
 \end{align}
where $| \Omega'(b; b_1, \ldots, b_m)|: = m$.
It is not clear whether this measure   agrees with the cut length defined in \eqref{eq:cutdefn}.
\end{rmk}

\begin{proof}[Proof of Lemma~\ref{lem:CremECH}]  It is convenient to work with the Cremona transformation $\Cc$  of \eqref{eq:Cremon}
rather than the Cremona move, since the latter involves a permutation.  Thus, we will show that the ECH capacities of $(a_0;a_1,\dots,a _n)$ agree with those of
$(a'_0; a'_1, \ldots, a'_n): = 
\Cc (a_0; a_1, \ldots, a_n)$.
To prove the lemma, it will also be convenient to  assume 
\begin{align}\label{eq:convenient} a_0 > a_1 + a_2.
\end{align}
This is permissible by continuity, since the inequality  $a_0 \ge a_1 + a_2$ follows from the fact that
$ c_1^{ECH}( B(a_1) \cup  B(a_2) ) \le c_1^{ECH}( B(a_0) )$. 

We will call any expression of the form $d a_0 - d_1 a_1 - \ldots- d_n a_n$, for the $d_i$ integers (not necessarily nonnegative), a {\bf pre-ECH capacity}.  Further, we will call a nonnegative tuple $(d; d_1, \ldots, d_n)$ satisfying $$
d^2 + 3d - d_1^2 - d_1 - \ldots - d_n^2 - d_n \ge 2k
$$
  {\bf ECH admissible} for $k$.  Thus, $c_k$ is the minimum of the pre-ECH capacities associated to ECH admissible tuples for $k$; we call such a tuple realizing $c_k$ a {\bf minimizer}.
\vspace{2 mm}

\NI
{\bf Step 1:}\ {\it  Equality of pre-ECH capacities.}\  First we claim that the pre-ECH capacity associated to $(d; d_1, \ldots, d_n)$ for $(a_0; a_1, \ldots, a_n)$ is the same as the pre-ECH capacity associated to $\Cc(d; d_1,\ldots,d_n) = (d'; d'_1, \ldots, d'_n)$ for $(a'_0; a'_1, \ldots, a'_n)$.  

To see this,  we need to show that
\[ a_0'd' - a'_1d_1' - a'_2d_2' - a'_3d'_3 = a_0d  - a_1 d_1 - a_2 d_2 - a_3 d_3.\]
This holds because of the easily checked fact that   $C^T J C = J$ where $C$ is the $4\times 4$ matrix that implements the Cremona transformation of~\eqref{eq:Cremon} on the tuple $(x_0,x_1,x_2,x_3)$ and $J$ is the matrix ${\rm diag}(-1,1,1,1)$.

For future reference, let us record the inverse Cremona transform $\Cc^{-1}$, implicit in the above equations, recovering the ordinary variables from the primed variables, defined via
\begin{equation}
\label{eqn:inverse}
d =  2d' - d_1' - d_2' - d_3', \quad  d_1 = d' - d_2' - d_3', \quad d_2 = d' - d_1'- d_3', \quad d_3 = d' - d_1' - d_2'.
\end{equation}
Note in particular that $\Cc^{-1} = \Cc$, in other words $\Cc$ has order $2$.

\vspace{2 mm}
\NI
{\bf Step 2:}\ {\it   Admissibility, part 1.}  Next, we note that the quantity $d^2 + 3d - d_1^2 - d_1 - \ldots - d_n^2 - d_n$ is invariant under the Cremona transform.   This follows from the stronger statement that the quantities
\[ d^2 - \sum d_i ^2, \quad \quad 3d - \sum d_i,\]
are invariant under $\Cc$. 
The first claim follows from the identity $C^T J C = J$ in  Step 1, while the second
(which is also  well-known) is  easy to check. 

\vspace{2 mm}
\NI
{\bf Step 3:} {\it  Admissibility, part 2.}
To proceed, we need to consider what happens if $d < d_1 + d_2$: the issue is that in this case, we would have $d'_3 < 0$, which would not be ECH admissible.  

In fact, we show that an admissible ECH minimizer for $k$ with respect to $(a_0; a_1, \ldots, a_n)$ never has this property.  To show this, assume the opposite, and  
consider $\Delta = d - d_1 - d_2 < 0$.  Define $d'' = d + \Delta, d''_1 = d_1 + \Delta, d''_2 = d_2 + \Delta$; otherwise we set $d''_j  = d_j$.  Now we note that
\begin{align}\label{eq:<}
a_0d'' - a_1d_1'' - a_2d''_2 = a_0d - a_1d_1 - a_2d_2 + \Delta(a_0 - a_1 - a_2) < a_0d - a_1d_1 - a_2d_2,
\end{align}
where in the strict inequality we have used the fact that $a_0 > a_1 + a_2$.  On the other hand
\begin{align*} &d''^2 + 3d'' - d''^2_1 - d''_1 - d''^2_2 - d''_2\\
 &\qquad\qquad = d^2 - d^2_1 - d^2_2 - \Delta^2 + 2 \Delta(d - d_1 - d_2) + 3d - d_1 - d_2 + \Delta
\\
&\qquad\qquad  = d^2 + 3d - d_1^2 - d_1 - d_2^2 - d_2 + \Delta^2 + \Delta \\
&\qquad\qquad  \ge d^2 + 3d - d_1^2 - d_1 - d_2^2 - d_2,
\end{align*}
since $\Delta<0$ is an integer.  
In particular, if the $d_i$ were minimizers for $k$,  then the $d''_i$ would be ECH admissible for some $k' \ge k$, but with a strictly smaller pre-ECH capacity, which is not possible.  

Now observe that the condition~\eqref{eq:convenient} also holds for the tuple  $(a'_0; a'_1, \ldots, a'_n)$, since  $a_3 > 0$ by
assumption. 
Therefore the above argument applies equally well to $(a'_0; a'_1, \ldots, a'_n)$.
In particular, we are justified in applying the inverse Cremona transform \eqref{eqn:inverse} to any ECH admissible minimizer with respect to $(a'_0; a'_1, \ldots, a'_n)$, and we will still get something ECH admissible. 

\vspace{2mm}
\NI {\bf Step 4.} \ {\it  Putting it together.}  From the previous steps, any ECH admissible minimizer for $k$, with respect to $(a_0; a_1, \ldots, a_n)$ induces by Cremona transform an ECH admissible minimizer for $k$, with respect to $(a'_0; a'_1, \ldots, a'_n)$
with the same pre-ECH capacity.  Similarly, any ECH admissible minimizer for $k$, with respect to $(a'_0; a'_1, \ldots, a'_n)$ induces by inverse Cremona transform an ECH admissible minimizer for $k$, with respect to $(a_0; a_1, \ldots, a_n)$, with the same pre-ECH capacity.  The Lemma now follows.
\end{proof}

\subsection{Length measurements}\label{ss:length}

We first review some well known facts about affine length, and then discuss how to measure the length of a lattice path with respect to a convex domain $\Om$.

The {\bf affine length} of a line segment $S$ of rational slope is the Euclidean length of its image under any integral affine transformation $A$ such that $A(S)$ is contained in the $x$-axis.
The affine length
$\Aff(C)$ of a curve $C$  is defined to be the sum of the affine lengths of a maximal collection of disjoint line  segments of rational  slope
that are contained in $C$.  
Below we consider only the affine length of curves of fixed concavity, either concave down or concave up. As the following examples show, the general notion is not very well behaved.

\begin{example}\label{ex:afflength}\rm (i). The affine length $\Aff(\p T(a,1))$ of the boundary of the triangle $T(a,1)$ with vertices $(a,0)$ and $(0,1)$ is $1+a$ if $a$ is irrational, and $1 + p/q + 1/q$ if  $a = p/q$ where $\gcd(p,q) = 1$.  It follows easily that the function $x\mapsto \Aff(\p T(x,1))$ is continuous at irrational $x$, but discontinuous at rational $x$.  Thus, for example, if 
we vary $\Om$ among  convex domains with fixed intersection with the axes, then 
the function $\Aff(\p^+\Om)$ is not continuous at $\p^+\Om$
 if $\p^+ \Om$ contains any rational line segment. \MS

\NI (ii) We may approximate the line from $(0,0)$ to $(a,1)$  in the uniform norm by a sequence of steps consisting of line segments of lengths $< \eps$  that are alternately horizontal and vertical. It is easy to check that the affine length of each such approximation is $a+1$, while if $a$ is irrational the line itself has zero affine length.
Thus we cannot expect the affine length to exhibit any good convergence behavior unless we restrict to curves of fixed concavity.
\end{example}

\begin{lemma}\label{lem:afflength} Let $R$ be the concave region in $\R^2_{\ge 0}$ that lies below the graph of a decreasing continuous function $f: [0,\ell_x]\to [0,\ell_y]$ with $f(0) = \ell_y$ and $f(\ell_x) = 0$. Suppose that the upper boundary $C$ of $R$ contains no line segments of rational slope. Let $C_n\subset R$, $n\ge 1$, be a sequence of curves given by the graphs of piecewise linear, decreasing, functions $f_n:[0,\ell_x]\to [0,\ell_y]$ with rational slopes monotone along $C_n$ that converge to $f$ in the uniform norm. Then $\Aff(C_n)\to 0$. \end{lemma} \begin{proof} Because $C$ has fixed concavity, it is a rectifiable curve, that is, it has a well defined Euclidean length, $\ell(C)$, which is the limit of the Euclidean lengths $\ell(C_n) $ of the curves $C_n$. Thus we may suppose that $\ell(C_n) \le \ell(C)+1$. Now $C_n$ consists of a finite number of line segments $(C_{n,i})_{i\in I_{n}}$ of Euclidean lengths $\ell_{n,i}$ and slopes $-p_{n,i}/q_{n,i}$, where $\gcd(p_{n,i},q_{n,i}) = 1$. By cutting $C$ in two and rotating through a right angle if necessary, we may suppose that $p_{n,i}\le q_{n,i}$ for all pairs $n,i$. For each $k,n$, let $I_{n,k} = \{i\in I_n : q_{n,i}\le k\}$ and consider the sum $L_{n,k}: = \sum_{i\in I_{n,k}} \ell (C_{n,i})$. Then for each $k$ \begin{align*} \Aff(C_n) = \sum_{i\in I_n} \Aff(C_{n,i}) &= \sum_{i\in I_n} \frac{ \ell(C_{n,i})}{\sqrt{p_{n,i}^2+ q_{n,i}^2}}\\ &\le L_{n,k} + \frac{1}{k} \sum_{i\notin I_{n,k}} \ell(C_{n,i}) \le L_{n,k} + \frac 1k \ell (C_n), 
\end{align*} The lengths $\ell(C_n)$ are convergent and hence bounded so that it suffices to show that for each $k$, $\lim_{n\to \infty} L_{n,k} = 0$. But because $p_{n,i}\le q_{n,i}$ and the slopes are monotone, $L_{n,k}$ is the sum of the lengths of at most $k^2$ line segments in $C_n$ of increasing slopes $-i/j$, where $0\le i\le j \le k$. Assume $\lim_{n\to \infty} L_{n,k} \ne 0$. Then there is a subsequence $\{n_r\}$ and an $\eps>0$ such that
$L_{n_r,k}>\eps$ for all $r$. 
After passing to a further subsequence, there is a fixed rational slope
$-i/j$ such that $C_{n_r}$ contains a segment of slope $-i/j$ whose
length is bounded below by a positive constant. These segments converge to a nonzero line segment in $C$ of rational slope $-i/j$, contrary to the hypothesis.
\end{proof}

\begin{lemma}\label{lem:aff} Let $\Om'\subset \R^2_{\ge 0}$ be a concave region with vertices at $(0,0)$,  $(x_\infty,0)$,  and $(0,y_\infty)$, and upper boundary given by a 
continuous curve $C = \p^+ \Om'$ from 
$(0,y_\infty)$ to $(x_\infty,0)$.   Let $a$ and $ (a_I)_{I\in \Ii} $ be the sizes of the triangles in the decomposition
in Remark~\ref{rmk:cut}. Then 
\begin{align}\label{eq:aff}
\Aff(\p^+\Om') = x_\infty +y_\infty - a - \sum_{I\in \Ii} a_I.
\end{align}
\end{lemma}

\begin{proof} See~\cite[Lem.3.6]{Ruel}.
\end{proof}

\begin{cor}\label{cor:aff1}  The convex region $\Om(b; (b_j))$ has volume $\Vol(\Om) = b^2 - \sum_j b_j^2$, and the affine length $\Per(\Om)$ of its boundary  is $3b - \sum_j b_j$.
\end{cor}
\begin{proof} The first claim is immediate.  
The second is a straightforward consequence of Lemma~\ref{lem:aff}, the fact that affine length is invariant  under integral affine  transformations, and the decomposition  of $\Om$ into $T(b) \less (\Om_0\sqcup\Om_1\sqcup\Om_2)$ that is discussed in \S\ref{ss:cut}.
\end{proof}

\begin{cor}\label{cor:aff2}   With $\Om=\Om(b; (b_j))$ as in Corollary~\ref{cor:aff1}, the affine length $\Per(\Om): = 3b-\sum b_j$ is continuous  at $\Om_0$ w.r.t. the Hausdorff distance on the pair $(\Om, \Om\cap \{xy=0\})$ 
 if and only if $\p^+(\Om_0)$ has no rational segments.  
 \end{cor}
 \begin{proof} This holds by adapting the arguments  in Example~\ref{ex:afflength}(i) and Lemma~\ref{lem:afflength}.  Notice that for regions that do not contain a neighborhood of the origin, but do intersect the axes in intervals of length $>0$, we also control the lengths of these intervals; in other words nearby regions have intersections of approximately equal length.   Further details are left to the reader.
 \end{proof}

 We next discuss a property of the quantity $\ell_\Omega(\Lambda)$ that is a crucial ingredient in our arguments in \S\ref{ss:subtract}
 about ECH capacities.
 Here $\ell_\Omega(\Lambda)$ denotes
 the length of an (oriented) lattice path $\La$ 
 with respect to the convex region $\Om: = \Om(b; (b_j))\subset \R^2_{\ge 0}$ and is defined as follows, see  \cite[App]{CG}.\footnote
 {Readers of \cite[App]{CG} should be aware that in that reference convex regions always contain a neighborhood of the origin.}
Orient  $\Omega$ counterclockwise about a point in its interior.
Then the  length $\ell_\Omega(\Lambda)$ of an oriented  lattice path
 is the sum
 \begin{align}\label{eq:Omlength} 
\ell_\Omega(\Lambda) = \sum_{e\in \La} \ell_\Omega(e),
\end{align}
where  the ${\Omega}$ length $\ell_\Omega(e)$ of any oriented edge $e\in \La$ is $e \times p_e$, where $p_e$ is a point on $\p \Omega$ with a tangent in the same direction as $e$.\footnote
 {
 By slight abuse of language, we say that the oriented line $L$ through $p_e\in \p\Om$ is tangent to $\p\Om$ if $\Om$ lies entirely in the left  half plane with boundary $L$. Since $\Om$ is convex, every point in $\p \Om$ has at least one tangent.}
We also denote $\ell_\Om$ by $\ell_X$, where  $X = X_\Om$.
 
  As above, we denote the closures of the components of $T(b) \less \Om$ by 
$\Om_0, \Om_1, \Om_2$, where $\Om_0$ is the (possibly empty) region containing the origin,  $\Om_1$ meets the $y$-axis, and $\Om_2$ meets the $x$-axis. 
Recall also that  an oriented lattice path is said to be {\bf concave} if it is the upper boundary of a concave region of $\R^2_{\ge 0}$, with initial point on the $y$-axis, and final point on the $x$-axis.

Let $P$ be a (generalized) 
convex lattice polytope, that we assume   translated  so that $P = \Om(b';(b'_j)) \subset T(b')$ has at least one vertex on the $y$ axis, one on the $x$ axis, and one on the slant edge of $T(b')$.  Denote its boundary by $\La: = \p P$.
 As in \S\ref{ss:cut},  $T(b') \less P$ is the union of three (possibly empty)  toric regions $Q_0, Q_1,Q_2$ that are affine equivalent to the concave regions $Q_0'=Q_0, Q_1', Q_2'$.
Correspondingly, $\La$ is the union of some line segments on the boundary of $T(b')$ together with three lattice paths $\La_i\subset Q_i, i=0,1,2$ oriented as the boundary of $P$. Define  $\La'_3$ to be the slant edge of $T(b')$,  so that 
$ \ell_{T(b)}(\Lambda'_3)=b b'$.
Further,  define
$\La_i'\subset Q_i'$, where $ i=0,1,2$, to be the concave lattice path affine equivalent to $\La_i$.  (Thus $\La_0' = \La_0$.) 
The next  result exploits the fact that the point $p_e$ corresponding to an edge $e$  in $\La_i$ lies in $\p \Om_i$.

\begin{lemma}\label{lem:needed} Let $\La$ be the boundary of a convex lattice polytope $P = \Om(b', (b'_j))$ in $\R^2_{\ge 0}$, and  define 
the concave lattice paths $\La_i', i=0,\dots,3$
 as above. 
Then, for any convex region $\Om = \Om(b; (b_j))$,
\begin{equation}
\label{eqn:needed2}
\ell_\Omega(\Lambda) = \ell_{T(b)}(\Lambda'_3) - \sum^2_{i=0}\ell_{\Omega'_i}(\Lambda'_i), 
\end{equation}
where  the $\Om_i'$ are the concave regions defined above.
\end{lemma}
\begin{proof}
 Since  the  upper  boundary of $\Om_0= \Om_0'$ disjoint from the axes is contained in that  of $\Om$ but with   opposite orientations, we have
$$
\ell_{\Omega_0'}(\Lambda_0) = -\ell_{\Omega}(\Lambda_0).
$$

We next claim that 
\begin{equation}
\label{eqn:turn0}
\ell_{\Omega}(\Lambda_1) =  b_1 - \ell_{\Omega'_1}(\Lambda'_1),
\end{equation}
where $b_1$ is the $T(b)$-length of the 
slant edge of $B(b')$ lying strictly above $\La_3$.
To see this, note first that, if $p$ is the point on $\p\Om$ whose tangent vector is parallel to an edge $e$ in $\La_1$, then  $A(p-(0,b))$ is a point on $\p \Om_1'$ parallel to $Ae$, where $A = \begin{pmatrix}-1 & 1\\ 0&1\end{pmatrix}$.
Hence, $\ell_{\Omega'_1}(Ae) = - Ae \times A(p - (0,b))$, with the negative sign due to the reversing of orientation.
But
\[- Ae \times A(p - (0,b)) = - e \times (p - (0,b) ) = 
 e \times (0,b) - \ell_{\Omega}(e).\]
The claim in \eqref{eqn:turn0} now follows by summing over all the edges $e$ in $\Lambda_1$, 
because the quantity $e \times (0,b) $ depends only on the horizontal displacement of $e$.
The analogous argument also holds for $\Lambda'_2$.  The claimed \eqref{eqn:needed2} now follows, since the edges in $\Lambda$ along the axes contribute nothing to either side of the equation, so that $\ell_\Omega(\Lambda)$ is the sum of contributions from $\Omega_1,\Omega_2,$ and $\Omega_3.$
\end{proof}

\section{ECH capacities and ball embeddings}\label{sec:ECH}

In this section we first prove Theorem~\ref{thm:main} which characterizes when a concave region embeds into a (generalized) convex region both in terms of ball embeddings and in terms of ECH capacities.
We then establish some useful results about the ECH capacities $c_k(X_\Om)$ that are known when $X$ has  finitely many negative weights $b_j$ and contains a neighborhood of the origin. In particular, Lemma~\ref{lem:ECHk} establishes the following useful formula.

\begin{align}\label{eq:cECHk}
c_k(X) = \min_{k = \ell - k_1 - \ldots - k_m} c_\ell B(b) - c_{k_1} B(b_1) - \ldots - c_{k_m} B(b_m).
\end{align}

Finally, we show in \S\ref{ss:elem} that the ECH capacities for $X_\Om$ agree with the elementary capacities defined by Hutchings in \cite{Helem}.

\subsection{Preliminary results}\label{ss:prelim}

We first review some basic properties of ECH capacities. For $X$ either a convex or concave toric domain, the ECH capacities are a sequence of real numbers
\[ 0=c_0(X) \leq c_1(X) \leq \hdots < \infty.\] By work of Hutchings in \cite{Hutchq}, for convex or concave toric domains\footnote{See also \cite{ECH_volume} for more general settings.}, the ECH capacities satisfy the following properties:
\MS

\NI
$\bullet$ {\bf (Monotonicity)} If $X_1 \se X_2$, then $c_k(X_1) \leq c_k(X_2)$ for all $k \geq 0.$
\smallskip

\NI
$\bullet$    {\bf (Scaling)} If $\la$ is a nonzero real number, then
    \[ c_k(\la X)=|\la| c_k(X).\]

\NI $\bullet$ 
 {\bf (Disjoint Union)} The following equality holds:
    \[ c_k(\bigsqcup_{i=1}^n X_i)=\max_{k_1+\hdots+k_n=k} \sum_{i=1}^n c_{k_i}(X_i).\]

\NI $\bullet$  {\bf (Volume)} The following equality holds:  
    \[ \lim_{k \to \infty} \frac{c_k(X)^2}{2k}=\Vol(X).\]

\MS

Other important ingredients in the proof of Theorem~\ref{thm:main} are 
the  following  results about the uniqueness of symplectic forms on convex domains and the
 connectedness of embedding spaces. 

\begin{prop} 
\label{prop:uniq} For any convex region $\Om$, the group of compactly supported diffeomorphisms of $\intt X_\Om$ acts transitively on
the space of symplectic forms on $X_\Om$ that are standard near the boundary. Moreover, 
the group of compactly supported symplectomorphisms of $(\intt X_\Om, \om_{std})$
 is contractible. 
\end{prop}

\begin{prop}
\label{prop:conn}
Let $\Omega_2$ be a convex region and let $\Omega_1$ be a concave region.  Then any two symplectic embeddings of $X_{\Omega_1}$ into $\intt(X_{\Omega_2})$ are isotopic via an ambient compactly supported isotopy of $\intt(X_{\Omega_2})$.
\end{prop}

The above claims  differ from the statements proved in \cite{CG} in two ways. Firstly we allow the negative weight decomposition of $\Om_2$ to be infinite --- note that the weight decomposition  of $\Om_1$ was always allowed to be infinite --- and secondly we enlarge the class of convex regions considered to include regions that do not contain a neighborhood of the origin.  
The first generalization imposes no real difficulty, basically because it follows from Proposition~\ref{prop:conn} that we can always slightly shrink $\Om_2$ so that it has a finite weight expansion.   However, the construction in \cite{CG} that proved Theorem~\ref{thm:main} when $\Om_2$ contains a neighborhood of the origin made crucial
use of Gromov's uniqueness result for symplectic forms on star-shaped domains that are standard at infinity. 
We replace this by the uniqueness result in Proposition~\ref{prop:uniq}, that follows from
a strengthened form of the uniqueness result in \cite[Thm.9.4.7]{JHOL} for symplectic forms on $S^2\times S^2$.

We start by proving Proposition~\ref{prop:uniq}.  
Our argument relies on the following uniqueness result
for symplectic forms $\om$ on $S^2 \times S^2$ that are standard near three or four spheres.

Let $S^2$ have the standard cylindrical coordinates $(z,\theta)\in [-1,1]\times S^1$, where the circles $\{\pm 1\}\times S^1$ are collapsed to points, and define $p_\pm = \lbrace z = \pm 1 \rbrace$.
Let $S_1, \ldots, S_4$ denote the four distinguished spheres 
\begin{align}\label{eq:S1234}
S_1 = \lbrace p_- \rbrace \times S^2, \quad S_2 = \lbrace p_+ \rbrace \times S^2, \quad S_3 = S^2 \times \lbrace p_- \rbrace, \quad  S_4 = S^2 \times \lbrace p_+ \rbrace.
\end{align} 
Call $dz \wedge d \theta$ the standard symplectic form (with the obvious extension over the poles), write $\om_{std}: = \sum^2_{i=1} dz_i \wedge d\theta_i$, and define the nonstandard set to be the closure of the set on which  $\om-\om_{std} \ne 0$.

\begin{lemma} 
\label{lem:s2s2}
Let $\om$ be a symplectic form on $S^2 \times S^2$ such that the nonstandard set for $\omega$ is compactly supported in $(S^2 \times S^2) \setminus \Ss$, where $\Ss$ is either $S_1  \cup S_2\cup S_3$ or $S_1\cup\ldots \cup S_4$, where the $S_i$ are as above.  Then, there is a
diffeomorphism
compactly supported in $(S^2 \times S^2) \setminus \Ss,$ between  $(S^2 \times S^2, \omega)$ and $(S^2 \times S^2, \om_{std} )$.

Moreover, the group of  compactly supported symplectomorphisms of  $S^2\times S^2\less \Ss$ is  contractible.
\end{lemma}

The proof is an elaboration of that in \cite[Ch.9.4]{JHOL}, and is deferred until the end of this section.
\MS

In the following it is convenient to consider convex regions that are good in the following sense.

\begin{definition}\label{def:good}  A convex region $\Om$ is said to be {\bf good} if it lies off the axes and  is rational with
 Delzant corners.\footnote{
i.e. the matrices formed by the corresponding conormals, which are integral, have determinant $1$.}
 \end{definition}
 
  In this case 
the leaves of the characteristic foliation over the boundary edges of $\Om$ are circles, and by collapsing these we obtain a toric manifold $\ov{X}_{\Om}$ with moment polytope $\Om$.

\MS

\begin{proof}[Proof of Proposition~\ref{prop:uniq}]
We must show that for  any convex region $\Om$, the group of compactly supported diffeomorphisms of $\intt X_\Om$ acts transitively on
the space of symplectic forms on $X_\Om$ that are standard near the boundary. Moreover, 
the group of compactly supported symplectomorphisms of $(\intt\, X_\Om, \om_{std})$
 is contractible. 
If $\Om$ contains a neighborhood of the origin then $X_\Om$ is star-shaped and the result is well known.
Therefore we concentrate on the proof for convex domains $\Om$ which do not contain a neighborhood of the origin. 
Our argument relies on the well known fact that the group of compactly supported symplectomorphisms of $(\intt\, X_\Om, \om_{std})$
is an ANR and hence homotopy equivalent to a CW complex. (The proof is sketched in \cite[Rmk.9.5.5]{JHOL}.)  It therefore suffices to
show that the group of symplectomorphisms with support in a fixed compact subset of $\Om$ is contractible. 
Thus we may suppose
that $\Om$ is rational, i.e. has finite negative weight decomposition, since for any compact subset $K\subset \intt\,\Om$,
 there is a rational region $\Om'\subset \intt\,\Om$ such that $K \subset \intt\,\Om'$.
\MS

\NI {\bf Step 1.} {\it The case when $\Om$ is good.}

 Consider the coordinates $z_k = \sqrt{ \mu_k / \pi } e^{i \theta_k}, \mu_k\ge 0, k=1,2,$ on $\mathbb{C}^2$.  Then, in these coordinates
\[ 
\omega_{std} = \frac{1}{2 \pi} \sum^2_{k=1} d \mu_k \wedge d \theta_k.
\]
We fix a point $(\ell_1,\ell_2)$ in the interior of $\Omega$, and a symplectic form $\om$ on $X_\Om$ that is
 standard near the boundary. By extending it by $\om_{std}$, we may consider it to be a symplectic form on $\mathbb{C}^2$.
Define the {\em nonstandard set} for $\omega$ to the closure of the set of points in $\mathbb{C}^2$ for which $\omega$ differs from $\omega_{std}$.   

The 
 radial retraction of $\C^2 \less \{\mu_1\mu_2=0\}$ towards $(\ell_1,\ell_2)$ is given by
 the following formula:
 \begin{align}\label{eq:retract}
 f_t: (\mu_1, \mu_2, \theta_1, \theta_2) \to ( (1-t) \ell_1 + t \mu_1, (1-t) \ell_2 + t \mu_2, \theta_1, \theta_2).
 \end{align}
Since  $f_t^*(\omega_{std}) = t  \omega_{std}$,
 the expansion $g_t: = f_t^{-1}$ has the property that for all $t\in (0,1]$ the form
$$
\om_t: = t  g_t^*(\om)
$$
is standard near infinity, with nonstandard set  contained in $\intt X_\Om$. Therefore, there is a family of diffeomorphisms $h_t, 0<t\le 1$ with support in $\intt X_\Om$ such that $h_1 = \rm{id}$ and $h_t^*(\om) = \om_t$ for all $t\in (0,1]$. 

Next observe that
 for sufficiently small $t=\eps>0$, this nonstandard set     is contained in $X_{Sq}$, where $Sq\subset \intt\, \Om$ is a small 
square  with center $(\ell_1,\ell_2)$, and the form  $\om_\eps$ induces a symplectic form on $X_{Sq}$ that is standard near the four spheres that lie over the boundary  $\p(Sq)$.  
Hence, by Lemma~\ref{lem:s2s2}, the form $\om_\eps$ is diffeomorphic to the product form $\om_{std}$ by a diffeomorphism $\phi$ that is the time $1$-map of an isotopy $\phi^t$ with support in $\intt\,  X_{Sq}$. Therefore for small enough $\eps>0$  the diffeomorphism $h_\eps\circ \phi$  pulls $\om$ back to $\om_{std}$.
This proves the first claim in Proposition~\ref{prop:uniq}.  

To see that the
group of compactly supported symplectomorphisms of $(\intt\,  X_\Om, \om_{std})$
 is  contractible, notice first that as in \cite[Rmk.9.5.5]{JHOL} it suffices to show that every compact subset $C$ of this group contracts to a point. Each such set $C$ consists of symplectomorphisms with support in some compact subset $\Om_\eps\subset \intt\, \Om$.   Therefore we may homotope $C$ via the maps $g\mapsto f_t \circ g\circ f_t^{-1}$  to be a compact subset $S$ of the group of symplectomorphisms of $X_{Sq}$ that are the identity near the boundary, and then appeal to Lemma~\ref{lem:s2s2} to find a further homotopy (consisting of symplectomorphisms of $X_\Om$ with support in $\intt X_{Sq}$) that contracts  $S$ to a point.
\MS

\NI
{\bf Step 2.}  {\it The case when $\Om$ is rational and intersects just one of the axes in an interval of positive length.}

In this case, we
may prove that  the form is diffeomorphic to the standard form by choosing the point $(\ell_1, \ell_2)$ to lie on this axis and then
arguing as 
 in Step 1, interpreting  the formula for the retraction as appropriate.  
 Note that now we apply Lemma~\ref{lem:s2s2}  in the case when the form is standard on just three spheres. 
 But this makes no essential difference to either part of the argument.

 \MS

\NI
{\bf Step 3.} {\it The case when $\Om$ is rational and intersects both axes in an interval of positive length.}

In this case, we may consider  $\Om$ as $\Om_2\less \intt\, \Om_1$, where $\Om_2$ contains a neighborhood of the origin and $\Om_1$ is concave.   Then, by \cite{CG},  there is a compactly supported diffeomorphism $\psi$ of $X_{\Om_2}$ such that $\psi^*(\om) = \om_{std}$. If we choose $r>1$ so that  $\om|_{X_{r\Om_1}} = \om_{std}|_{X_{r\Om_1}}$,  the restriction of $\psi$ to $X_{r\Om_1}$ is a symplectic embedding in $X_{\Om_2}$. Therefore by \cite[Prop.1.5]{CG} there is an ambient isotopy $g_t$ of $(X_{\Om_2}, \om_{std})$ such that $g_0 = id$ and $ g_1|_{X_{r\Om_1}} = \psi\circ \io$, where $\io: X_{r\Om_1}\to X_{\Om_2}$ is the inclusion. 
Then $(g_1)^{-1}\circ \psi$ is the desired compactly supported diffeomorphism of $X_\Om$. 

To prove that the group $G={ \rm Symp}_c(\intt X_\Om, \om_{std})$ is  contractible,  consider its action on the space ${\mathcal  Cyl}$ of symplectically embedded cylinders that are isotopic to the unique cylinder\footnote
{One can work with  spheres instead of cylinders by first shrinking $\Om$ to a convex set $\Om'$ with Delzant boundary, and then considering symplectomorphisms of the compact toric manifold $\ov X_{\Om'}$ that are the identity near the components of the boundary divisor that do not lie over the axes.}
 $C_x$ in $X_\Om$ that lies over the $x$-axis by a symplectic isotopy with support in $\intt\,  X_\Om$.  As in \cite[Ch.9.5]{JHOL}, $G$ acts transitively on ${\mathcal Cyl}$, and also the space ${\mathcal Cyl}$ is  contractible since it contains a unique $J$-holomorphic element  for  each $\om_{std}$-compatible $J$ that is standard near $\p X_\Om$. As in the situation considered there, this
implies that every compact subset of the group $G$ deformation retracts into the subgroup of $G$  consisting of elements that are the identity near $C_x$.
But this latter group is contractible by Step 2.   Further details are left to the reader.
\MS

This completes the proof.
 \end{proof}

We now turn to the proof of Proposition~\ref{prop:conn}.
This is proved by essentially the same argument as in \cite{CG}; we just have to clarify the various kinds of possibilities for a generalized  convex toric domain.   Nevertheless, the argument is somewhat tricky and we go over it carefully. 
The main ingredient is a \lq\lq blowup--blowdown\rq\rq \, correspondence, that converts the
problem of finding an ambient isotopy between two  symplectic  embeddings 
$g_0$ and $g_1$  of $X_{\Omega_1}$ into $\intt\, (X_{\Omega_2})$ to
the problem of constructing an isotopy (of symplectic forms) between two cohomologous symplectic forms $\om_0, \om_1$ on 
a fixed symplectic manifold $Y$.  

If $\Om_2$ is good in the sense of Definition~\ref{def:good}, it is the moment polytope of a toric manifold $\ov X_{\Om_2}$, and we denote by $\Cc_2$ the chain of  symplectic $2$-spheres that lie over the boundary $\p \Om_2$  Similarly, if $\Om_1$ is rational with Delzant corners,  the boundary of the  domain $X_{\Om_1}$ can be collapsed along its characteristic foliation to form a chain of symplectic  spheres $\Cc_1$.  In both cases pairs of adjacent spheres in the chains $\Cc_1,\Cc_2$ are symplectically orthogonal. (This holds because it is true in the toric model.)
More generally, given any symplectic embedding $g:X_{\Om_1}\to \intt\, {X}_{\Om_2}$ we can remove the interior of its image and then collapse the boundary appropriately to obtain a chain of symplectic spheres $\Cc_1'$.  
We call the symplectic  manifold $Y: = Y_g$ obtained in this way the blowup of $\ov X_{\Om_2}$ along $g$, and denote its symplectic form by $\om_{g,std}$ and the internal chain of symplectic spheres by $\Cc_{Y}$.   Conversely, given a symplectically  embedded image $\Cc_1': = \io(\Cc_1)$ of $\Cc_1$ in $\ov X_{\Om_2}$ such that the spheres in $\Cc_1'$ are symplectically orthogonal, 
we can blow down $\Cc_1'$ by removing these spheres and inserting a copy of $\Om_1$. 

\begin{lemma}
\label{lem:blowupdown} Suppose that both $\Om_1$ and $\Om_2$ are rational with Delzant corners, and let $(Y,\om_{g,std})$ be the manifold
obtained as above from a symplectic embedding $g: X_{\Om_1} \to \intt X_{\Om_2}$ by collapsing $\p X_{\Om_2}$ to $\Cc_2$  and blowing up $g(X_{\Om_1})$ to $\Cc_Y$.
Then there is a bijection between the following:
\begin{itemize}
\item [{\rm (i)}] Symplectic embeddings $X_{\Omega_1} \to \op{int}(X_{\Omega_2})$, up to ambient, compactly supported  isotopy.
\item [{\rm (ii)}] Equivalence classes of symplectic forms $\omega'$ on $Y$, standard near $\Cc_2$, such that $\omega'$ restricts to $\om_{g,std}$ on $\Cc_Y$, modulo compactly supported diffeomorphisms  of $Y\less \Cc_2$ that are the identity on $\Cc_{Y} $.
\end{itemize}
\end{lemma}   

\begin{proof} Since a very similar result (concerning the embeddings of disjoint balls rather than a concave region) is established in \cite[\S2.1]{McPolt} (see also \cite[Thm.7.1.20]{INTRO}), we only sketch the proof here.
First, consider a symplectic embedding $g_1: X_{\Om_1}\to \op{int}(X_{\Omega_2})$, and slightly extend it to $X_{r_1\Om_1}$ for some $r_1>1$. Similarly extend $g$ to  $X_{r_1\Om_1}$. 
Because the Hamiltonian group acts transitively on the points of $\intt\, (X_{\Om_2})$ 
we may alter  $g_1$ by a Hamiltonian isotopy so that it agrees with
$g$ on 
$X_{r_0 \Omega_1}$ for small enough $r_0>0$.  Let $\psi$ be a  compactly supported diffeomorphism  of 
 $\intt\,  X_{r_1 \Omega_1}$ that restricts to multiplication by $r_0$ on  $X_{\Omega_1}$, and let
 $\psi_Y: (X_{\Om_2}, g(X_{\Om_1})) \to (X_{\Om_2}, g(X_{r_0\Om_1})) $ be the extension by the identity of $g\circ \psi \circ g^{-1}$.
 Then  $\psi_Y^{-1}\circ g_1\circ \psi (X_{\Om_1}) = g(X_{\Om_1})$  so that the pushforward of $\om_{std}$ by $\psi_Y^{-1}\circ g_1\circ \psi$ induces  a symplectic form $\om_{g_1}$ on $Y$ that satisfies the conditions in (ii). 
 It is now straightforward to check that if we vary $g_1$ by a homotopy then $\om_{g_1}$ varies by an isotopy that is constant near
 $\Cc_Y\cup \Cc_2$.  Hence (i) implies (ii).
 \MS
 
 Given a representative symplectic form $\om'$ as in (ii), we can assume our form is standard near $\Cc_{Y}$, and since it is already standard near $\Cc_2$, it blows down to give a symplectic form $\omega$ on $\op{int}(X_{\Omega_2})$, standard near the boundary.  By Proposition~\ref{prop:uniq}, there is a compactly supported diffeomorphism $\phi$ of $\intt\, (X_{\Om_2})$ such that $\phi^*(\om_{std}) = \om$, so $\phi \circ g$ is a symplectic embedding $(X_{\Om_1}, \om_{std}) \to \bigl(\intt\, (X_{\Om_2}), \om_{std}\bigr)$.  This gives a well-defined bijection on equivalence classes since  two choices for $\phi$ differ by composition with a compactly supported symplectomorphism of $\intt\,  X_{\Om_2}$,
 and this group is path-connected by Proposition~\ref{prop:uniq}.   
\end{proof}

\begin{rmk}\rm As in \cite{CG}, this argument applies equally well to the case when $\Om_1$ is disconnected, since the Hamiltonian group acts transitively on $n$-tuples of points in $\intt\, (X_{\Om_2})$. Thus there is 
no  need for the chain of spheres $\Cc_Y$ to be connected. 
\end{rmk}

We are now ready to prove Proposition~\ref{prop:conn}, that states that for convex $\Om_2$ and concave $\Om_1$ any two symplectic embeddings of $X_{\Om_1}$ into $\intt\, (X_{\Om_2})$ are isotopic.

\begin{proof}[Proof of Proposition~\ref{prop:conn}]

By Lemma~\ref{lem:blowupdown}, it suffices to show that any two symplectic forms  on $Y$ that are standard near $\Cc_2$ and on $\Cc_Y$ are diffeomorphic by a compactly supported diffeomorphism of $Y\less \Cc_2$ that is the identity on $\Cc_Y$. These forms blow down to symplectic forms $\om_1, \om_2$ on $X_{\Om_2}$ that are standard near the boundary and on the contractible set $g(X_{\Om_1})$. Thus 
Proposition~\ref{prop:uniq} implies  that there is a compactly supported diffeomorphism of $X_{\Om_2}$ that takes one to the other.  It remains to adjust these forms by an isotopy so that this diffeomorphism can be chosen to be the identity on the contractible set $g(X_{\Om_1})$.
\end{proof}

The following is a standard corollary which will be useful to us:

\begin{cor}
\label{cor:sta}
Let $\Omega_1$ be concave and $\Omega_2$ be convex.  There is a symplectic embedding
\[ \intt\, (X_{\Omega_1}) \to \intt\, (X_{\Omega_2})\]
if and only if for all $0 < t < 1$ there is a symplectic embedding
\[ X_{t \cdot \Omega_1} \to \intt\, (X_{\Omega_2}).\]
\end{cor}
\begin{proof} Suppose given a sequence of embeddings $g_k:X_{t_k \cdot \Omega_1} \to \intt\, (X_{\Omega_2})$ where $t_k\to 1$. 
Proposition~\ref{prop:conn} implies that for each $k$ there is a compactly supported Hamiltonian isotopy $\phi_{k,s}, 0\le s\le 1,$ of $\intt\,  X_{\Om_2}$ such that 
$\phi_{k,1}\circ g_{k+1}|_{X_{t_k \cdot \Omega_1}} = g_k$.
Thus, by replacing $g_{k+1}$ by $\phi_{k,1}\circ g_{k+1}$ we  can arrange that $g_{k+1}$ extends $g_k$ for all $k$, which gives a well defined embedding
of $X_{\Om_1} = \bigcup_k X_{t_k\cdot \Om_1}$ into $\intt\, (X_{\Om_2})$.
\end{proof}

It remains to prove Lemma~\ref{lem:s2s2} that claims that all symplectic forms on $S^2 \times S^2$ that are constant on three or four of the spheres $S_i, 1\le i\le 4,$ are standard. Moreover the group of symplectomorphisms that are the identity near these spheres is  contractible.
\MS

\begin{proof}[Proof of Lemma~\ref{lem:s2s2}]
Let $\Ss$ be the union of the spheres $\bigcup_{i=1}^kS_i$ defined in \eqref{eq:S1234},  where $k=3$ or $4$.
We must first show that if $\om$ is a symplectic  form on $S^2 \times S^2$ with nonstandard set disjoint from $\Ss$
there is a
symplectomorphism,
compactly supported in $(S^2 \times S^2) \setminus \Ss$ between  $(S^2 \times S^2, \omega)$ and $(S^2 \times S^2, \om_{std} )$.
\MS

\NI {\bf Step 1} {\it There is a 
diffeomorphism $\phi$
of $S^2 \times S^2$,

with the following properties: 
\begin{itemize}\item[{\rm (a)}] $\phi ^*(\om)$ is the product form $ \sum^2_{i=1} dz_i \wedge d\theta_i $,

\item[{\rm (b)}]  $\phi$ fixes product neighborhoods of the $S_i$, preserving the normal coordinate\footnote{Here, by a product neighborhood of $S_1$, we mean a neighborhood of the form $K \times S_1$, and the normal co-ordinate refers to the coordinate on $K$; these terms are defined analogously for the other $S_i$.};
\item[{\rm (c)}]   $\phi$
is the identity on
neighborhoods of $S_1$ and $S_3$; 
\item[{\rm (d)}]  $\phi$ is the identity on
neighborhoods of any intersection point of the $S_i$.
\end{itemize}
}

\NI {\it Proof:}  Claims (a), (c) are proved in  \cite[Thm. 9.4.7]{JHOL}; we begin by briefly reviewing the proof.

Choose an $\om$-tame almost complex structure $J$ that is standard in a neighborhood of each of the spheres $S_i$, and generic elsewhere, and
consider the moduli spaces $\mathcal{M}_A$ and $\mathcal{M}_B$ of $J$-holomorphic spheres (modulo parametrization) in classes $A = [S_1], B=[S_3]$.  Note that   these moduli spaces are compact because $J$ is generic away from the four spheres $S_i$. 
Consider the map $\mathcal{M}_A \times \mathcal{M}_B \to S^2 \times S^2$ given by mapping the pair of spheres $(C,C')$ to their unique point of intersection; the proof of \cite[Thm. 9.4.7]{JHOL} shows that this is a diffeomorphism.  Since each $A$-sphere (resp. $B$-sphere)  intersects $S_3$ (resp. $S_1$) in a unique point, we may identify $\mathcal{M}_A \times \mathcal{M}_B$ with $S^2 \times S^2$ and hence obtain a map
 $$
 \Psi: S^2 \times S^2 \to (S^2 \times S^2,\omega)
 $$
 that is the identity on $S_1\cup S_3$ and fixes the four points $(p_\pm, p_\pm)$.  Moreover,  (b) holds because the fact that
  $J$ is a product near  $\bigcup_{i=1}^4S_i$ implies that the spheres $\{p\}\times S^2, S^2\times  \{p\}$ are holomorphic for $p$ sufficiently close to $p_\pm$.  In particular, $\Psi={\rm id}$ near the four points $(p_\pm,p_\pm)$.
\MS

We next arrange that (c) holds by
 altering the map $\Psi$ to a map $\Psi'$ as follows.  We can assume our neighborhood has the form $(S^2 \times U) \cup (U \times S^2)$, where $U$ is a disc.  Outside of this neighborhood, we set $\Psi' = \Psi$.  On $S^2 \times U$, we note that in view of (b) we can write $\Psi(q_1,q_2) = (\eta(q_1,q_2),q_2)$, where $\eta$ is smooth.  Now choose a subdisc $D$ in $U$ containing $p_-$,
 a smooth map $f$, preserving $U$, mapping $D$ to $p_-$, and having nonnegative determinant, and define $\Psi'(q_1,q_2) = (\eta(q_1,f(q_2)),q_2)$.  We define $\Psi'$ on $U \times S^2$ analogously.   Given our choices, this is a well-defined diffeomorphism.  Our choices also imply that $\Psi'$ satisfies $(b) , (c), (d)$.
 Using a similar retraction, we may also  arrange that the following condition holds:
 \MS
 
 (c$^\prime$): {\it There are diffeomorphisms $\psi_2, \psi_4$ of $S^2$ such that $\Psi'(q_1,q_2) = (q_1, \psi_2(q_2)),$ for $(q_1,q_2)$ near $S_2 = \{p_+\}\times S^2$ and
 $\Psi'(q_1,q_2) = (\psi_4(q_1),q_2)$ for $(q_1,q_2)$ near $S_4 = S^2\times \{p_+\}$.} 
 
 \MS
 
 Finally to prove (a) we observe that, because both $(\Psi')^*(\om)$ and the product form are nondegenerate and of the same sign on all the spheres $\{q\}\times S^2, S^2\times \{q\},$ where $q\in S^2,$ the straight line between these two forms consists of symplectic forms. Thus Moser's argument constructs an isotopy $\Phi_t, \Phi_0 = {\rm id.}$ such that $\Phi_1^*(\Psi')^*(\om) = \om_{std}$.   This proves (a). Moreover, properties (b),(c), (d) still hold.
  Indeed, when the symplectic forms agree, the Moser isotopy generated by the straight line is constant, hence (c) and (d) hold.  
  The same argument ensures that  $\phi: = \Psi'\circ \Phi_1$ satisfies (b)
 if we  first arrange that $(\Psi')^*(\om) = \om$ near all four spheres $S_i$.  Since (c$^\prime$) holds, this  may be accomplished by adjusting $\Psi'$ near the sphere $S_2$ (resp. $S_4$) by an isotopy that near $S_2$ (resp. $S_4$) depends only on the coordinate $q_2$ (resp. $q_1$).  \end{proof}
 \MS
 
 \NI {\bf Step 2:}. 
 {\it  With 
 $\phi$ as in Step 1, there is
 a symplectomorphism $\gamma$ of $S^2 \times S^2$, such that $\phi \circ \gamma$ continues to restrict to the identity on
neighborhoods of $S_1$ and $S_3$,
and also 
restricts to the identity on a neighborhood of $S_2:= p_+ \times S^2$.} 
\begin{proof}
We  define $\gamma$ as follows.  We first note that the restriction of $\phi$ to $S_2$ is area-preserving.  Now   
choose a path $\beta_t$ from the identity to 
$\phi^{-1}|_S$, in the space of area-preserving maps.  We know that $\phi$ fixes neighborhoods of $p_-$ and $p_+$, so standard arguments
imply that we can choose this path such that each $\beta_t$ fixes 
a neighborhood of $p_-$ pointwise.  

Choose also a smooth function $\alpha$ from $[0,1]$ to itself that is nondecreasing, equal to $0$ in a neighborhood of $0$ and equal to $1$ in a neighborhood of $1$.

We now define $\gamma$ via the rule
\[ \gamma: S^2 \times S^2, \quad (q_1,q_2) \to (q_1, \beta_{\alpha(z(q_1))}(q_2)),\]
 where by $z(q_1)$ we mean the projection to the $z$-coordinate.
Then, $\gamma$ restricts to the identity on neighborhoods of $S_1: = p_- \times S^2$ and $S_3: = S^2 \times p_-$.
Moreover, $\phi \circ \gamma$ is the identity
in a neighborhood of
$S_2$.   Hence, $\phi \circ \gamma$ has the claimed properties.    
\end{proof}

\NI {\bf Step 3.}  {\it There is a diffeomorphism
 $\phi'$
of $S^2 \times S^2$ that satisfies all the conditions in Step 1 and in addition restricts to the identity on a neighborhood of $S_2$.}

\begin{proof}
First note that the diffeomorphism $\gamma$ defined in step 2 satisfies the following.

\begin{itemize}\item[{\rm (a)}]
 $\gamma^*(\om_{std})$ is standard in a neighborhood of $S_1 \cup S_3 \cup S_2$; 
 
\item[{\rm (b)}] the line $t  \gamma^*(\om_{std}) + (1-t) \omega_{std}$ is symplectic.  
\end{itemize}

Claim a) is immediate from the definition of $\gamma$.  To check (b), we compute that the forms $dz_1$ and $d\theta_1$ are fixed by $\gamma^*$,  while the form $dz_2$ pulls back to the sum of $\beta_z^*dz^2$ and a multiple of $dz_1$, with the analogous statement holding for $d \theta_2$; the claim follows from this.
  
Given these facts, we can compose $\phi \circ \gamma$ with a Moser isotopy to produce a map $\phi'$ that satisfies properties (a) - (d) from Step 1, while in addition 
restricting to the identity in a neighborhood
of $S_2$. \end{proof}

\NI
{\bf Step 4.}  {\it Completion of the proof of the first claim.}
\begin{proof}
We now consider the sphere $S_4 = S^2 \times p_+$.  We repeat the argument from the previous two steps.  Namely, we take 
\[ \gamma'(q_1, q_2) = (\beta'_{\alpha(z(q_2))} q_1, q_2)\]
and consider $\phi' \circ \gamma'$. 
The same considerations as above apply: $\phi' \circ \gamma'$ fixes neighborhoods of all four $S_i$, $\gamma'^*(\om_{std})$
 is standard in a neighborhood of all the $S_i$, and the line $t  \gamma'^*(\om_{std}) + (1-t) \omega_{std}$ is symplectic.  Thus, after composing with a Moser isotopy, there exists the required map.
 \MS
 
 \NI
 {\bf Step 5.}  {\it Proof of the second claim.}
 Gromov~\cite{GRO} showed in 1985 that the identity component of the full group 
  of symplectomorphisms of $(S^2\times S^2, \om_{std})$ has the homotopy type of $\SO(3)\times \SO(3)$, and this proof easily adapts to show that the subgroup considered here deformation retracts to the subgroup of 
$\SO(3)\times \SO(3)$ that fixes a neighborhood of the four spheres $\bigcup_{i=1}^4 S_i$. But this consists only of the identity element.  
This completes the proof. For more details see \cite[Ch.9.5]{JHOL}.
\MS

This completes the proof of Lemma~\ref{lem:s2s2}. \end{proof}

\subsection{Proof of Theorem~\ref{thm:main}}\label{ss:proofmain}

Let us now proceed with the proof of Theorem~\ref{thm:main}.  The crux of the issue is the following result:

\begin{thm}
\label{thm:main1}
Let $\Omega_1$ be rational and concave, with rational weights $(a_1, \ldots, a_n)$ and let $\Omega_2\subset \R^2_{>0}$ be  convex with rational weights $(b; b_1, \ldots, b_m)$.  Then there is a symplectic embedding
\[ X_{t\Omega_1} \to \intt\, (X_{\Omega_2})\]
for all $0<t<1$ if and only if there is a symplectic embedding
\[ \bigsqcup_i \intt\, (B(a_i)) \sqcup \bigsqcup_j \intt\, (B(b_j)) \to \intt\, (B(b)).\]
\end{thm}

Let us defer the proof for the moment and explain how it implies Theorem~\ref{thm:main}.  

\begin{proof}[Proof of Theorem~\ref{thm:main}, assuming Theorem~\ref{thm:main1}.]

The argument is generally similar to the argument in \cite{CG, Mell}.
\MS

\NI {\bf Step 1.} { \it (i) implies (ii) and (iii).}

A result due to Traynor  \cite{Traynor} implies that if $\Delta$ is equivalent to the triangle with vertices $(0,0), (z,0)$ and $(0,z)$ after applying an integral affine transformation, then $X_\Delta$ contains a symplectically embedded ball $\intt\,(B(z))$.  Hence, (i) implies (ii).
The fact that (i) implies (iii) follows from the Monotonicity Property of ECH capacities stated at the beginning of \S\ref{ss:prelim}.

\MS

\NI {\bf Step 2.} {\it  (ii) implies (i).}   \

First note that it suffices to prove the result for the case where $\Omega_2$ is completely off the axes.  Indeed, an embedding into $\intt\, (X_{\Omega_2})$ in this case gives an embedding into the corresponding $\intt\, (X_{\Omega_2})$ in the other cases, since $\intt\, (X_{\Omega_2})$ is a subset of the interior in the other cases, and the weights in the other cases are the same.  We will therefore assume that $\Omega_2$ is completely off the axes in the remainder of this proof.

Now fix a parameter $\lambda \in (1,2)$.  Then there exists a rational concave domain  $\Omega'_1$ with weights $(a'_1, \ldots, a'_n)$ and a good convex $\Omega'_2$ with weights $(b; b'_1, \ldots, b'_m)$ such that 
each $a'_i \le a_i$, each $b'_j \le b_j$, and
\[ 
\frac{1}{\lambda} \Omega_1 \subset \intt\, (\Omega'_1) \subset \Omega_1, \qquad 
 \Omega_2 \subset \Omega'_2 \subset \lambda_{\ell_1,\ell_2} \cdot \Omega_2.\]
Here by $\lambda_{\ell_1,\ell_2} \cdot \Omega_2$ we mean the radial expansion of $\Om_2$ by the factor $\la>1$ with center $(\ell_1,\ell_2)\in \intt (\Om_2)$ via the inverse of the  map $f_\la$ in \eqref{eq:retract}.
Then, by Theorem~\ref{thm:main1}, we have an embedding
\[ \frac{1}{\lambda^2} X_{\Omega_1} \to \intt (X_{\Omega_2}), \quad \mbox{ for all } \la>1,
\]
which gives the desired embedding in view of Corollary~\ref{cor:sta}.
\MS

\NI
{\bf Step 3.}  {\it  (iii) implies (ii).}\
By (iii), we have that
\[ c_k(X_{\Omega'_1}) \le c_k(X_{\Omega'_2}) \quad \forall k.\]
Thus, by the definition of the weight expansion and the Traynor trick, we obtain that
\[ 
c_k\Bigl( \bigsqcup\, \intt\, (B(b'_j)) \sqcup \bigsqcup \intt\, (B(a'_i)) \Bigr) \le c_k(B(b)) \quad \forall k.
\]
It is known that ECH capacities give a sharp obstruction to all ball packing problems of a ball [Hutchings, Remark 1.10]. 
Hence (ii) holds, as claimed.
\end{proof}

We now prove Theorem~\ref{thm:main1}.

\begin{proof}[Proof of Theorem~\ref{thm:main1}.]

The \lq\lq only if\rq\rq\, direction follows from the Traynor trick described in Step 1 of the proof of Theorem~\ref{thm:main}, so we just have to prove the \lq\lq if\rq\rq\, direction.  We prove this in steps, by a modification of the inflation method.
\MS

 We embed $ X_{r\Omega_1}$ for $r$ small and rational into $\intt\, (X_{\Omega_2})$, and denote by $(Y, \om_Y)$ the symplectic manifold obtained from $\ov X_{\Om_2}$ by  blowing up $ X_{r\Omega_1}$ as in Lemma~\ref{lem:blowupdown}. Thus $Y$ contains two chains of spheres: $\Cc_Y$, which is  the blowup of $ X_{r\Omega_1}$ and  $\Cc_2$ which is formed by collapsing the circle orbits in  $\p X_{\Omega_2}$.
 
\MS
  
 \NI
{\bf Step 1.} {\it  For every $\eps>0$, there is a symplectic form $\om_Y'$ on $Y$ that agrees with $\om_Y$ near $\Cc_2$ and with $s \om_Y$ on $\Cc_1$, where $\frac 1r-\eps< s <\frac 1r$.}

As in \cite[Thm. 2.1]{CG},
this form  is constructed by the inflation process described in  \cite{McOp}. (For a very simple example, see
\cite[\S2.1]{Mell}.)
 The inflation requires a pseudoholomorphic curve $C$, whose intersection number with the sphere $S_i$ in
$\Cc_2$ is $\ka a_i$, while that with the sphere $S_j$ in $\Cc_1$ is $\ka b_j$, where $\ka$ is a large constant such that $\ka a_i, \ka b_j$ are integers.  The existence of such a curve $C$ follows from Seiberg--Witten theory and the existence of the ball embedding.
This process yields a form that has the desired integrals over the spheres   in $\Cc_1$ and $\Cc_2$. Indeed, we rescale so that the sizes of the spheres in $\Cc_2$ do not change during the inflation, while
those of $\Cc_1$ increase from $ra_i$ to $\frac{r + \ka }{1+\ka} a_i$, and hence become arbitrarily close to $ a_i$ as $\ka$ increases.
  Further these spheres 
 remain symplectically orthogonal  since, apart from the rescaling,  the form is changed only near the curve $C$ along which we inflate.  Thus, by the symplectic neighborhood theorem, a neighborhood $N$ of $\mathcal{C}_{2}$ is symplectically isotopic to its toric model with the standard symplectic form.  Moreover, we can extend this isotopy to $Y$ by the identity: more precisely, the isotopy is induced by a (time-dependent) vector field $v$ and we can cut off this vector field via pointwise multiplication with a compactly supported function $f$ in $N$ such that the isotopy generated by $v$ and $fv$ agree on a small-enough sub-neighborhood of 
 $\Cc_2$  in $N$.  Thus, there is an ambient isotopy of $Y$ that maps the form obtained by inflation
 to a symplectic form $\omega'$ that is standard in a neighborhood of $\Cc_2$, as claimed.  

\MS

 \NI
 {\bf Step 2.} {\it  Completion of the proof.}
  Since $X_{r\Om_1}$ is diffeomorphic to $X_{\Om_1}$, and  the new symplectic form $\om'$ on  $Y$  restricts on $\Cc_1$ to
  $s$ times the original form,  
  the blowdown of $Y$ along $\Cc_1$ is diffeomorphic to $\ov X_{\Om_2}$.
Moreover, this blowdown is equipped with a symplectic form $\om$ that agrees with $\om_{std}$  near $\p(\ov X_{\Om_2}) = \Cc_2$ 
and by construction there is a symplectic embedding $\io:(X_{\Om_1}, t \om_{std})\to (\intt\,  X_{\Om_2},\om)$, where $t: = rs$ is arbitrarily close to $1$. 
But by Proposition~\ref{prop:uniq}, there is a compactly supported diffeomorphism $\psi$ of  $\intt\,  X_{\Om_2} =  
 \ov X_{\Om_2}\less \Cc_2$ such that $\psi^*(\om) = \om_{std}$.   Therefore $\psi\circ \io$ is the desired embedding
 $(X_{\Om_1},t \om_{std})\to (\intt\,  X_{\Om_2},\om_{std})$.
 \end{proof}
 
 \subsection{Subtraction formula for the ECH capacities}\label{ss:subtract}
 
The aim of this section is to prove the following formula.

\begin{lemma}\label{lem:ECHk} Let $X = X_\Om$ where $\Om =\Om(b;b_1,\ldots)$.
Then
\begin{equation}
\label{eqn:echsub}
c_k(X) = \min_{k = \ell - k_1 - \ldots - k_m}\ c_\ell B(b) - c_{k_1} B(b_1) - \ldots - c_{k_m} B(b_m).
\end{equation}
\end{lemma}

This formula is proved in  \cite[Thm.A.1]{CG}  in the case when the convex set $\Om$ contains a neighborhood of the origin. 
The proof in the general case follows by essentially the same argument; the crucial new step is proved in  Lemma~\ref{lem:needed}.
 Below we use the notation  of \cite{CG} that was introduced in \S\ref{ss:length}.

\begin{proof}
Let $X$ be our given convex toric domain, with weight expansion $(b; b_1, \ldots)$.   The strategy of proof is to prove, for all $k$, the following inequality
\begin{equation}
\label{eqn:needed}
c_k(X)\ \le \ \min_{k_1 - \ell = k}\ c_{k_1}(X) - c_\ell\bigl(\sqcup B(b_i)\bigr)\ \le \ c_k(X).
\end{equation}
This implies \eqref{eqn:echsub}.  Indeed, in the case where there are finitely many $b_i$, \eqref{eqn:echsub} follows immediately from \eqref{eqn:needed} by the Disjoint Union axiom, applied to $\sqcup B(b_i)$; the same argument works in the infinite case, because the ECH capacities are defined in this case as a supremum over embedded compact Liouville domains.
\MS

\NI {\bf
 Step 1.}  {\it We have $c_k(X)\ \le \ \min_{k_1 - \ell = k}\ c_{k_1}(X) - c_\ell\bigl(\sqcup B(b_i)\bigr)$.}

 By definition of the weight sequence, there is a symplectic embedding
\[ X \sqcup ( \sqcup_i B(b_i)) \to B(b).\]
It then follows from the Disjoint Union and Monotonicity Axioms that for any $k$ and $\ell$,
\begin{align}\label{eqn:i1} 
c_{k}(X) + c_\ell(\sqcup_i B(b_i) ) \le c_{k + \ell}(B(b)).
\end{align}
 
 As explained in \S\ref{ss:cut}, the $b_i$ are, canonically, the weights of (possibly empty) concave toric domains $X_0, X_1$ and $X_2$.  For any concave toric domain, its ECH capacities agree with that of its canonical ball packing.  Hence, we obtain the equality
\begin{equation}
\label{eqn:concaveball}
c_{\ell}( \sqcup B(b_i)) = \max_{k_0 + k_1 + k_2 = \ell}( c_{k_0}(X_0) + c_{k_1}(X_1) + c_{k_2}(X_2)),
\end{equation}
since both sides equal $\max\ \bigl\{ \sum_i c_{\ell_i}(B(b_i)) \ \bigl| \sum \ell_i = \ell\bigr\}.$ Hence, combining with \eqref{eqn:i1} yields the inequality
\begin{equation*}
\label{eqn:i2}
c_{k}(X) \le\min_{\ell - k_0 - k_1 - k_2 = k} c_{\ell}(B(b)) - c_{k_0}(X_0) - c_{k_1}(X_1) - c_{k_2}(X_2).
\end{equation*}
 In view of \eqref{eqn:concaveball}, this completes the proof of Step 1. 
\MS

\NI {\bf Step 2.}   {\it We have $ \min_{k_1 - \ell = k}\ c_{k_1}(X) - c_\ell\bigl(\sqcup B(b_i)\bigr)\ \le c_k(X).$}

  We  prove  this by showing that, given $k$, there exists $k_0, \ldots, k_2$ such that
\begin{equation}
\label{eqn:other}
c_k(X) \ge c_{k + k_0+k_1 + k_2} B(b) - \sum_i c_{k_i}(X_i),
\end{equation}
which,  in view of \eqref{eqn:concaveball}, then implies the desired inequality.

To proceed, given a convex lattice path $\Lambda$ (possibly closed), let $L(\Lambda)$ denote the number of lattice points in the region bounded by $\Lambda$.  For a concave lattice path $\Lambda$, let $L'(\Lambda)$ denote the number of lattice points, not including lattice points on the upper boundary.  In the case (for example) of convex toric domains $X'$ that are completely off the axes, we have that
\begin{equation}
\label{eqn:formula}
c_k(X') =\min \lbrace \ell_{\Omega}(\Lambda) \rbrace,
\end{equation}
where the minimum is over closed lattice polygons $\Lambda$ with $L(\Lambda) = k + 1$; see \cite[Thm. 1.11]{Hutchq}.

We now use this formula to get a lower bound on $c_k(X_{\Omega})$.  

Let $\Lambda$ be the boundary of a convex lattice polygon $P$ with $L(\Lambda) := k+1$, that we assume translated so that it meets both axes and also the slant edge of the simplex $T(b')$.
  As in Lemma~\ref{lem:needed}, we decompose $\Lambda$ into (possibly empty) lattice paths $\Lambda_0, \Lambda_1$ and $\Lambda_2$, together with (possibly empty) edges on the axes or on the slant edge of  $T(b')$.
Then, $\Lambda_0, \La_1, \Lambda_2$ are affine equivalent to concave lattice paths $\Lambda'_0, \Lambda'_1$ and $\Lambda'_2$, and we define $k_i := L'(\Lambda'_i)$, for  $i=0,1,2$. We further denote by
 $\Lambda'_3$  the  slant edge of $T(b')$ 
 and we define $k_3 := L(\Lambda'_3)$.

Now we observe the following.  First of all, $k = k_3 - k_0- k_1 - k_2 - 1$.   Hence
\begin{equation*}
\label{eqn:obs1}
 c_{k + k_0+ k_1 + k_2}(B(b)) = c_{k_3-1}(B(b)) = \ell_{T(b)}( \Lambda'_3).
 \end{equation*}
Moreover, by the formula for the ECH capacities  of concave toric domains given in \cite[Thm.1.21]{Choi},
\begin{equation}
\label{eqn:obs2} 
\ell_{\Omega_i}(\Lambda'_i) \le c_{k_i}(\Omega_i), \quad 0 \le i \le 2.
\end{equation}
Now recall the following identity 
\begin{equation}
\label{eqn:needed20}
\ell_\Omega(\Lambda) = \ell_{T(b)}(\Lambda'_3) - \sum^2_{i=0}\ell_{\Omega_i}(\Lambda'_i). 
\end{equation}
from Lemma~\ref{lem:needed}.
If 
$\Omega$ is off the axes, then, in view of \eqref{eqn:formula}, we may combine \eqref{eqn:needed20} with \eqref{eqn:obs1} and \eqref{eqn:obs2} to obtain \eqref{eqn:other}.
On the other hand, if  $\Omega$ touches the axes, we  reduce to the off--axes case by noting that $\Omega$ contains its dilation by any factor $\lambda < 1$ about an interior point. Therefore, by monotonicity, the above lower bound still holds.
This completes the proof. \end{proof}

\subsection{Elementary ECH capacities}\label{ss:elem}

We now show that
the key formula \eqref{eqn:needed} also holds for the elementary ECH capacities  $c^{elem}_k, k\ge 0,$ that are
defined in Hutchings~\cite{Helem}.  The previous work \cite{Helem} showed that these 
 capacities agree with  the standard ECH capacities  for 
certain convex toric domains that are more restricted than the ones we consider here (for these convex toric domains, the region $\Omega$ is required to be the region between the graph of a concave function and the axes) 
 but have the advantage that
  the elementary capacities for $\C P^2$ are known to agree with those for the ball, while the corresponding statement for the standard capacities is at present unknown.    This fact is useful for us because we want to study packings into manifolds without boundary; see Corollary~\ref{cor:super}.

To begin, we recall the definition of  $c^{elem}_k$ for nondegenerate
 Liouville domains $X'$.  We complete $X'$ by attaching cylindrical ends, denoting the resulting manifold by $\overline{X}'$, and define 
\begin{equation}
\label{eqn:defnelemECH}
c^{elem}_k(X') = \text{sup} \quad \text{inf} \quad E(u),
\end{equation}
where the sup is over all choices of $k$ points and cobordism-compatible almost complex structures $J$, the inf is over $J$-holomorphic curves $u$ passing through these points, asymptotic at $+\infty$ to a Reeb orbit set, and $E(u)$ denotes the energy of the curve.  For the $X'$ that are actually relevant here, all $u$ are asymptotic to Reeb orbit sets and $E(u)$ is just the action of the corresponding orbit set.  
Recall also that a cobordism-compatible almost complex structure is  of symplectization type on the cylindrical end and is compatible with the symplectic form on $X$.  


We now show  that
the two definitions do agree on
 the general convex toric domains considered here. Our proof
adapts arguments from \cite[Sec. 5]{Helem} in combination with a generalized ``ECH index\rq\rq\,  axiom.

\begin{prop}\label{prop:elem}
Let $X$ be any convex toric domain.  Then $c^{elem}_k(X) = c_k(X)$.
\end{prop}
\begin{proof}

\vspace{3 mm}

\NI {\bf Step 1:} {\it
The upper bound for $c^{elem}_k(X)$ follows from known work.}  First, we note that by \cite[Thm. 6.1]{Helem},
 we have
\begin{equation}
\label{eqn:upperboundelemech} 
c^{elem}_k(X) \le c_k(X)
\end{equation}
whenever $X$ is a four-dimensional Liouville domain.   Note that this implies that the elementary ECH capacities of $X$ are finite --- easier arguments  suffice for this, but we will anyways want the strength of the bound \eqref{eqn:upperboundelemech}.  
\MS

\NI
{\bf Step 2}: {\it  The proof, modulo the ``ECH index\rq\rq\, property, in the ``off the axes case\rq\rq. } 
Let us now consider the problem of finding a matching lower bound.  Let us assume to start that $X$ is a convex toric domain completely off the axes, with smooth boundary.
Then $X$ is a four-dimensional Liouville domain.  Let $L = c^{elem}_k(X) + 1.$ The boundary of $X$ is degenerate (since it is Morse-Bott).  However, standard theory (see e.g. \cite{Hutchq}) implies that we can perturb $X$ to a nondegenerate Liouville domain $X'$ with the following properties:
\begin{itemize}
\item The orbits of $\partial X'$ of action $\le L$ are entirely in the torus fibers.  
\item The nullhomologous (in $\partial X')$ orbit sets $\alpha$ of $\partial X'$ of action $\le L$ correspond to labeled convex polygonal lattice paths $\Lambda_{\alpha}$.    The action of such an orbit set is given by $\ell_{\Omega} (\Lambda_{\alpha})$, where $\ell_\Omega$ is as defined in \eqref{eq:Omlength}.
\item An $\alpha$ from the previous item has a well-defined, integer valued ECH index $I$, and it satisfies the following bound: $I(\alpha) \le 2(L(\Lambda_{\alpha}) - 1)$ 
\end{itemize}

Now, by the Spectrality Axiom for the elementary ECH capacities proved by Hutchings in 
\cite[Thm. 4.1]{Helem},
$c^{elem}_k(X') = \mathcal{A}(\alpha)$, where $\alpha$ is a Reeb orbit set, and $\mathcal{A}$ denotes its symplectic action.  The Spectrality Axiom also implies that $\alpha$ can be assumed  nullhomologous in $X'$; thus, by the first bullet point $\alpha$ is nullhomologous in $\partial X'$.  In general, the ECH index of an orbit set $\alpha$ is only defined up to ambiguity of the divisibility of $c_1(\xi) + 2 PD([\alpha])$; in the present case, though $c_1 = 0$, and we have just shown that $\alpha$ is nullhomologous, so $\alpha$ has a well-defined integer valued ECH index $I$.

We now claim that $\alpha$ can in addition be assumed to satisfy the following:
\begin{equation}
\label{eqn:echk}
I(\alpha) \ge 2k.
\end{equation}
We defer the proof of \eqref{eqn:echk}, which is a bit technical, to the subsequent steps, and first explain why it implies that 
Proposition~\ref{prop:elem} holds.
Combining \eqref{eqn:echk}, and the second and third bullet points above, we see that $c^{elem}_k(X')$ is the 
$\ell_{\Omega}$ length of a convex polygonal lattice path with at least $k+1$ lattice points.  However, by known formulas (similar to the formulas in the previous section), e.g. from
\cite{Hutchq},
 $c_k(X')$ is the minimum of the $\ell_{\Omega}$ length over such paths.  Hence, it follows that $c^{elem}_k(X') \ge c_k(X')$.  Hence, $c^{elem}_k(X') = c_k(X')$ in view of 
\eqref{eqn:upperboundelemech} 
and so by continuity the same formula holds for $X$.  
\MS

\NI
{\bf Step 3:} {\it  The generalized ECH Index Axiom --- first considerations.}  It thus remains to explain the proof of \eqref{eqn:echk}.  This is very similar to the proof of the ``ECH Index\rq\rq\, Axiom for ECH capacities, proved by Hutchings in \cite[Thm. 4.1]{Helem}.
 However, it is stated by Hutchings only for star-shaped domains, and so does not directly apply here.  

Recall from \eqref{eqn:defnelemECH}
 that $c^{elem}_k(X')$ is the supremum over $J$ of the minimal energy of a $J$-holomorphic curve passing through $k$ general points.
Following Hutchings, we note that:
\begin{itemize}\item[-] All relevant curves $u$ are somewhere injective, since for any multiply covered curve we can look at its underlying somewhere injective part, reducing the energy without changing the point constraints.
\item[-]  Since $X'$ has  nondegenerate boundary, we can replace the $\text{sup}/ \text{inf}$ with a $\text{max} / \text{min}$, because the set of possible actions of  an orbit set is discrete. 
\item[-]  If $u$ is a somewhere injective $J$-holomorphic curve in the completion $\overline{X}'$ of $X'$, passing through $x_1, \ldots, x_k$ and asymptotic to an orbit set $\alpha$, the ECH index inequality implies that
\begin{equation}
\label{eqn:indexineq}
\ind(u) \le I(u),
\end{equation}
where $I(u)$ is the ECH index and $\ind(u)$ is its Fredholm index.
\item[-] Finally,  the first and second bullet points in Step 2 imply that $\alpha$ is actually nullhomologous in $\partial {\ov X}'$. 
\end{itemize}

\MS

\NI {\bf  Step 4.}  {\it The ECH index computation.}  We now claim that all relevant curves satisfy
 $I(u) = I(\alpha)$.  To prove this, it suffices to show that the ECH index of a relative homology class in $X'$ depends only on its asymptotics.  So, let $Z$ and $Z'$ be two different elements in $H_2(X;\alpha)$.  Then, $S := Z'- Z \in H_2(X)$ (i.e. is a homology class, not a relative one).  Recall that the ECH Index is defined by the formula $I(Z) = c_{\tau}(Z) + Q_{\tau}(Z) + CZ^I_{\tau}(Z).$  Thus it suffices to show that
\[ c_{\tau}(S) = 0 \quad \quad Q_{\tau}(Z + S) = Q_{\tau}(Z),\]
where $c_{\tau}$ denotes the relative Chern class and $Q_{\tau}$ denotes the relative intersection pairing.   

Let us start by considering the Chern class term.  Because the homology of $X$ is generated by the homology of a fiber, we can assume that $S$ is a fiber; we can then assume that $S$ is a fiber close to $\partial X$.  We therefore need to compute $ch(\Lambda^2 TX)$, evaluated on $S$.  We can assume that a neighborhood of $\partial X$ is identified with a neighborhood of $Y \times \lbrace 0 \rbrace$ in its symplectization, identifying the given almost complex structure with a symplectization admissible almost complex structure $J$, and identifying $\partial X$ with $Y \times \lbrace 0 \rbrace$.  There is then a frame given (in symplectization coordinates) by $\lbrace \partial_\theta, J \partial_\theta, \partial_s, J \partial_s \rbrace$.  Here, $\theta$ denotes the restriction of the angular vector field on the moment plane, relative to a choice of origin anywhere in the interior of the moment image of $X$; and, $s$ denotes the $\mathbb{R}$-coordinate in the symplectization.  Thus, $TX$ is actually trivial as a complex vector bundle in this neighborhood, so $c(\Lambda^2 TX)[S]$ vanishes (for example, we could just take the nowhere vanishing form given by $\lbrace \partial_\theta \wedge \partial_s \rbrace$ for the pullback of $\Lambda^2 TX$ to the chosen representative of $S$).

We now consider the relative intersection pairing.  By bi-linearity of the intersection pairing, it suffices to show that 
\[ Q_{\tau}(S,S) = 0, \quad Q_\tau(Z,S) = 0.\]  
That $Q_{\tau}(S,S) = 0 $ follows by taking two disjoint representatives of a fiber.   

That $Q_\tau(Z,S) = 0$ follows by again using our symplectization model.  To elaborate, we can choose $s$ close to $0$ so that a representative for $Z$ intersects $\lbrace s \rbrace \times Y$ transversally; call this oriented one-manifold $h$.  Then, we can choose a representative $S'$ for $S$ that also lies entirely in $\lbrace s \rbrace  \times  Y$ and that is transverse to $h$.  Then   
\[ Q_\tau(Z,S) = \# \lbrace h \cap S' \rbrace.\]
However, $h$ is nullhomologous in $\lbrace s \rbrace \times Y$, since it is nullhomologous in $\lbrace 0 \rbrace  \times Y$ by Step 4, hence the above count of intersections vanishes.
\MS

\NI {\bf Step 5.} {\it Completion of the proof when $\Om$ is off the axes.}  

For generic choice of $J$ and $x_1, \ldots, x_k$, we must have $ind(u) \ge 2k$.  Hence, by the index inequality \eqref{eqn:indexineq}, we must have $I(\alpha) \ge 2k$.  This gives the desired inequality in the generic case.  However, in fact the $\text{max} \text{min}$ in the definition of $c_k^{elem}$ can be assumed to occur under generic choice of data.  Indeed, certainly the maxmin is bounded from below by the subset of generic data.  On the other hand, if the max is achieved at $(J, x_1, \ldots, x_k)$, then one can approximate this data with generic data, and apply a Gromov compactness argument: if the curves for the approximating data have energies accumulating to a point strictly less than $c^{elem}_k$, then they must limit to a curve for $(J, x_1, \ldots, x_k)$ that also has strictly smaller energy than the maxmin (noting, again, that the set of actions of orbit sets is discrete).  In Remark~\ref{rmk:compact}, we give some further 
details of this Gromov compactness argument (which uses Taubes' compactness for currents, to circumvent the lack of an a priori genus bound).
\MS

\NI {\bf Step 7.} {\it General convex domains.}

We now consider a general convex toric domain, not necessarily off the axes, and not necessarily a Liouville domain, and prove the same formula in this case by reducing to the off the axes case.  The ECH capacities of such a domain can be defined as a supremum over embedded nondegenerate Liouville domains.  The elementary ECH capacities of such a domain are defined as a supremum over embedded ``admissible\rq\rq\, domains.   An admissible domain is a disjoint union of Liouville domains and closed symplectic $4$-manifolds.  In the present situation, $X$ contains no closed symplectic $4$-manifolds; hence, the elementary ECH capacities and the ECH capacities are supremums over the same set.  Thus, it follows from \eqref{eqn:upperboundelemech} that this upper bound holds for general convex toric domains as well.  On the other hand, we can find a matching lower bound for our $X$ by approximating a general convex domain by Liouville domains off the axes.
\end{proof}

\begin{rmk} \label{rmk:compact} {\it  Further details of the Gromov compactness argument.}

For completeness, we provide more details about the Gromov-Taubes compactness argument above

As explained in Step 5 above, it suffices to show that if $J_i \to J$ in $C^{\infty}$, then $J_i$-holomorphic curves $u_i$ with uniformly bounded action converge to a $J$-holomorphic curve that is admissible for the maxmin for the elementary ECH capacities.   The subtlety is that there is no a priori bound on the genus, so we have to use a compactness theorem due to Taubes and generalized by Doan-Walpulski \cite{dw}.

Taubes' compactness gives convergence as a current and as a point set to a proper holomorphic map.  
The argument  requires a compact manifold, possibly with boundary, and to apply it in this case we take an exhaustion of $\overline{X'}$ by compact sets and pass to a subsequence; this gives convergence on compact sets to some proper holomorphic map $u$  and what remains is to show that the domain of $u$ can be assumed a finitely punctured compact Riemann surface.  The first point is that, as in the proof of Step 3 above, we can assume that the $u_i$ are all asymptotic to the same orbit set $\alpha$.  Next, by again applying Taubes compactness, except to translates of the portions of the $u_i$ in the symplectization end of $\overline{X'}$, and 
arguing as in \cite[Lem. 5.11]{hutchings_lecture},
we can guarantee convergence to a possibly broken holomorphic current, whose action is the action of the $u_i$.  This implies that the $u_i$ can be assumed to have homology class independent of $i$.  As in the argument above, we can next assume that the $u_i$ have no multiply covered components.  Then we can apply the relative adjunction formula (see e.g. \cite[Eq. (3.3)]{hutchings_lecture}): the relative adjunction formula implies a lower bound on the Euler characteristic of the $u_i$, given an upper bound on the writhe, and a bound on the writhe is known, see e.g. \cite[Eq. (3.9)]{hutchings_lecture}.  There is also a bound on the number of components of $u_i$, because the $u_i$ have $k$-marked point constraints, and we can discard any component not passing through a marked point without increasing energy.  Thus, by passing to a subsequence, we can conclude that each component of the $u_i$ has genus and number of punctures independent of $i$, and then the standard SFT compactness theory applies to give the desired convergence to a finitely punctured compact Riemann surface.
\end{rmk}

\section{The accumulation point theorem}\label{sec:acc}

Let $\Om$ be a convex domain. This section proves Theorem~\ref{thm:acc}, stating that the steps of any staircase in $X_\Om$  must converge 
to a special point that depends only on the volume and perimeter of $\Om$.

\subsection{Preliminaries}\label{ss:acc1}

In this section, we are concerned with the capacity functions for (generalized) convex toric domains $X_\Omega$ as well as for the rational symplectic $4$-manifold $M_{\mathcal{B}}$ whose symplectic form is encoded by the blowup vector $\mathcal{B}:=(b;(b_j))$. As we will see, if $\Omega(b;b_1,\hdots,b_n)$ and $M_{\mathcal{B}}$ are encoded by the same sequence $(b;(b_j))$, then the capacity functions are equal. If there is a well-defined $X_\Omega$ or $M_{\mathcal{B}}$ given by $(b;(b_j))$, we let $c_{(b;(b_j))}$ denote the capacity function for $X_\Omega$ or $M_{\mathcal{B}}$.

We analyze the capacity function $c_{(b;(b_j))}$ by the method used in \cite{AADT}, that goes back to \cite{Mell,ball}.  We refer the reader to \cite[\S2.3]{AADT} or \cite[\S2]{BHM} for more background details.

Each ellipsoid $E(1,z)$, where $ z\ge 1,$ can be identified with the concave domain $X_{\Om'}$, where $\Om'$ is the triangle with vertices $(0,0), (z,0), (0,1)$. Thus, as described in \ref{ss:cut}, it has a {\bf weight decomposition} $\bw(z) = (a_1,a_2,\dots)$ with $a_1=1$ and $a_i\ge a_{i+1}$.
This is finite, ending in copies of $1/q$, exactly if $z= p/q$ is rational with $\gcd(p,q)=1$.

The entries $a_i$ of $\bw(p/q)$ satisfy the identities
\begin{align}\label{eq:wtai}
\sum a_i = p/q+1-1/q,  \qquad \sum_i a_i^2 = p/q.
\end{align}
The following conditions are known to be equivalent:\footnote
{The equivalence of the first three conditions is proved in \cite{Mell}, while that of the fourth needs in addition \cite[Prop.1.9]{Hutchq}.}
\begin{itemize}\item[{\rm (a)}]
the ellipsoid $E(1,p/q)$ embeds symplectically in $\C P^2(b)$\footnote
{ $\C P^2(b)$ is the complex projective plane with a K\"ahler form that integrates over the line $L$ to $b$.};
\item[{\rm (b)}] the disjoint union of $n$ open  balls of sizes $a_1,\dots, a_n$ embeds into the open ball $\intt\,  B(b)$;
\item[{\rm (c)}] there is a symplectic form on the $n$-fold blowup $\C P^2\# n \ov{\C P}\,\!^2$ of $\C P^2(b)$ in the class Poincar\'e dual to $bL - \sum_{i=1}^n a_i E_i$, where $L$, resp. $E_i$, is the homology class of the line, resp. of  the blown up points.
\item[{\rm (d)}] the {\bf volume constraint} $b^2  - \sum_{i=1}^n a_i^2 > 0$ holds, and also $\sum_{i=1}^n d_ia_i < db$ whenever $(d;d_1,\dots,d_n)$ are nonnegative integers (not all zero) such that 
$\sum_{i=1}^n(d_i^2+d) \le d^2 + 3d$.
\end{itemize}
The obstructions to the existence of the symplectic form in (c) are the volume constraint together with the 
 the constraints imposed by the {\bf exceptional divisors} $\bE$ in $\C P^2\# n \ov{\C P}\,\!^2$.
These classes $\bE$ are the homology classes $dL-\sum_{i=1}^n m_i E_i$ in $\C P^2\# n \ov{\C P}\,\!^2$ of symplectically  embedded spheres of self-intersection $-1$, and, because they have nonzero Gromov--Witten invariant, always   have symplectic representatives even as the cohomology class of the symplectic form varies.
It follows that every symplectic form $\al$ on $X_n(b)$ must integrate positively over every such $\bE$, that is
the class  $
{\rm PD}(\al): =b L - \sum_{i=1}^n a_i E_i$  must satisfy
\begin{align}\label{eq:constraint}
b d - {\textstyle \sum_{i=1}^n} d_i a_i > 0,\quad \mbox{ for all exceptional }\; \bE = dL-{\textstyle\sum_{i=1}^n} d_i E_i.
\end{align} 
Notice that because 
\begin{align}\label{eq:Dioph}
c_1(\bE) = 3d-\sum_{i=1}^n  d_i=1. \quad\mbox{ and } \quad \bE\cdot\bE = \sum_{i=1}^n d_i^2= -1,
\end{align}
 the tuple $(d;d_1,\dots,d_n)$ satisfies the condition in (d) above. However, condition (d) is more general in that it can be satisfied by classes $\bE$ satisfying the equations in \eqref{eq:Dioph} that are not represented by exceptional spheres. Nevertheless, it follows from Seiberg--Witten theory that these classes do all have (possibly disconnected) $J$-holomorphic representatives, and (d) (which is proved in \cite{Hutchq} using ECH) shows that these classes still impose embedding restrictions. 
In the following, we will call a class $\bE=dL-\sum_{i=1}^n d_i E_i$ that satisfies \eqref{eq:Dioph} a {\bf quasi-exceptional class}.
The exceptional classes are distinguished by the fact that they intersect all the other exceptional classes non-negatively --- a fact that can have important repercussions.
\MS

In the situation at hand, the target is not the ball $B(b)$ but the convex domain $X_\Om$ or the closed manifold $M_{\mathcal{B}}$.  However, by Theorem~\ref{thm:main} and the above discussion about embedding ellipsoids,
 an ellipsoid $E(1,z)$ (which is concave) embeds into $X_\Om$ or $M_{\mathcal{B}}$ if and only if there is an embedding 
$$
 \bigsqcup_{1\le i\le n}\intt\, B(a_i) \sqcup \bigsqcup_{j\ge 1}\intt\, B(b_j) \;\se \; B(b),
$$
where $X_\Om = \Om(b; (b_j))$ or $\mathcal{B}=(b;(b_j))$ and $\bw(z) = (a_i)$.

To clarify notation,  we shall denote the classes in the blowup of $\C P^2(\mu b)$ corresponding to the balls $B(a_i)$ by $E_i, 1\le i 
\le n$, and those corresponding to the negative weights $b_j, j\ge 1,$ by  $\TE_j$.
As in \cite{AADT} we denote quasi-exceptional classes as
\begin{align}\label{eq:E}  
\bE = dL  - \sum_{j\ge 1}\Tm_j \TE_j - \sum_{1\le i\le n} m_i E_i,
\end{align}
 where we assume\footnote
 {
 Since the entries of $(b_j)$ are always assumed nonincreasing, the obstruction from a given class $\bE$ is always greatest when 
 its entries are ordered. Therefore, because reordering does not affect the breakpoint of a class, when analyzing the capacity function we may always assume that the obstruction classes are ordered.}
    the $\TE_j$ and $E_i$ are {\bf ordered} so that $\Tm_1\ge \Tm_2\ge\dots$, and $m_1\ge m_2 \ge \dots \ge m_n$.
By \eqref{eq:constraint}, for every quasi-exceptional class $\bE$,  the size $\mu$ of the target 
$\mu X_\Om$ or $\mu M_{\mathcal{B}}$ must satisfy
$$
\mu b d \ge \sum_i m_ia_i + \sum_j \Tm_j (\mu b_j).
$$
Therefore, for $(b;(b_j))$ and rational $z$ with weight decomposition $(a_i)$, we always have
\begin{align}\label{eq:mu}  
c_{(b;(b_j))}(z)\ge  \mu_{\bE}(z) = \frac{\sum_i m_ia_i }{db - \sum_j \Tm_j b_j}, 
\end{align}
where $(a_i)$ is the weight decomposition of $z$.
\MS

In order for the obstruction $ \mu_{\bE}(p/q)$  from  $\bE$ to be larger than the volume $V_{(b;(b_j))}(p/q)$, and hence potentially visible in the capacity function $c_{(b;(b_j))}$, 
there must be a rational point $z=p/q$ where $\mu_{\bE}(p/q)$ is  
 larger than the volume constraint $V_{(b;(b_j))}(p/q)$.   If such a point exists, we say that $\bE$ is {\bf obstructive} at $z$.  In this case,
it turns out that there is a unique point $a=p/q$ (called the {\bf break point} of $\bE$) such that
the vector  $(m_1,\dots,m_n)$ formed by the coefficients of the $E_i$ in
$\bE$  are almost parallel to the integral weight decomposition
 $W(p/q): = q \bw(p/q)$, while a suitable multiple of the vector given by the other  coefficients
$(\Tm_1, \dots,\Tm_N)$ of $\bE$  is a close approximation to the (possibly infinite) set of negative weights $(b_1,b_2,\dots)$.
Moreover the difference $ \mu_{\bE}(z) - V_{(b;(b_j))}(z)$ is a local maximum when $z=p/q$, and the graph of $ \mu_{\bE}(z)$ (which is piecewise linear)
 has a convex corner at $a=p/q$.  The following upper bound for the value of $ \mu_{\bE}(p/q) $  is proved in \cite[Lem.2.28]{AADT}:
\begin{align}\label{eq:light} 
 \mu_{\bE}(p/q)  \le \sqrt\frac{p/q}{\Vol(b;(b_j))}\; \Bigl(\frac{\sqrt{b^2-\sum_j b_j^2}}{
 \sqrt{b^2\frac{d^2}{d^2+1}-\sum_j  b_j^2}}\Bigr)
\end{align}
We will say that $\bE$ is {\bf live at} $z$   if $$
c_{(b;(b_j))}(z) = \mu_\bE(z)>  \sqrt\frac{z}{\Vol(b;(b_j))} = V_{(b;(b_j))}(z).
$$  By definition, a quasi-exceptional  class $\bE$ is obstructive at its break point, but we cannot assume that it is  live there.  Note that quasi-exceptional classes that are not exceptional are usually not live.
Many of our results do not depend on the fact that the classes considered define live obstructions.  A sequence of quasi-exceptional classes $(\bE_k)_{k\ge 1}$ will be said to form a {\bf pre-staircase} in $X_{(b;(b_j))}$ if these classes are obstructive and if their break points $p_k/q_k$ form a convergent sequence.

\begin{lemma}\label{lem:basicE} \begin{itemize}\item[{\rm(i)}]
Each quasi-exceptional class has only a finite number of break points.
\item[{\rm(ii)}]  There are only finitely many ordered  quasi-exceptional classes of 
degree  $\le d$.
\end{itemize}
\end{lemma}
\begin{proof}  Let $\bE = dL - \sum_{j=1}^N \Tm_j \TE_j - \sum_{i=1}^n m_i E_i$, where each $m_i\ne 0$.

As explained above, each break point for $\bE$ has integral  weight decomposition  $W(p/q) = (q:=q_1,q_2,\dots,1: = q_n)$ of length $n$; but there are only  finitely many rational numbers of bounded weight length.  This proves (i).   (ii) holds because the integers $\Tm_j, m_i$ are nonnegative and, because $c_1(\bE) = 1$, satisfy the identity $\sum_j \Tm_j +\sum_i m_i = 3d-1$.
\end{proof}

Following \cite[(2.19)]{AADT}, we write 
\begin{align}\label{eq:definition}
\la_a: =\sqrt{\frac{a}{\Vol(b;(b_j))}}, \qquad
\bw&: =(\la_a b_1, \la_a b_2,\dots; a_1,\dots,a_n) \\ \notag
\bbm_\bE= (\widetilde{\bbm}, \bbm):  & = (\Tm_1,\dots, \Tm_N, 0, \dots;  m_1,\dots, m_n)
\end{align} 
Here, the vector $\bw: = (w_\nu)_{\nu\ge 1}$  has infinitely many entries if there are infinitely many negative weights $b_j$, and we extend $\bbm_\bE$ by adding zeros to the tuple $(\Tm_j)$ as  necessary so that its entries match those  of  $\bw$. Note that  $\bw$  depends on the target, while $\bbm_\bE$ depends on the obstructive class $\bE$.
\MS

\subsection{Proof of Theorem~\ref{thm:acc}}

We  will prove the following version of Theorem~\ref{thm:acc}. It is more general since it applies to any sequence of obstructive classes, not only to those that affect the capacity function.

Here, we let $\Vol: = \Vol(b;(b_j)):=b^2-\sum b_j^2$ and $\Per: = \Per(b;(b_j)):=3b-\sum b_j$ be the volume and perimeter corresponding to both $X_\Omega$ or $M_{\mathcal{B}}$. 

\begin{prop}\label{prop:accum1} \begin{itemize}\item [{\rm (i)}] Let $(b; (b_j))$ be the weights corresponding to $X_\Omega$ or $M_{\mathcal{B}}$, 
and $\bE_k, k\ge 1, $ an infinite sequence of obstructive quasi-exceptional classes with distinct break points $p_k/q_k$. If $\Per\ge 2\sqrt{\Vol}$, 
then $p_k/q_k$ 
 must converge to the unique  solution $a_0\ge 1$  of the  equation
\begin{align}\label{eq:acc} z^2 - \Bigl(\frac{\Per(b;(b_j))^2}{\Vol(b; (b_j))} - 2\Bigr)z + 1. 
\end{align} Additionally, if $\Per<2\sqrt{\Vol}$, such a sequence cannot exist. 

\item [{\rm (ii)}]   If the capacity function $c_{(b;(b_j))}$ has infinitely many nonsmooth points, then these nonsmooth points converge to 
the accumulation point $a_0$ and
$$
c_{(b;(b_j))}(a_0) = \sqrt{\frac{a_0}{\Vol(b;(b_j))}} =:V_{(b;(b_j))}(a_0).
$$
\end{itemize}
\end{prop}

The point $a_0$ in the above theorem is called the {\bf accumulation point} of $X_\Omega$ or $M_{\mathcal{B}}$, and in the case of $X_\Om$  is often denoted as $a_0^\Om$.

\begin{rmk}\label{rmk:perfectObs}\rm In all  cases that have been fully calculated,
 if the capacity function $c_{(b;(b_j))}$ has infinitely many  nonsmooth points, infinitely many of these points are local maxima at the break points of these classes.    This happens because the obstructing classes $(\bE_k)_{k\ge 1}$ that form the staircase are perfect (that is for each $k$ the entries $(m_{k1}, \dots, m_{kn_k})$ of $\bbm_{\bE_k}$ form the integral weight decomposition of the corresponding break point $p_k/q_k$), and are live at their break points.  The corresponding obstruction is given in a neighborhood of $p_k/q_k$ by a function whose graph for $z<p_k/q_k$ is a line  through the origin and for $z>p_k/q_k$ is horizontal (see \cite[Lem.16]{BHM} for example),  so that the $p_k/q_k$ are convex outer corners. 

 Thus  these outer corners are visible in the capacity function.  However, in many cases of descending stairs there is another obstruction that obscures the intersection points between the obstruction functions  of adjacent steps:  see for example, \cite[Example 32, Fig.5.3.1]{BHM} or \cite[Rmk.1.2.7]{MMW}.   The examples in \S\ref{sec:stair} below are less explicit: we simply show that there are infinitely many different
obstructive classes that are live somewhere, and therefore must form a staircase.  Our  proof of 
part (ii) of Proposition~\ref{prop:accum1}  does not show that when there is a staircase  the break points of the obstructing classes must give visible peaks, even though this seems very plausible. 
\end{rmk}
\MS

\begin{cor}\label{cor:accum1}  For each  $\Omega(b;(b_j))$ or $M_{(b;(b_j))}$, there is a constant  $N: = N^{(b;(b_j))}$ such that no obstructive class has break point $> N$.
\end{cor}
\begin{proof} If not, we could find a sequence of obstructive classes  whose break points diverge to $\infty$, which contradicts Proposition~\ref{prop:accum1}.
\end{proof}

\begin{cor}\label{cor:accum2} {\bf (Ellipsoidal Packing Stability)}  For each  $\Omega(b;(b_j))$ or well-defined blowup vector $(b;(b_j))$, there is a constant $a^{(b;(b_j))}_{\max}$ such that 
$c_{(b;(b_j))}(z) = V_{(b;(b_j))}(z)$ for all $z\ge a^{(b;(b_j))}_{\max}$.
\end{cor}
\begin{proof}  We know from \cite[2.30]{AADT} that, for each obstructive class $\bE$, each interval  $I_{\bE}$
on which $\mu_\bE(z)> V_{(b;(b_j))}(z)$
contains a unique point $a_\bE$ (called the break point)  whose continued fraction expansion $(w_1,w_2,\dots, w_n)$ has strictly shorter length than that of every other element of  $I_{\bE}$. Therefore,
by Corollary~\ref{cor:accum1}, it suffices to show that  for every obstructive class $\bE$ 
there is $R>0$ such that the upper bound of $I_\bE$ is at most  $a_\bE + R$.

If $a_\bE$ is not an integer then the weight expansion of $k: = \lceil  a_\bE \rceil$ is $(1^{\times k})$ and has strictly shorter weight expansion than $a_\bE$.  
Therefore  in this case it suffices to take $R\ge 1$.   On the other hand, if $a_\bE = k$ is an integer, then the  weight expansion of every $z>a_\bE$ begins with $1^{\times k}$ 
which implies that $\mu_{\bE}(z) = \frac{\sum_{i=1}^k m_i}{db-\sum \Tm_j b_j}$ is constant for $z> a_\bE.$ Since $V_{(b;(b_j))}(z)\to \infty$ as $z\to \infty$, the interval
 $I_\bE$  must have finite length.  Since there are only a finite number of such break points, we may take $R$ to be the supremum of $1$ and these lengths.
\end{proof}

\begin{rmk}\rm The function $\Om\mapsto a^\Om_{\max}$, where  $a^\Om_{\max}$  is the minimal $z$-value such that ellipsoidal packing stability holds, exhibits rather interesting behavior.   For a partial calculation in the case of the  family of polydiscs $X_{\Om_{[0,1]\times [0,s]}}$ see \cite{JinL}.
\end{rmk}

\begin{cor}\label{cor:accum3} 
Let  $\bE_k = \bigl(d_k; (\Tm_{kj})_{j\ge 1}; (m_{ki})_{i\ge 1}\bigr), k\ge 1$ be a sequence of obstructive classes in
 $\Om = \Om(b; (b_j))$ with distinct break points.
  Then for all $j$ for which $b_j\ne 0$ there is $k_j$ such that $\Tm_{kj}\ne 0$ for all $k\ge k_j$.
In particular, if there are infinitely many $b_j$, the number $N_k$ of nonzero entries $\Tm_{kj}$ in $\bE_k$ tends to infinity.
\end{cor}
\begin{proof} If not, by passing to a subsequence, we may suppose that there are integers  $j_0, n$ such that $\Tm_{kj_0} = 0$ for all $k\ge n$.   Then by \eqref{eq:mu}, the obstruction $\mu_{\bE_k}(p/q)$  is the same for both $\Om$ and $\Om_{j_0}: = (b; (b_j)_{j\ne j_0})$, while $\Vol(\Om_{j_0}) > \Vol(\Om)$. 
 By assumption the classes $\bE_k$ form a pre-staircase for $\Om$, which implies by Proposition~\ref{prop:accum1}~(ii) that, if $\bE_k$ has break point $p_k/q_k$ then $\lim_k p_k/q_k = a_0$ and
 $$
 \lim_{k\to \infty} \mu_{\bE_k}(p_k/q_k) = \sqrt{\frac{a_0}{\Vol(\Om)}}.
 $$
Since $\Vol(\Om_{j_0})> \Vol(\Om)$, the volume constraint  for $\Om_{j_0}$ is smaller than that for $\Om$. Hence the classes
$(\bE_k)_{k\ge 1}$ also form a pre-staircase for $X_{\Om_{j_0}}$  and
$ \lim_{k\to \infty} \mu_{\bE_k}(p_k/q_k) = \sqrt{\frac{a_0}{\Vol(\Om_{j_0})}}$.  But this is impossible.
\end{proof}
Finally, here is a consequence in the closed case.

\begin{cor}\label{cor:curvature}
Let $M$ be a rational symplectic $4$-manifold with $c_1(\omega) \cdot [\omega] \le 0$.  Then $M$ does not admit an infinite staircase.
\end{cor}

\begin{proof} This is an immediate consequence of Proposition~\ref{prop:accum1}~(i).
\end{proof}

As mentioned in Remark~\ref{rmk:curvature}, the quantity $\Per(M_{\mathcal{B}})=c_1 \cdot [\omega]$ is a classical topological invariant of symplectic $4$-manifolds that can be interpreted  (up to a universal constant) as the {\bf total scalar curvature}, i.e. the integral of the Hermitian curvature of any compatible metric. 
Note also that, by a theorem of Taubes \cite{Taubes}, every closed symplectic $4$-manifold with $b^+_2 \ge 2$ has nonpositive total scalar curvature.  Of course such manifolds are never blowups of $\mathbb{C}P^2$, but one could perhaps interpret the previous theorem as evidence that these never have infinite staircases.  For other evidence, see Entov--Verbitsky~\cite{EV}.

\begin{proof}[Proof of Proposition~\ref{prop:accum1}]
Let $a=p/q$ be the break point of an obstructive  quasi-exceptional class $\bE$.  Using the notation in \eqref{eq:definition}, we define its {\bf error vector} $\eps$ by 
\begin{align}\label{eq:eps}
\bbm_{\bE} = \frac{d}{\la_a b} \bw + \eps.
\end{align} 
A straightforward calculation using \eqref{eq:mu} and \eqref{eq:eps} shows that
\begin{align}\label{eq:key1}
\mu_{\bE}(a) > \la_a \;\Longleftrightarrow \; \eps\cdot \bw > 0.
\end{align}
As in the proof of \cite[(4.6)]{AADT},  the identity $-1 = \bE\cdot\bE =  d^2  -\bbm_{\bE}\cdot \bbm_{\bE} $ 
and the fact that $\bw\cdot \bw = \la_a^2 b^2$ readily imply that if $\bE$ is obstructive at its break point, i.e. if $\mu_\bE(a) > \la_a$, then
 we must have
\begin{align}\label{eq:eps2}
\eps\cdot \eps< 1.
\end{align}
Even though the vector $\bw =:(w_\nu)_{\nu\ge1}$  has infinite length,\footnote
{
For clarity we restrict to the case when $\Om$ has infinitely many negative weights $(b_j)$ since the finite case is dealt with in \cite{AADT}.} its entries have finite sum since $\sum _j b_j < b$.  Therefore, since $\sum_j \Tm_j + {\textstyle \sum_{i=1}^n} m_i = 3d-1$ (because $c_1(\bE) = 1$),  we can estimate
\begin{align} \notag
\bigl|-1-{\textstyle \sum_\nu} \eps_\nu\bigr|  &  = \bigl| -1+ \frac d{\la_ab} ({\textstyle \sum_\nu} w_\nu) - (3d-1) \bigr|\\ \notag
&=  \frac d{\la_ab}\  \bigl| ({\textstyle \sum_\nu} w_\nu) - 3\la_a b \bigr|\\ \notag
& = \frac d{\la_ab} \bigl| a+1-\la_a (3b- {\textstyle \sum_j} b_j) -  1/q\bigr|\\ \notag
&= \frac d{\la_ab} \bigl| a+1- \Per(b;(b_j))\sqrt {\frac{a}{\Vol(b;(b_j))}} -  1/q\bigr|\
\\ 
&= \frac d{\la_ab}\bigl| f(a) - 1/q\bigr|
\end{align}
where  $f(z)$ is the function
\begin{align}\label{eq:f}
f(z): =  z+1- \Per(b;(b_j)) \sqrt { \frac{z}{\Vol(b;(b_j))}}.
\end{align}
Note that the
the third equality above uses 
$$
{\textstyle \sum_\nu }w_\nu = \la_a{\textstyle \sum_j} b_j + {\textstyle \sum_{i=1}^n} a_i =  \la_a{\textstyle \sum_j} b_j  +  (a + 1 -  1/q).
$$

With $f(z)$ as above, the equation $f(z)=0$ has the same roots as the accumulation equation \eqref{eq:acc} when $\Per(b;(b_j))\ge0$: to see this just multiply $f(z)$ by its conjugate $z+1 + \sqrt {z \frac{\Per(b;(b_j))^2}{\Vol(b;(b_j))}} $. When $\Per(b;(b_j))<2\sqrt{\Vol(b;(b_j))}$, $f(z)$ has no real roots. Thus our aim is to 
find estimates for $|\sum \eps_\nu|$ that imply that $f(a_n)$ must converge to zero if  $\bE_n$ is a sequence of obstructive classes with distinct breakpoints $a_n$.

To this end, write $\eps= \eps_1+\eps_2+\eps_3$, where the $\eps_i$ are defined as follows. The first vector $\eps_1$ consists of the first $3d$
entries of $\eps$, with zeros at all other entries. The second vector $\eps_2$ 
(which may be infinitely long) consists of the entries of $\eps$ corresponding to the other $b_j$, with zeros at all other entries. The final vector $\eps_3$ consists of the entries of $\eps_i$ corresponding to the $a_i$.
We will bound $ |\eps\cdot (1, . . . , 1)|: = |\sum_\nu \eps_\nu|$ by
bounding the dot product for  each of the three vectors $\eps_i$.

For $|\eps_1 \cdot (1, . . . , 1)|$, we have by Cauchy-Schwarz that
\begin{align}\label{eq:eps11}
|\eps_1 \cdot (1, . . . , 1)|\le \sqrt{3d}.
\end{align}

For $|\eps_2\cdot(1,...,1)|$, we note first that the nonzero entries of
$\eps_2$ are entries in $ -\frac{d}{\la_ab} \bw$. This is a consequence of the 
fact that $\bbm_{\bE}$ has at most $3d - 1$ nonzero entries, 
since the class $\bE$ must have $ 1=c_1(\bE)  = 3d - \sum \Tm_j - \sum m_i$.
 On the other hand,  $\sum_j b_j$  converges, and so we obtain the key estimate:
\begin{align}\label{eq:eps12}
|\eps_2 \cdot(1,...,1)| = \sum_{j> 3d} \frac db b_j \le \ka_d d,
\end{align}
where $\ka_d > 0$ is a constant that tends to $0$ as $d\to \infty$. This crucial fact holds because $\sum_j b_j$ converges.  In particular, the constant $\ka_d$ depends only on $(b; (b_j))$, and not on $\bE$.

For $\eps_3$, we have by Cauchy--Schwartz that  $|\eps_3\cdot(1,...,1)|<\sqrt  L$, where $L$ is the weight length of $a$, i.e. the number $n$ of nonzero entries in    
the weight expansion $(a_1,\dots,a_n)$ of $p/q$.
To find an appropriate bound, notice first that
since $\eps\cdot\eps< 1$ by \eqref{eq:eps2}, we do know that each  entry of $\eps$ has norm $<1$.
By definition, the weight decomposition $(a_1,\dots,a_n)$  of $p/q$ has last two entries $1/q< 1$ so that the fact that  $m_{n-1} \ge m_n \ge 1$ implies that the last two entries in  $\eps$  are at least $1 - \frac d{\la_a bq}$.
As in \cite[(4.5)]{AADT}, this easily implies that 
\begin{align}\label{eq:eps3}
\frac d{\la_a b q} > 1/4, \quad \mbox{ where }\;\; a = p/q.
\end{align}
Next observe that if $k< a=p/q  < k+1$, then $L$ is largest if $a$ has weight expansion $[1^{\times k}, (\frac 1q)^{\times (q-1)}]$, which shows that  $L< a + q$. Together with \eqref{eq:eps3} this implies that  
\begin{align}\label{eq:eps4}
|\eps_3\cdot(1,...,1)| \le \sqrt{L} \le\sqrt{ a + \frac {4d} {\la_a b}}.
\end{align}
Putting this all together, we obtain the following variant of the key inequality [Eq. 4.8, AADT],:
\begin{align}\label{eq:eps5}
\frac{d}{\la_ab}\bigl|f(a) -\frac 1q\bigr|& \le 1 + \sqrt{3d} + \sqrt{a + 4\frac{d}{\la_ab}} + \ka_d d \\ \notag
&\le  1 + \sqrt{3d} + \sqrt{a} + 2\sqrt{\frac{d}{\la_a b}} + \ka_d d.
\end{align}

Now observe that by Lemma~\ref{lem:basicE}~(ii), any sequence of distinct  obstructive classes $\bE_k$ must have degrees $d_k$ that tend to infinity.
We now show that 
the corresponding sequence of 
break points $a_k$ must be bounded.   For if not, since $f(a_k)> a_k/2$ for large $a_k$ by \eqref{eq:f},  there is $C>0$ such that
 $f(a_k)/\la_{a_k} >  C \sqrt a_k$ for large $k$, so that the left hand side of \eqref{eq:eps5} is $\ge C' d_k \sqrt{a_k}$. But then it is not bounded by the terms on the right hand side of   
 \eqref{eq:eps5} as $k\to \infty$.
We conclude that 
 there is $M$ such that
 $a_k\le M$ for all $k$.

 Thus we have
 \begin{align*}
\frac{d_k}{\la_{a_k}b}\bigl|f(a_k) -\frac 1{q_k}\bigr|
&\le  1 + \sqrt{3d_k} + \sqrt{M} + 2\sqrt{\frac{d_k}{\la_{a_k} b}} + \ka_{d_k} d_k.
\end{align*}
But this inequality implies that $a_k\to a_0$.  For otherwise, we may pass to a subsequence such that the  $a_k$ are bounded away from $a_0$ so that $f(a_k)$ is bounded away from $0$.  Since the $a_k = p_k/q_k \le M$ are distinct (by hypothesis),  we must have $1/q_k \to 0$, so that the left hand side is $\approx C\ d_k$ for some constant $C>0$.
But this is impossible 
because $\ka_{d_k}\to 0 $ as $d_k\to \infty$.
This completes the proof of (i).
\MS

It remains to show that the existence of infinitely many nonsmooth points $(z_k, y_k)$ of $c_{(b;(b_j))}$ implies that
the points $z_k$ converge to the accumulation point $a^{(b;(b_j))}_0$ and that
$$
c_{(b;(b_j))}(a^{(b;(b_j))}_0) = V_{(b;(b_j))}(a^{(b;(b_j))}_0) =\sqrt {\frac{a^{(b;(b_j))}_0}{\Vol(b;(b_j))}}.
$$
  To prove this, we  argue as in Steps 0 and 4 of the proof of Theorem~1.13 in \cite[\S4]{AADT}. For the remainder of the proof, we set $a_0:=a^{(b;(b_j))}_0$.

  It suffices to show that every sequence of such points has a subsequence that converges to $(a_0, V_{(b;(b_j))}(a_0))$.  
 By Lemma~\ref{lem:basicE}, we may pass to a subsequence so
 that  $(z_k, y_k)$  lies on the graph of $\mu_{\bE_k}$, where the $\bE_k$ are distinct with breakpoints $a_k$  and have degrees $d_k$ that diverge to $\infty$. 
 By passing to a further subsequence 
  we may assume that  the sequence 
   $(a_k)$ is monotonic with limit $a_\infty$, and that, if $(a_k)$ is not constant, then by (i), $a_\infty = a_0$.
    Further, since, as explained in the proof of Corollary~\ref{cor:accum2},  the interval between $z_k$ and $a_k$ consists of points whose weight length is longer than that of $a_k$,
   the distance $|a_k-z_k|$ is bounded. Therefore 
 we may assume that the $z_k$ also converge, with limit $z_\infty$.  Since the points $(z_k,y_k)$ lie on the graph of $c_{(b;(b_j))}$, and the degree $d_k\to \infty$, it follows from \eqref{eq:light} that $c_{(b;(b_j))}(z_\infty) = V_{(b;(b_j))}(z_\infty)$. In the case where $a_\infty=a_0$, we are finished. It remains to consider the case where $a_k$ is constant and $a_k=a_\infty$.
 
  In this case, their constant value $a_\infty$ is obstructed, so by the argument above is not $z_\infty$.  Therefore, it suffices to show that  this is not possible. 
 
 If $a_\infty\ne z_\infty$, because $\bE_k$ is obstructive between $z_k$ and $a_k$ by construction, 
  the line segment between $(z_k,\mu_{\bE_k}(z_k))$ and  $(a_k,\mu_{\bE_k}(a_k))$ lies above the volume constraint, and converges to
  the line segment between  $(z_\infty,\mu_{\bE_k}(z_\infty)) = (z_\infty, V_{(b;(b_j))}(z_\infty))$ and $(a_\infty,\mu_{\bE_k}(a_\infty))$.
  Further, since $d_k\to \infty$ the values $\mu_{\bE_k}(a_k)$ converge to $V_{(b;(b_j))}(a_\infty)$.
  Therefore the end points of the limiting line segment both  lie on the graph of $V_{(b;(b_j))}$.
  Since this graph is concave down, and the end points are distinct, this is impossible.
\end{proof}

\begin{cor}\label{cor:ellip} If $b$ is irrational, the ellipsoid $E(1,b)$ has no staircase.
\end{cor}

\begin{proof}
In this case, $\Vol(E(1,b)) = b$ while $\Per(E(1,b))= 1+b$, so that the
 accumulation point  is $b$.   We may 
compute the embedding function $c_{E(1,b)}$ for 
$$
   \lfloor{b}\rfloor \le a \le \lceil{b}\rceil
$$
as follows.
By Gromov's non-squeezing theorem, we have
$$
   c_{E(1,b)}(a) = 1 \;\;\mbox{ for }\;\;  1 \le a \le b.
$$
Further, because  $E(1,\la b)\se E(\la,\la b) = \la E(1,b)$ for 
all $\la > 1$, there are embeddings $E(1,a)\se \frac ab E(1,b)$ for all $a>b$ so that
$$
   c_{E(1,b)}(a) \le a/b \;\;\mbox{ for }\;\; b \le a \le  \lceil{b}\rceil.
$$
(This is known as the  \lq\lq subscaling\rq\rq\, property of $c_{E(1,b)}$.)
On the other hand, if we compute the ECH capacity\footnote
{
Recall that the ECH capacities $c^{ECH}_k$ of the ellipsoid $E(a,b)$ are the numbers $\{ka + \ell b \ | \ k,\ell\ge 0\}$ arranged (with multiplicities) in nondecreasing order.}
$c^{ECH}_k(E(1,b)) $ for $k =  \lceil{b}\rceil$, we have $c^{ECH}_k(E(1,b)) = b$ always, and  
$$
     c^{ECH}_k(E(1,a)) = a, \;\;\mbox{ for }\;\; b \le a \le  \lceil{b}\rceil.
$$
It therefore follows that $c_{E(1,b)}(a) \ge a/b, b \le a \le  \lceil{b}\rceil$.  Thus, since these bounds agree, 
$$
  c_{E(1,b)}(a) = a/b ,\;\;\mbox{ for }\;\;  b \le a \le  \lceil{b}\rceil.
$$
Since we have now computed $c_{E(1,b)}(a)$ in a neighborhood of $b$, it is therefore clear there is no infinite staircase accumulating at $b$.
\end{proof}

\begin{rmk}\rm  We show in Example~\ref{ex:E'}, that instead of using ECH capacities to find a lower bound for $c_{E(1,b)}(a)$ on the interval $[b, \lceil{b}\rceil]$, one can use the obstruction coming from an appropriate  perfect class $\bE$.
Moreover, we  will see  in \S\ref{ss:ghost} that given any irrational ellipsoid $E(1,b)$ there are infinitely many obstructive classes whose breakpoints $p_k/q_k$ (necessarily) converge to $b$. However these classes are overshadowed by the obstruction coming from this class $\bE$. Thus, in the language of \cite{MMW},  these obstructive classes form a prestaircase but not a staircase. This example shows that one should expect many domains $X_\Om$ to have infinitely many obstructive classes, even though $X_\Om$ may not have a staircase.
\end{rmk}

\section{The subleading asymptotics of ECH capacities}\label{sec:subECH}

Let us review some notation and context from the introduction.  It was shown in \cite{ECH_volume, Hutchq} that  the ECH capacities of a closed $4$-manifold $X$ detect the volume of $X$ in the following sense
\begin{align}\label{eq:ek0}
\lim_{k\to \infty} \frac{c_k^2(X)}{2k} =  \Vol(X),\;\;\quad \mbox{ where }\  \Vol(X): =  b^2 - \sum b_i^2.
\end{align}
 We define the {\bf subleading asymptotics} $e_k(X)$ 
by setting
\begin{align}\label{eq:ek}
e_k(X): = c_k(X) - \sqrt{ 2k \Vol(X) }
\end{align}

The purpose of this section is to prove Theorem~\ref{thm:convexper}, which shows that when $X$ is a toric domain the  subleading asymptotics of the $c_k$  detect  the perimeter of $X$, and then derive its consequences, Theorem~\ref{thm:packing},  Corollary~\ref{cor:new} and Corollary~\ref{cor:super}.  

\subsection{The main result}\label{ss:asympt}

In this section, we prove Theorem~\ref{thm:convexper}.

We first check this theorem for the ball. 

\begin{lemma}\label{lem:Perball}  Theorem~\ref{thm:convexper}  holds when $X = B(1)$ is the ball. Moreover, 
$\limsup_k e_k(B(1)) = -1/2$.
\end{lemma}
\begin{proof} Recall that the ECH capacities of the unit ball $B(1)$ are given by
$$
c_k(B(1)) = n\qquad \mbox{ if } \;\;\; \frac{n^2+n}2\le k \le \frac{n^2+3n}2.
$$
Thus $n : = c_k(B(1))\approx  \sqrt {2k}$  which, because $\Vol(B(1)) = 1$, is consistent with \eqref{eq:ek0}. The difference 
$e_k = c_k(X) - \sqrt{ 2k \Vol(X) }$ is most negative when $k = 
(n^2+3n)/2$ for some $n$ and most positive when $k = 
(n^2+n)/2$.  From this it is easy to check that $\liminf_k e_k(B(1)) = -3/2 = -\Per/2$, while
$\limsup_k e_k(B(1)) = -1/2$.
\end{proof}

For a domain of the form $\Om(b;(b_j))$ as the $\Per(\Om)=3b-\sum b_j,$
 the ball $B(b)$ is the only domain of the form $\Om(b, (b_j))$ with $\Per(X) = 3b$. Hence, we  assume from now on that $\Per(X)< 3b$, a fact that is used in the definition of $\eps_2$ in \eqref{eqn:othereps}.  
We will prove Theorem~\ref{thm:convexper}
by first establishing that $\liminf_{k} e_k(X) \le - \Per(X)/2$ and then that 
$\liminf_{k} e_k(X) \ge - \Per(X)/2$.
The following warm-up lemmas will be useful. 

\begin{lemma}\label{lem:warmup}
Fix $(b_1,\ldots,b_n)$, a finite set of real numbers, and a parameter $\eps> 0$.   Then, there are infinitely many positive integers $m$ such that $mb_j$ is within $\eps$ of an integer for all $j$.
\end{lemma}

\begin{proof}
Fix a $\mathbb{Q}$-basis $(1,a_1,\ldots,a_s)$ for $\Q\langle 1,b_1,\ldots,b_n\rangle.$  
Write each 
\[ b_j =\frac{ p_{0j}}{q_{0j}} + \sum_{i=0}^s \frac{ p_{ij}}{q_{ij}} a_i,\]
and define $T_{0}: = \prod_{i\ge 0,j\ge 1} q_{ij}$. 
Now consider $m = k T_{0}$, where $k$ ranges over the positive integers.  Then for each $1\le j\le n$, 
\[ k T_{0} b_j = k \frac{T_{0}}{q_{0j}} p_{0j} + \sum_i \frac{T_{0}}{q_{ij}} p_{ij} k a_i.\]
Thus, if $M: = \max_{i,j} \frac{T_{0}}{q_{ij}} p_{ij}$ and we choose $k$ so that 
 each $k a_i$ is within $\eps/M$ of an integer, the quantity 
 $k T_{0} b_j$ is within $\eps$ of an integer. There are infinitely many such $k$  because, by 
 Kronecker's Approximation Theorem~\footnote
{
See for example the discussion by Keith Conrad at https://mathoverflow.net/questions/18174}, \
 the set of fractional parts of $(k a_1, \ldots, k a_s)_{k\ge 1}$ is dense in the torus.  
\end{proof}

\begin{lemma}\label{lem:warmup2} Let $b_j$ be any  sequence of positive numbers with $\sum b_j <\infty$, and for each integer $m>0$ 
define $S_m : = \{j \ | \ b_j>1/(2m)\}$.
Then, for any $\eps> 0$, there are only finitely many integers $m$  such that $|S_m| 
 \ge m\eps.$
\end{lemma}
\begin{proof}
We argue by contradiction, showing that if the conclusion is false, then $\sum b_j$ must diverge. More precisely, 
this assumption allows us to find infinitely many disjoint subsets of the index set for the $b_j$, whose corresponding sums
each contribute at least a uniform amount to the sum, forcing divergence.

To implement this, assume that the claim is false for some $\eps>0$, and choose an increasing sequence $m_0, m_1,m_2,\dots $ of integers with 
\[ (m_1-m_0)/m_1 \ge 1/2, \quad 1/m_1 < \eps/4,\] 
\[ (m_2-m_0-m_1)/m_2 \ge 1/2, \quad 2/m_2 <  \eps/4, \ldots \] 
\[ (m_p - m_0 - \ldots - m_{p-1})/m_p \ge 1/2, \quad p/m_p < \eps/4 \]
and such that for each $m_p$, we have 
$|S_{m_p}| \ge m_p \eps.$  
Now choose exactly $\lceil m_0 \eps\rceil$ indices of the elements $b_j \in S_{m_0}$, such that the corresponding contribution of these indices to $\sum b_j$ is at least $\eps/2$.   Next, choose exactly 
$\lceil m_1 \eps\rceil - \lceil m_0\eps\rceil$ 
indices of the $b_j$ in $S_{m_1}$, with different indices from those previously chosen for  $S_{m_0}$, which must contribute at least 
$(m_1 - m_0) \eps/ (2m_1) - 1/(2m_1) \ge \eps/8$
 towards $\sum_j b_j$.  Continuing in this manner -- i.e. choosing exactly $\lceil  m_2 \eps\rceil - \lceil  m_1 \eps\rceil - \lceil m_0 \eps\rceil$ of the indices of $b_i$ in $S_{m_2}$ that are different from those we picked for $S_{m_1}$ and $S_{m_0}$, etc., --- produces a contribution towards of $\sum_j b_j$ at least $\eps/8$ for each subset, forcing the sum to diverge.
\end{proof}

\begin{lemma} \label{lem:ECHle} For any convex toric domain $X$,
we have
$\liminf_{k} e_k(X) \le - \Per(X)/2.$
\end{lemma}

\begin{proof}
We give the proof in several steps. We first define a sequence $k$ where we expect $e_k(X) $ to be small, and then estimate $e_k(X) $
for these values of $k$.  \MS

\NI
{\bf Step 1:} {\it Choosing  $k$.} 
We assume without loss of generality that $b = 1$ and, by Lemma~\ref{lem:Perball} that $\Per < 3$. We fix  a parameter $\eps> 0$, and, to make more readable formulas in what follows, 
we define the auxiliary parameters $\eps_1, \eps_2$ where
\begin{equation}
\label{eqn:othereps}
0<\eps_1 < \eps/32, \quad 0<\eps_2 < \eps/(16 (3-\Per)).
\end{equation}
 Now choose $M$ such that $\sum_{j > M} b_j \le \eps_1.$  (This is possible since $\sum b_j$ converges.)
 Then apply Lemma~\ref{lem:warmup}  to $(b_1,\ldots,b_M)$, with parameter $\eps_2$ , and let $\Mm: = \Mm(M,\eps_1)$ be the infinite set of integers  guaranteed by this lemma.  Fix $m\in \Mm$, and, for
all $j\ge 1$, define $k_j: = [ m b_j]$, where $[\;\; ]$ denotes the closest integer:
 if there are two equally close integers we choose the smaller one.  
 Thus
 $$k_j = m b_j + \delta_j, \quad \mbox{ where } \ -1/2 < \delta_j \le 1/2.
 $$
Since $m b_j\to 0$ as $j\to \infty$, there are only finitely many  nonzero $k_j.$  
 For each $m\in \Mm$, we now consider $k = k(m)$ defined by 
\begin{equation}
\label{eqn:defnk}
2k: = m^2 + 3m - \sum_j  (k_j^2 + k_j).
\end{equation}

\NI
{\bf Step 2:} {\it For sufficiently large $m\in \Mm$, we have}
\begin{align} \label{eqn:easy} \sum_j |\delta_j| b_j&\le \eps/8
\\  \label{eqn:slightlydelicate} 
\sum_j\bigl( |\delta_j| + \delta_j^2\bigr) &< m \eps/8.
\end{align}

\begin{proof} Above, we chose $M$ so that $\sum_{j>M} b_j < \eps_1$, and then chose $m\in \Mm$ to be one of the sequence of numbers
with  $|mb - k_j| =:|\de_j|< \eps_2$ for all $j\le M$.
Therefore 
\begin{align*}
\sum_j |\delta_j| b_j & = \sum_{j\le M} |\delta_j| b_j + \sum_{j> M} |\delta_j| b_j \\
& \le \eps_2 \sum_j b_j + \sum_{j>M} b_j \\
& < \eps_2 (3-\Per) + \eps_1\;  \le \; \eps/8,
\end{align*}
where the last inequality uses \eqref{eqn:othereps}.

Next consider \eqref{eqn:slightlydelicate}.  By applying Lemma~\ref{lem:warmup2} with $\eps$ replaced by $\eps/32$  we may conclude that
for sufficiently large $m$ there are at most $m \eps/32$ of the $b_j$ with $m b_j \ge 1/2$.  Each of these has $|\delta_j| \le 1$ and $\delta_j^2 \le 1$, so that their total contribution to the sum in \eqref{eqn:slightlydelicate} is no more than $m \eps/16$.  As for the remaining $b_j$, by definition these have $mb_j \le 1/2$, so that they have $\delta_i = - mb_i$ and $\de_i^2< mb_i$; thus, for sufficiently large $m$ their total contribution is bounded by $2 m \sum_{i > M} b_i \le 2 m \eps_1 \le m\eps/16,$ where in the final inequality we have again applied \eqref{eqn:othereps}.
\end{proof}

\NI {\bf Step 3.} {\it We estimate $e_k:=c_k - \sqrt{2k \Vol}$ for the sequence of  $k$ chosen in \eqref{eqn:defnk}.}

\NI {\it Proof.}\
To begin, we find formulas for $c_k$ and $2k \Vol$.
We
 saw in \eqref{eqn:echsub}  that
\begin{equation}
\label{eqn:echsub0}
c_k(X) = \min_{k = \ell - k_1 - \ldots - k_q}\ c_\ell B(b) - c_{k_1} B(b_1) - \ldots - c_{k_q} B(b_q).
\end{equation}
Therefore, if $k,m,k_j$ are as in \eqref{eqn:defnk}, and we use 
 the formula for the ECH capacities of a ball from Lemma~\ref{lem:Perball} we find that
\begin{align}
\label{eqn:upperbound}
c_k(X) &\le c_{(m^2 + 3m)/2}(B(1)) - \sum_j c_{ (k_j^2 + k_j)/2} B(b_j)\\ \notag
 & = m - \sum_j (m b_j+ \delta_j) b_j 
 = m(1 - \sum b_j^2) - \sum \delta_j b_j \\  \notag
 &= m\ \Vol - \sum \delta_j b_j.
\end{align}
Again using $k_j = m b_j+ \delta_j$, we can rewrite \eqref{eqn:defnk} as
\[ 
2k = m^2 + 3m - \bigl(m^2 \sum  b_j^2 + m \sum b_j + 2m \sum \delta_j b_j + \sum \delta_j + \sum \delta_j^2\bigr).\]
By multiplying this by $\Vol: = 1-\sum b_j^2$, using the identity $ \Per = 3- \sum_j b_j$ and then the estimates in Step 2,  we find that
\begin{align*}\notag 2k\Vol &= m^2 (\Vol)^2 + m (\Per)(\Vol) - 2\Vol \, m \sum \delta_j b_j - \Vol( \sum \delta_j + \delta_j^2)
\\ 
&= (m^2(\Vol)^2)\bigl[ 1 + \Per/(m\Vol) - 2/(m\Vol) \bigl(\sum \delta_j b_j\bigr)\\
&\notag\qquad\qquad  \qquad\qquad \qquad\qquad\qquad\qquad   - 1/(m^2\Vol)\cdot \bigl( \sum \delta_j + \delta_j^2\bigr) \bigr]\\ \notag
& \ge (m^2(\Vol)^2)\bigl[ 1 + \Per/(m\Vol) - 2/(m\Vol)\cdot ( \eps/8) - 1/(m\Vol)\cdot (\eps/8)\bigr] \\ \notag
& \ge (m^2(\Vol)^2)\bigl[ 1 + 1/(m\Vol)\cdot (\Per -  \eps/2)\bigr].
\end{align*}
Therefore, 
\begin{align*}
\sqrt{2k\ \Vol} &\ge m\Vol \sqrt {1 + \Per/(m\Vol) - \eps/(2m\Vol)}\\
&  \ge m\Vol\bigl(1 +  \Per/(2m\Vol) - \eps/(2m\Vol)\bigr)
\end{align*}
where the last inequality holds for fixed $\eps$ and sufficiently large $m\in \Mm$.
Therefore, using the upper bound for $c_k(X)$ in \eqref{eqn:upperbound} and the inequality
\eqref{eqn:easy}, we find that
\begin{align*} e_k  = c_k - \sqrt{(2k\ \Vol)} &\le \bigl(m\Vol  - \sum \de_j b_j \bigr) - m\Vol - \Per/2 + \eps/2   \\
& \le  - \Per/2 + \eps.
\end{align*}
Since $\eps$ is arbitrarily small,  the proof of Lemma~\ref{lem:ECHle}  is complete.
\end{proof}

We next prove the other inequality.

\begin{lemma} \label{lem:ECHge} For any convex toric domain $X$,
we have
\[\liminf_{k} e_k(X) \ge - \Per(X)/2.\]
\end{lemma}

\begin{proof}
We apply the sequence subtraction formula in Lemma~\ref{lem:ECHk}  to write
\begin{align}\label{eqn:ck}
 c_k(X) = c_\ell B(b) - c_{k_1}B(b_1) - \ldots - c_{k_r} B(b_r)
 \end{align}
for suitable $\ell, k_i$ such that 
\[ \ell - (k_1 + \ldots k_r) = k.\] 
Let us write 
\[ K = k_1 + \ldots + k_r, \quad 2V_1 = \sum b_i^2, \quad \Vol = b^2 - \sum b_i^2  \]
to simplify the notation.    

\vspace{2 mm}

First note that as $k$ tends to $\infty$, so do $\ell$ and $K$. This is clear for $\ell$ since $\ell\ge k $.  
On the other hand if $K$ remains bounded then the contribution of the balls $B(b_i)$ to the limit
$\lim_k \frac {c_k^2(X)}{k^2} = \Vol(X)$ tends to zero, which implies that $\Vol(X) = \Vol(B(b))$, a contradiction.

We now complete the proof by reducing to the concave case and applying a variant of Cauchy--Schwarz.
Fix $\eps> 0$.
By the formula for the ECH capacities of a ball given in Lemma~\ref{lem:Perball}, we can assume
 that 
 \begin{align}\label{eqn:elld}
  2 \ell = d^2 + 3d.
  \end{align} 
 for some $d$. Indeed if $d^2 + d <2 \ell < d^2 + 3d$ then $c_k(X) = c_{k+1}(X)$ which would imply that
$e_k(X) > e_{k+1}(X)$, so that it suffices to find a lower bound for  $e_{k+1}(X)$.

 In \cite[Lem.3.8]{Ruel}, Hutchings 
   showed  that for any union of balls $\sqcup B(b_i)$ 
 $$
 \limsup_k \bigl(c_k(\sqcup B(b_i)) -   \sqrt{(2K)(2V_1)}\bigr) \le - \frac 12 \sum b_i
 $$
 where $V_1,K$ are as above.  Therefore, if 
  $k$ is sufficiently large, so that, by the above,  $K = \sum k_i$ is also large, 
  we may conclude that   
\[  c_{k_1}B(b_1) + \ldots + c_{k_r} B(b_r) \le \sqrt{ (2 K ) (2V_1)} - \sum b_i/2 + \eps.\]

Thus, in view of \eqref{eqn:elld}, for any $\eps>0$, we have
\begin{equation}
\label{eqn1}
c_k(X) \ge bd - \sqrt{ (2K)(2V_1) } + \sum b_i/2 - \eps
\end{equation}
 sufficiently large $k$. 
On the other hand, by the variant of the Cauchy-Schwarz inequality for the Minkowski metric applied to the vectors $(b,\sqrt{2 V_1}), (\sqrt{d^2 + 3d}, \sqrt{2K})$, we have
\[ b \sqrt{d^2 + 3d} - \sqrt{ 2V_1} \sqrt{2 K} \ge \sqrt{ b^2 - 2 V_1}\sqrt{ d^2 + 3d - 2K} = \sqrt{\Vol} \sqrt{2k}.\]
Combining this equation with \eqref{eqn1} then gives
\[ c_k(X) - \sqrt{ (2k)\Vol } \ge b ( d - \sqrt{d^2 + 3d}) + \sum b_i/2 - \eps.\]
Since 
$ d - \sqrt{d^2 + 3d} \ge -3/2,$
we  conclude that $$
c_k(X) - \sqrt{ (2k)\Vol } \ge - \frac{3b - \sum b_i}{2} -  \eps= -\Per/2 - \eps.
$$
Since $\eps$ was arbitrary,   the result follows. 
\end{proof}

We can now give the promised proof.

\begin{proof}[Proof of Theorem~\ref{thm:convexper}]
This holds by combining Lemma~\ref{lem:ECHle} with Lemma~\ref{lem:ECHge}.
\end{proof}

\subsection{Consequences for full fillings}\label{ss:fullFilling}

Let us now prove Theorem~\ref{thm:packing}.  

Since Theorem~\ref{thm:packing} potentially involves disjoint unions, we start with the following helpful lemma.

\begin{lemma}\label{lem:sublead}
Let $e_k$ denote the subleading terms for either the ECH capacities or the elementary ECH capacities.  Let $X$ be the disjoint union of domains $Q_1, \ldots, Q_r$ 
Then  
\[ \liminf_k e_k(X) \ge \sum_i \liminf_k e_k(Q_i).\]
\end{lemma}

\begin{proof}

Fix k.  By the disjoint union property of elementary ECH capacities or ECH capacities, $c_k(X)$ is the max, over $k_i$ such that $k_1 + \ldots + k_r \leq k$, of $c_{k_1}(Q_1) + \ldots + c_{k_r}(Q_r).$  
We consider $k_i = \lfloor k (V_i / V)\rfloor$, where $V_i$ is the volume of $Q_i$ and $V$, the total volume, is the sum of the $V_i.$  Let $p_i$ be the liminf of $e_k(Q_i)$.  By our assumption on the liminf, for large enough k, 
\[ c_{k_1}(Q_1) + \ldots + c_{k_r}(Q_r) \ge \sqrt{ 4 k_1 V_1} + ... + \sqrt{4 k_r V_r} + p_1 + \ldots + p_r - \eps_1,\]
where $\eps_1 > 0$ is arbitrarily small.  On the other hand, by plugging in directly,
\[ \sqrt{ 4 k_1 V_1} + ... + \sqrt{4 k_r V_r} = \sqrt{ 4k V} - \eps_2,\]
with $\eps_2 > 0$ arbitrarily small for sufficiently large $k.$  
Since $k_1 + \ldots + k_r \le k,$ 
\[ c_k(X) \ge c_{k_1}(Q_1) + \ldots + c_{k_r}(Q_r),\] 
hence the result.
\end{proof}

We can now prove Theorem~\ref{thm:packing}.

\begin{proof}[Proof of Theorem~\ref{thm:packing}]
In the notation of the theorem, let $X_1 = X_{\Omega_1} \sqcup \ldots \sqcup X_{\Omega_n}$ and $X_2 = X$.

If such a symplectic embedding exists, then for each $k$, the Scaling property of ECH capacities implies that
\[ \lambda c_k(X_1) \le c_k(X_2),\]
for all $\lambda < 1$, hence
\[ c_k(X_1) \le c_k(X_2)\]
for all $k$.  Since $X_1$ and $X_2$ have the same volume, it then follows that
\[ e_k(X_1) \le e_k(X_2)\]
for all $k$.  By taking the liminf over $k$, and applying Theorem~\ref{thm:convexper} and Lemma~\ref{lem:sublead}, together with \cite[Thms. 4.1, 7.3]{Helem}
for the case where $X_2 = \mathbb{C}P^2$,
it therefore follows that
\[ - \sum_i \Per(\Omega_i) \le - \Per(X_2),\]
hence the result.
\end{proof}

Let us now prove the promised corollaries.

\begin{proof}[Proof of Corollary~\ref{cor:new}]
We have $\Per(E(1,a)) = 1 + a$ under our assumptions.  Let $\Per$ denote the perimeter of $X$.  Assume that a full filling exists, so that $\Vol(X) = a$.  
By Theorem~\ref{thm:packing},
we have
\[ (1+a) \ge \Per.\]
We can then rewrite this as
\[ \sqrt{\Vol/a} + \sqrt{a \Vol} \ge \Per,\]
hence,
\[ a\Vol + 2 \Vol + \Vol/a \ge \Per^2\]
and
\[ a^2 + a(2 - \Per^2/\Vol) + 1 \ge 0.\]
This implies that $a \ge a_0.$
\end{proof}

\begin{proof}[Proof of Corollary~\ref{cor:super}]
Since a ball and $\mathbb{C}P^2$ both have positive perimeter, this is an immediate consequence of Theorem~\ref{thm:packing}.
\end{proof}

\begin{rmk} This corollary is consistent with what we know about full fillings of monotone manifolds such as $\C P^2$ or $\CP^1\times \C P^1$.  As noted in \cite{AADT}, these manifolds admit increasing staircases whose steps  converge to $a_0$ with inner corners  on  the volume curve; but these inner corners occur at rational values $a< a_0$, that are ratios of side lengths of appropriate almost toric fibrations. 
\end{rmk}

\begin{rmk}
\label{rem:wormleighton}
Work of Wormleighton (\cite[Thm. 3]{Wormleighton}) shows that for concave toric domains $X_{\Omega}$,
\[ \liminf_{k \to \infty} e_k(X_{\Omega}) \ge -\frac{1}{2} \Per(X_{\Omega}).\]
Thus, by our argument, Theorem~\ref{thm:packing} holds for packings by concave toric domains as well.
 
It would be interesting to know if the same phenomenon holds for the concave into concave or convex into concave cases.  However, ECH capacities alone cannot prove this.  Indeed, one can take $X$ to be the ellipsoid $E(1,3/2)$, and $Y$ to be a concave rearrangement with the same weights but larger perimeter.  Then, one would like to show that $X$ can not embed into $Y$;  however, ECH capacities cannot see the difference between $X$ and $Y$ at all, since the weights determine the ECH capacities \cite{Choi}. 
\end{rmk}

\section{Domains with no staircases}\label{sec:nostair}
 
 We now discuss conditions on $\Om$ that imply that $X_\Om$ has no staircase. Since the accumulation point $z_\infty$ of any staircase is $\ge 1$ and satisfies the condition $z_\infty + \frac 1{z_\infty} = \frac{\Per^2}{\Vol} - 2$, one obvious necessary condition for the existence of a staircase is that 
 $\frac{\Per^2}{\Vol} \ge 4$, or equivalently $\Per \ge 2\sqrt{\Vol}$.  Thus no domain $\Om$ that is disjoint from the axes and has  smooth boundary of positive curvature can support  a staircase. Our first main result is that even if only a small part of $\p \Om$ has positive curvature, 
 $X_\Om$ has no staircase.
In fact we prove something stronger: namely such a domain does not support infinitely many obstructive classes with different centers.\footnote
{Our argument does not rule out the possibility that there are infinitely many different obstructive classes with the same center, though this seems very unlikely.}

We then describe a few cases in which we can rule out the existence of a staircase  by showing, using properties of the Gromov width,  that the accumulation point is obstructed.  Finally we note that
the general question of which domains have staircases seems to be  very subtle. As we pointed out in Corollary~\ref{cor:ellip}, it follows easily from the accumulation point theorem that   no irrational ellipsoid $E(1,a)$ has a staircase. Nevertheless, we show in Proposition~\ref{prop:obs} that such an ellipsoid always has infinitely many obstructive classes --- these just happen to be overshadowed and so not visible in the capacity function.

 \subsection{Curvy domains have no staircases}\label{ss:curvy}

In this section, we prove the following more precise version of Theorem~\ref{thm:curvy}:

\begin{prop} \label{prop:round} Let $\Om$ be a convex region  whose boundary contains a $C^3$-smooth segment $S$ that is not linear.
Then $X_\Om$ admits only finitely many obstructive classes, and in particular has no staircase.
\end{prop}

Before proving the proposition, we give a proof of Theorem \ref{thm:class}, which utilizes Proposition~\ref{prop:round}.

\begin{proof}[Proof of Theorem \ref{thm:class}]
    Here, $X_\Omega$ is a smooth convex toric domain. If $X_\Omega$ contains a neighborhood of the origin, then $\p^+\Omega$ must be smooth.  So either $X_\Omega$ is an ellipsoid, or the perimeter $\p^+\Omega$ 
    has a curvy point and there is no staircase by Proposition~\ref{prop:round}. If $X_\Omega=E(1,b)$, then 
    by Corollary~\ref{cor:ellip} and work of Cristofaro-Gardiner in \cite{CG2}, $X_\Omega$ has an infinite staircase exactly if it is a ball, a scaling of $E(1,2),$ or a scaling of $E(1,3/2).$
    
    If $X_\Omega$ does not contain a neighborhood of the origin, then again $\p\Omega$ must have a curvy point, and the claim holds as above.
\end{proof}

We now return to the proof of Proposition \ref{prop:round}, which
is based on the following lemma.   Recall from the discussion after \eqref{eq:E} that the only really relevant obstructive  classes $\bE = dL - \sum_{j=1}^N \Tm_j \TE_j - \sum_{i=1}^n m_i E_i$ are ordered, that is  $\Tm_1\ge \Tm_2\ge \dots$ and $m_1\ge m_2\ge\dots$.

\begin{lemma}\label{lem:round}   Let $\Om = \Om(b;(b_j))$ be such that $(b_j)$ contains a subsequence $(b_k')_{k\ge k_0},$ 
of the form $b_k' = c/k^r + o(1/k^r)$ for some $r>1$ and constants $c>0, k_0 \ge 1$. Then there is $N_0$ such that no  (ordered) obstructive class $\bE$
with $\Tm_j>0$ for all $j<N$  has $N\ge N_0$. 
\end{lemma}

The proof is deferred to the end of the subsection

\begin{cor}\label{cor:round}  In the situation of Lemma~\ref{lem:round}, the set of centers of  classes that are obstructive for $X_\Om$ is finite.
  In particular, $X_\Om$  has no staircase.
\end{cor}
\begin{proof}  If there were infinitely many such centers, then they would have to a form a sequence that converges to
the accumulation point $a_\Om$.  Since we may always replace these classes by the corresponding ordered classes,  Lemma~\ref{lem:round} implies that there is $N_0$ so that the obstructions given by this set of ordered classes do not depend on the $b_j, j\ge N_0$. Hence, because $\Vol(\Om) < \Vol(\Om_N)$ where $\Om_N = \Om(b; (b_j)_{j <N})$,
 these classes are also obstructive in $X_{\Om_N}$ for all $N\ge N_0$. Thus their centers would have to converge to $a_{\Om_N}$
 for all such $N$.  But if $a_{\Om_{N}} = a_{\Om_{N+1}}$ for some $N>N_0$ then  $\frac{\Per(\Om_{N})}{\Vol(\Om_{N})} = 
 \frac{\Per(\Om_{N+1})}{\Vol(\Om_{N+1})} = :\frac PV $, which is impossible when $b_{N+1}$ is so small that $b_{N+1} P < V$.
\end{proof}

\begin{proof}[Proof of Proposition~\ref{prop:round} assuming Lemma~\ref{lem:round}] In the cutting process that expresses $\Om$ as $\Om(b;(b_j))$ there is a cut that is tangent to $\p \Om$ at some point $p\in S$. 
 By an affine change of coordinates, we may assume that $p = (0,0)$, that the cut is along the $x$-axis and that near $p$ the curve $S$ is the graph of a $C^2$-smooth increasing function $f:[0,\eps)\to [0,\infty)$.  Moreover, since $S$ is nonlinear,   by changing the choice of $p$, if necessary, we may assume that $f''(0)>0$.  
 Thus we are in the situation of Figure~\ref{fig:S}, and, for some $k_0$, there is a sequence of cuts tangent to $S$ at points $(x_k,f(x_k)), k\ge k_0, $ such that $x_k\to 0$ and $f'(x_k) = 1/k$.
 
 Assume first that
 \begin{align*}
 f(x) = \frac{c}{2} x^2 + \frac{d}{3} x^3
 \end{align*}
  for some constants $c,d$ so that
 $f'(x_k) = cx_k + d x^2= \frac 1k$.  Let us also assume for convenience that $d > 0$.  Then this has positive solution
\begin{align*} x_k = \frac{ -c + \sqrt{ c^2 + 4\frac{d}{k} }}{2d} &= \frac{c}{2d} \bigl(-1 + \sqrt{1 + \frac{4d}{c^2k}}\bigr) \\
& = \frac{c}{2d}\bigl(\frac{2d}{c^2k} - \frac{4d^2}{c^4 k^2}\bigr) + O(\frac{1}{k^3})\\
& = C_1 \frac{1}{k} + C_2 \frac{1}{k^2} + O(\frac{1}{k^3})
\end{align*}
 for some explicit constants $C_1, C_2$.  Then 
 \[ f(x_k) =  C'_1 \frac{1}{k^2} + C'_2 \frac{1}{k^3} + O(\frac{1}{k^4}).\]
 Let $y_k$ denote the point where the tangent line through $(x_k,f(x_k))$ meets the $x$-axis.  Then the cuts have size
\begin{align*} y_k - y_{k+1} &= (x_k - x_{k+1}) - k ( f(x_k) - f(x_{k+1}) ) \\
&= (C_1 - C'_1) \frac{1}{k(k+1)} + O( \frac{1}{k^3}) = \frac{1}{2c} ( \frac{1}{k(k+1)}) + O( \frac{1}{k^3}).
\end{align*}
 Thus we get a sequence of cuts of sizes $y_k - y_{k+1} = \frac{1}{2ck(k+1)} + O(\frac{1}{k^3})$.  The same result holds in the case $d \le 0$. Indeed, the $d = 0$ case is simpler, and  the $d < 0$ case follows by essentially the same argument.
 
 In general,   when $x\approx 0$, $f(x) = \frac{c}{2} x^2 + \frac{d}{3} x^3 + O(x^4)$, for some constants $c$ and $d$. A similar argument then shows that, as before, for large $k$ the cuts have size $C \frac{1}{k(k+1)}+ O(\frac1{k^3})$. 
 
 Therefore, in all cases we may apply Lemma~\ref{lem:round} to conclude that $X_\Om$ admits only finitely many obstructive classes. 

  \end{proof}
  
  \begin{figure}
\includegraphics[scale=0.6]{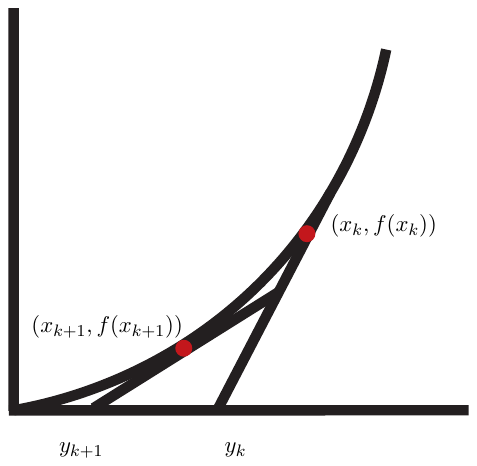}\vspace{-4in}
\caption{This figure illustrates the cutting algorithm in the proof of Proposition~\ref{prop:round}.}
\label{fig:S}
\end{figure}

\begin{proof}[Proof of Lemma~\ref{lem:round}] Suppose, by contradiction that there is an  infinite sequence of ordered obstructive classes $\bE_N$ where  $N$ increases to $ \infty$ such that
$\Tm_j= 0$ for $j\ge N$    while $\Tm_{N-1}> 0$.
In the notation of \eqref{eq:definition}, for each $N$  write 
\[ \bbm_{\bE} = \frac d{\la_a}\bw + \eps\] 
as in \eqref{eq:eps}, where $\eps$ is the error vector and we have set $b=1$ for simplicity.
(The other quantities $\la_a, \bw$  are defined in \eqref{eq:definition}.  To simplify notation, we also suppress the index $N$.)
 Note that by Corollary~\ref{cor:accum2} we may choose constants $M_1,M_2$ (independent of $\bE_N$) 
 so that  $0 < M_1 \le \lambda_a \le M_2$.
We aim to  show that under the given conditions  when $N$ is sufficiently large we cannot have $\eps\cdot\eps< 1$,  basically    because  sequences of the form $b_k' = c/k^r + o(1/k^r), k\ge 0,$ cannot be well approximated by integers.
\MS

For each $N$ the sequence $b_j, j\ge N,$ contains a subsequence with entries from the $b'_k, k\ge k_0$.  Let $K(N)$ denote the smallest value of $k$ in this sequence.
Then the error vector $\eps$ contains a subsequence with entries of the form $- d b'_k/\lambda_a, k\ge K(N)$.

 Since $\eps\cdot \eps< 1$, we must have, for $N$ sufficiently large, 
\begin{align*}
M_2^2 &> \sum_{k \ge K_N}   d^2 (b'_k){^2}  \ge d^2 
\sum_{k\ge K} \frac{c^2}{2k^{2r}} \\ & \ge (dc)^2/2 \int_{K}^\infty x^{-2r} dx \ge
 \frac {(dc)^2}{2(2r-1)K^{2r-1}},
\end{align*}
so that $dc \le \sqrt{2(2r-1)} M_2 K^{r-1/2}$
(Note that we have absorbed the $o(1/k^{2r})$ term into the $c/k^{2r}$ term  in the above equation, hence the appearance of a $2$ in the denominator.)

Now suppose $N$ so large that $K: = K(N)\ge k_0+ 5$. Then there are 
 four distinct $j_i < N, 1\le i\le 4, $ with $b_{j_i} = b'_k$ for some $k\ge k_0$ with $b_{j_i} \le 2 c (K-i)^{-r}$.  (Note that we have again absorbed the $o(1/k^r)$ term into the coefficient, hence the appearance of the $2$.)

For such $j_i$ we have
$$
d b_j \le 2 \frac {dc}{ (K-i)^r}  \le 2M_2 \cdot 2^{r} \sqrt{ (2r-1) } \frac{ K^{r-1/2}}{K^r}.
$$
(We have absorbed the ratio $K/(K-i)$  into the $2^r$ term.)  Thus for sufficiently large $N$ (and hence large  $K=K(N)$) we have $\frac{d b_{j}}{\lambda_a} < 1/2$ for these four $b_j$.
But then the corresponding coefficients of $\eps$ are $> 1/2$ since the corresponding entries $\Tm_j$ in $\widetilde{\bf m}$ are at least $1$. Therefore, since there are four of these terms, we find that $\eps\cdot\eps> 1$, a contradiction.
\end{proof}

\subsection{Cases in which the accumulation point is obstructed}\label{ss:accobstr}
We can also rule out staircases in a different way by giving suitable constraints on the perimeter such that $X_\Omega$ is blocked from having a staircase because  the capacity function is greater than the volume at the accumulation point.  We  illustrate this
approach by describing some examples in which the obstruction comes from the the Gromov width $c_{Gr}(X_\Omega):= \max\{ \la | B(\la)\se X_\Om\}$. For the domains $X_\Om$ this is just
 the reciprocal of
 the first ECH capacity.) 
 
\begin{prop} \label{prop:noStairGromov}
     Let $X_\Omega$ be a convex toric domain such that $c_{Gr}(X_\Omega) \leq d \le \sqrt{\Vol(\Om)}$.
     Then, if 
     \begin{equation} 
     \label{eq:PVbound} \Per(\Omega)<d+\frac{\Vol(\Omega)}{d},
     \end{equation} 
     there cannot be a full filling at the accumulation point of $X_\Omega.$ In particular, $X_\Omega$ does not have a staircase. 
\end{prop}
\begin{proof} 
     Let $a_0 \geq 1$ be the accumulation point of $X_\Omega.$ We show that under the given conditions, the Gromov width obstructs the accumulation point.
     If the Gromov width of $X_\Omega$ is $\leq d,$ then 
     \[c_{X_\Omega}(a_0) \geq c_{X_\Omega}(1) \geq 1/d.
     \]
    For there to be a staircase accumulating at $a_0$, we then must have
    \[
     \sqrt{a_0/\Vol(\Omega)} \geq 1/d \implies
     d^2a_0 \geq \Vol(\Omega).
     \]
      Hence, it suffices to show that under the assumptions on $\Omega$, we have
     \begin{equation} \label{eq:GrVolBound} 
     d^2a_0< \Vol(\Omega). 
     \end{equation}
Since $\Per(\Omega)<d+\Vol(\Omega)/d,$ we have
     \[ \frac{1}{a_0}+a_0+2=\frac{\Per(\Omega)^2}{\Vol(\Omega)}<\frac{(d+\Vol(\Omega)/d)^2}{\Vol(\Omega)}=\frac{d^2}{\Vol(\Omega)}+\frac{\Vol(\Omega)}{d^2}+2,\]
     and hence, because $\frac{\Vol(\Omega)}{d^2}\ge \frac{d^2}{\Vol(\Omega)}$ we conclude $a_0<\Vol(\Omega)/d^2$ as desired. 
\end{proof}

In the following corollary, we list some domains that satisfy the assumptions of Proposition~\ref{prop:noStairGromov} and as such have no staircases. For brevity we will say that a {\bf piece of the boundary is irrational} if it has zero affine length. Thus it is the connected union of  line segments of irrational slope with segments that are nonlinear, but not necessarily $C^3$-smooth.

\begin{cor} The  domain $X_\Omega$ cannot have a staircase if $\Omega$ is given by: 
\begin{itemlist}
    \item[{\rm (i)}]  The region bounded by the axes and an irrational piece from $(0,1)$ to $(b,0)$ with $b \geq 1$.
    \item[{\rm (ii)}] The region bounded by the axes and a curve  $S$ from $(0,1)$ to $(b,0)$ with maximum $y$-coordinate of $d$ such that there is a point $(c,d)\in S$ 
    for which  $\Aff(S)\leq d +\frac{c}{d} - 1$.
    \item[{\rm (iii)}] A \lq\lq fuzzy\rq\rq\, polydisk: i.e. for $b \geq 1$, a domain bounded by the axes, a horizontal line from $(0,1)$ to $(0,b)$, and a strictly convex piece  from $(1,b)$ to $(0,b+\eps)$. 
\end{itemlist}
\end{cor}

\begin{proof}
To rule these out, we merely write the Gromov width, perimeter, and volume bounds and then  check that \eqref{eq:PVbound} in Proposition~\ref{prop:noStairGromov} holds. 

    For (i), if the point on $\Omega$ with maximum $y$-coordinate is $d \geq 1,$ then $c_{Gr}(X_\Omega) \leq \min(b,d)$, $\Per(\Omega)=b+1,$ and $\Vol(\Omega)>bd$. 

    For (ii), we have $c_{Gr}(X_\Omega) \leq \min(b,d)$,  $\Vol(\Omega)>bd+c,$ and $\Per(\Om)\leq b+d +\frac{c}{d}$.

    For (iii), we have Gromov width $=1$, $\Per(\Omega)=2b+1+\eps$ and $\Vol(\Omega)>2b+\eps$. Note, we require the strictly convex assumption, so we get the strict bound on $\Vol(\Omega)$. 
\end{proof}

\begin{rmk}
If $c_{Gr}(X_{\Omega})=d$ then $ d \le \sqrt{\Vol(\Om)}$.  If we also  have  equality in \eqref{eq:PVbound},  then $1/d = c_{X_\Om}(1)$ equals the 
volume obstruction $\sqrt{\frac{a_0}{\Vol}}$ at  the accumulation point and there can be no
increasing staircase, though there could be a decreasing one. 
Note that we also have equality in \eqref{eq:PVbound} if we relax the strict convex assumption for the fuzzy polydisc, i.e. if our domain is given by replacing the vertical right hand edge of the polydisc by an irrationally slanted line. 

   Just as in the ellipsoid case, one might be able to rule  out the existence of a  descending staircase by computing an appropriate ECH capacity. 
  \end{rmk}

\subsection{Ghost stairs for irrational ellipsoids}\label{ss:ghost}

We now show that any irrational ellipsoid $E(1,\al)$ supports both ascending and descending sequences of perfect classes $(\bE_k)_{k\ge 1}$ that are obstructive but do not constitute staircases because they are {\bf overshadowed}: in other words there is a different class $\bE'$ (called an {\bf overshadowing class}) whose obstruction goes through the accumulation point and is at least as large as those from the $\bE_k$.
It turns out that these classes $\bE_k$ are determined by the  convergents to the irrational number $\al$, where the even (resp. odd) convergents give rise to the ascending (resp. descending) obstructive classes.
We begin by introducing some terminology.

A {\bf quasi-perfect class} $\bE=(d;\btm;{\bf m})$ is a quasi-exceptional class such that ${\bf m}=W(p,q)$ for some positive integers $p,q$. Here, $W(p,q)$ is the normalized (or integral)  weight expansion given by $W(p,q):=q\bw(p/q).$ A {\bf perfect class} is an exceptional quasi-perfect class. For a (quasi)-perfect class, we call $p/q$ the {\bf center} of the class. From \eqref{eq:wtai}, the Diophantine equations \eqref{eq:Dioph} for a quasi-perfect class are 
\begin{equation} \label{eq:DiophPerfect}
    3d-\sum_{i=1}^n \Tm_i=p+q \quad \text{and} \quad d^2-\sum_{i=1}^n \Tm_i^2=pq-1.
\end{equation} 

One essential feature of $W(p,q)$ is its relation to the continued fraction of $p/q.$ Following \cite[Section 2.2]{ball}, if the continued fraction of $p/q$ is $[a_0;a_1,\hdots,a_n],$ then 
\[ W(p,q)=\bigl(X_0(p/q)^{\times a_0},X_1(p/q)^{\times a_1},\hdots,X_n(p/q)^{\times a_n}\bigr)
\]
for some integers $X_i(p/q).$ Here the last entry is $X_n(p/q): =  1$, and the other integers $X_i : = X_i(p/q)$ $i = n-1, n-2,\dots$  are defined  recursively by
$X_{i-1} = a_i X_i + X_{i+1}$, where we set  $X_{n+1}(p/q): = 0$.  

 If $z$ is sufficiently close to $p/q =  [a_0;a_1,\hdots, a_n]$, then its continued fraction (which may be infinite) has the same initial terms
$a_0,\dots, a_{n-1}$ as $p/q$; moreover, when $n$ is odd, the $n$th entry is $a_{n}$ for $z< p/q$ and $a_{n} - 1$ for $z> p/q$.
The components $w_i(z)$ of the weight expansion $\bw(z)$ are now  linear functions of  $z$: for example if $7/3 < z < 5/2$ then 
$\bw(z) = (1,1,z-2, z-2,5 - 2z,\dots)$.  It is again convenient to write
\[ \bw(z)=\bigl(x_0(z)^{\times a_0},x_1(z)^{\times a_1},\hdots\bigr)\]
where the $x_i$ are appropriate linear functions. Thus for $z\in (7/3, 5/2)$ as above, $x_0(z) = 1, x_1(z) = z-2, x_2(z) = 5-2z$ and so on.

\begin{lemma} \label{lem:convObs}  Let $\bE = (d, \btm, W(p,q))$ be a quasi-perfect class with center $a: = p/q$, where $a$ has continued fraction expansion 
 $a= [a_0;a_1,\hdots,a_n]$.
     \begin{itemize}\item[{\rm(i)}] If $n$ is odd (resp. even) and $|p/q-z|>0$ is sufficiently small, then
 $x_{n+1}(z) = p_n-q_nz$ (resp. $x_{n+1}(z)=-p_n+q_nz$). 

\item[{\rm(ii)}] For suitable $z_1,z_2$ with $z_1< p/q < z_2$ we have
  \begin{align*}
   W(p,q) \cdot \bw(z)=\begin{cases}
qz \quad \text{if $z_1<z<p/q$} \\
  p \quad \text{if $p/q \leq z <z_2$}.\end{cases}
\end{align*}
 \item[{\rm(iii)}] 
  if $p/q$ is an odd convergent of $z,$ then $z<p/q$ and $\mu_{\bE,\bm{b}}(z)$ is given by 
  \begin{align}\label{eq:perfectObs1}
\mu_{\bE,\bm{b}}(z) =    \frac{qz}{bd-\btm \cdot \bm{b}}.
\end{align}
 \item[{\rm(iv)}] 
  if $p/q$ is an even convergent of $z,$ then $p/q<z$ and $\mu_{\bE,\bm{b}}(z)$ is given by 
  \begin{align}\label{eq:perfectObs2}
\mu_{\bE,\bm{b}}(z) =    \frac{p}{bd-\btm \cdot \bm{b}}.
\end{align}
\end{itemize}
\end{lemma}
\begin{proof} The first  claim is proved in \cite[Lem.2.2.1]{ball}, while the second, though it is implicit in the results in \cite[\S2.2]{ball} 
is most clearly proved in \cite[Lemma 16]{BHM}.  Claim (iii) and (iv) 
is an obvious adaptation of this proof 
     to the case when $\btm$ and $\bm{b}$ are vectors rather than single numbers: see also the discussion concerning equation~(2.22) in \cite{AADT}. 
\end{proof}

\begin{lemma}
    The classes \begin{equation} \label{eq:Epq}
        \bE(p,q)=\bigl(p;W(p,p-q),1;W(p,q)\bigr)
    \end{equation} are perfect classes.
    \end{lemma}
    
    \begin{proof} We first check \eqref{eq:DiophPerfect} by using the properties of the weight expansion in \eqref{eq:wtai}. We have that 
    \[ 3d-\sum \Tm_i=3p-(p+p-q-1)-1=p+q\] 
    and 
    \[ d^2-\sum \Tm_i^2=p^2-(p(p-q))-1=pq-1.\]
    Hence, $\bE(p,q)$ is a quasi-perfect class. 
    
    To see that
    $\bE(p,q)$ is perfact  we argue by induction on the number of nonzero entries in the tuple
    $\bigl(p;W(p,p-q),1;W(p,q)\bigr)$.
      We reduce this number by Cremona moves.\footnote
      {
  As noted in \cite{ball} these preserve the class of exceptional curves. For more detail see the beginning of \S\ref{ss:perf}  below.}
    To describe these, we write
    $$
    \bE: = (d;\Tm_1,\dots, \Tm_k; m_1,\dots, m_n) = :\bigl(d; c_{\tilde{1}},\dots, c_{\tilde{k}}; c_1,\dots,c_n\bigr)
    $$
    and denote by $C_{{\tilde i}, {\tilde j}, \ell}$ the move that replaces $d$ by $2d-\de$, where $\de: = c_{\tilde i} + c_{\tilde j} + c_\ell$, and subtracts $\de - d$ from
    the entries in the places  ${\tilde i}, {\tilde j}, \ell$.  These moves preserve the Diophantine identities by \cite[Prop 1.2.12]{ball}.
    Thus, when $p>2q$,
\begin{align*}
C_{{\tilde 1}, {\tilde 2}, 1}\bigl( \bE(p,q)\bigr) & = C_{{\tilde 1}, {\tilde 2}, 1}\bigl(p,p-q,q,\dots; q,q,\dots\bigr)\\
&
 = \bigl(p-q; p-q,0,\dots; 0,q, \dots \bigr) \approx   \bE(p-q,q),
 \end{align*}
 where  $ \approx $ means equality after deleting the two zero entries.
 To deal with the case $1<p/q < 2$, we note that the entries in the
 class $\bE(p,q)$ are a rearrangement of those in $\bE(p,p-q)$, where now $p/(p-q) > 2$.  Therefore if $1<p/q < 2$ we can rearrange to $\bE(p,p-q)$ and then reduce the length  
 as before.  Note that  all the  tuples obtained by this reduction process can be rearranged to have the form $  \bE(p,q)$ for some $p,q$. The shortest such tuple is $\bE(1,1) = (1;1;1)$, which corresponds to the exceptional class $L - E_1 - E_2$.
\end{proof}

We next consider obstructive classes for the ellipsoid $E(1,\al)$. Let $[a_0;a_1,\hdots]$ denote the (potentially infinite) continued fraction of $\al$. The convergents $$
p_1/q_1=[a_0;a_1],\ p_2/q_2=[a_0;a_1,a_2],\ \hdots,\ p_{n}/q_{n}=[a_0;a_1,\hdots,a_{n}],\hdots,
$$
form a decreasing sequence when $n$ is odd and an increasing sequence when $n$ is even. 
Let $z_{n}:=p_{n}/q_{n}$ and
\[ \bE_n:=\bE(p_n,q_n):=\bigl(p_{n};W(p_{n},p_{n}-q_{n}),1; W(p_{n},q_{n})\bigr)
\]
denote the corresponding sequence of perfect  classes. 
We saw in Corollary~\ref{cor:ellip} that if $\al$ is irrational the ellipsoid $E(1,\al)$ does not have a staircase. Indeed its obstruction $c_{E(1,\al)}$ is constant and equal to $1$ for $z\le \al$, the accumulation point, and  is given on the interval $ \al \le z\le  \lceil \al\rceil$ by the straight line through the origin of slope $1/\al$.  We show in Example~\ref{ex:E'} below  that this obstruction is given by the perfect class
with $$
\bE': = (d;\Tilde{\bf m}; \bm) = (k; k-1, 1^{\times (k-1)}; 1^{\times (k+1)}), \;\;\mbox { where } \;\; k<\al < k+1,
$$
so that $c_{E(1,\al)}(z) = z/\al, \;\; \al\le z\le \lceil \al\rceil$. 
 Nevertheless, we now show that there are infinitely many other obstructive classes for $E(1,\al)$ that give obstructions 
 whose peaks also lie on this line. Thus they do not form a staircase. Further, Proposition~\ref{prop:obs} implies that the obstruction from $\bE_n$ when $n$ is even is equal to $c_{E(1,\al)}(z)=1$ for $p_n/q_n<z<\al$, so all of these obstructions are live, but as they agree for all $n$, they do not form a staircase either. 
 \MS
 
\begin{prop} \label{prop:obs}
    Let $\ell(\al)$ be the length of the continued fraction of $\al>1$ with convergents $p_n/q_n$.     For $n< \ell(\al)$, the classes $\bE_n: = \bE(p_n,q_n)$ are obstructive for $c_{E(1,\al)}(z)$. In particular, if $\al \not\in \Q$, then there are infinitely many obstructive classes for $c_{E(1,\al)}(z)$, however they do not form a staircase.
\end{prop}
\begin{proof}
    We begin with the case where $n$ is odd implying that $p_n/q_n>\al.$
      Let $[a_0;a_1,\hdots,a_n,\hdots]$ denote the continued fraction of $\al$.  To show $\bE_n$ is obstructive for $n$ odd, we will explicitly compute $\mu_{\bE_n,E(1,\al)}$ at the center $z_{n}.$   The negative weight expansion corresponding to $E(1,\al)$ is
    \[ (b;\bm{b}):=(\al;\al-1,\bw(\al-1)).\]
    We denote $\mu_{\bE_n,E(1,\al)}$ by $\mu_{\bE_n,\al}$
    By Lemma~\ref{lem:convObs}, we have that 
    \[ \mu_{\bE_n,\al}(z_{n})=\frac{p_{n}}{\al p_{n}-\btm \cdot \bm{b}}.
    \]
 
    We now compute $\btm \cdot \bm{b}$.  By definition of $\bE_n$ in \eqref{eq:Epq}, $\btm$ is the weight expansion of $p_n/(p_n-q_n)$ with a $1$ adjoined at the end. We first assume that $\al>2,$ which implies the continued fraction of $p_n/(p_n-q_n)$ is $[1;a_0-1,a_1,\hdots,a_{n}]$. Observe that $(p_n-q_n)/q_n$ has continued fraction $[a_0-1;a_1,\hdots,a_n].$ Hence,
  \begin{equation} \label{eq:btmObs}
      \btm=(W(p_n/(p_n-q_n)),1)=(p_n-q_n,W(p_n-q_n,q_n),1).
  \end{equation}
Note that $(p_n-q_n)/q_n$ is an odd convergent of $z: = \al-1=[a_0-1,\hdots,a_n,a_{n+1}].$

    By \eqref{eq:btmObs} and Lemma \ref{lem:convObs}~(ii), we have that
    \begin{align*} \btm \cdot \bm{b}&=(p_n-q_n,W(p_n-q_n,q_n),1)\cdot(z,\bw(z)) \\
   & = (p_n-q_n)\cdot z+W(p_n-q_n,q_n) \cdot \bw(z) +x_{n+1}(z) \\
           & = (p_n-q_n)\cdot z+q_n\cdot z+x_{n+1}(z) \\
    &=p_n \cdot z+ x_{n+1}(z).
    \end{align*}
But we saw in Lemma \ref{lem:convObs}~(i) that for $z $ less than and sufficiently close to $p_n/q_n-1 = (p_n-q_n)/q_n$, we have
$x_{n+1}(z) = ( p_n-q_n) - q_n z$.  Therefore 
     \begin{equation} \label{eq:obsFinal}
 \btm \cdot \bm{b} = ( p_n-q_n) (1+z) =( p_n-q_n) \al ,
    \end{equation} 
where the last equality holds because $z = \al-1$.

    If $1 \leq \al \leq 2$ with continued fraction $[a_0,\hdots,a_n],$ then the continued fraction of $p_n/(p_n-q_n)$ is $[1+a_1,a_2,\hdots,a_n]$. A similar computation shows that $\btm \cdot \bm{b}=\al(p_n-q_n).$

From \eqref{eq:obsFinal}, if $n$ is odd, we have
\begin{equation} \label{eq:obsEn}
     \mu_{\bE_n,\al}(z_{n})=\frac{p_{n}}{\al p_{n}-\al(p_{n}-q_{n})}=\frac{z_{n}}{\al}. 
\end{equation}
This is above the volume obstruction at $z_n:$
\[\mu_{\bE_n,\al}(z_{n})=\frac{z_{n}}{\al }>\sqrt{\frac{z_{n}}{ \al}}=  V_{E(1,\al)}(z_{n})\]
because  $z_{n}>\al$ for odd $n$. On the other hand, for $k<\al< k+1$. we know from Corollory~\ref{cor:ellip}
that  $c_{E(1,\al)}(z) = z/\al $.  Therefore these obstructions are not visible in the capacity function.

The situation when $n$ is even is very similar. In that case, following a similar process and using the results of Lemma~\ref{lem:convObs} for $n$ even, we have that 
\[ \btm \cdot \bm{b}=p_n(\al-1).\]
Hence, since  the centers of the obstructions are $< b$, we find that
\[ \mu_{\bE_n,\al}(z_n)=\frac{p_n}{\al p_n-(\al-1)p_n}=1.\] 
Thus, by Corollary \ref{cor:ellip}, $c_{E(1,a)}(z)=1$ for $z<\al$. 
\end{proof}

\begin{example}\label{ex:E'}  \rm We now show that the capacity function for $E(1,\al)$ is given on the interval $[\al, n+1]$ by the class
 $\bE': = (k; k-1, 1^{\times (k-1)}; 1^{\times (k+1)})$, where  $\al \in (k,k+1).$   Recall from the above that the negative weight expansion corresponding to $E(1,\al)$ is
    \[ (b;\bm{b})=(\al;\al-1,\bw(\al-1)).\]
  Since $\al \in (k,k+1),$ we have $\bw(\al-1)=(1^{\times (k-1)},\hdots).$  Thus
    \[ (k-1,1^{\times (k-1)}) \cdot (\al-1,\bw(\al-1))=(k-1)(\al-1)+(k-1)=\al(k-1).\]
    Hence, for $z \leq k+1,$ we find that
    \[\mu_{\bE',\al}(z)=\frac{z}{\al k-\al(k-1)}=\frac{z}{\al}\]
as claimed.
\end{example}

\section{Domains with staircases}\label{sec:stair}

This section is devoted to the proof of Theorem~\ref{thm:inftystair}, which claims that there is a family of rational domains $\Om_n, n\ge n_0,$ of increasing cut-length that support staircases. The construction and proof use the methods developed in \cite{BHM,MM,MMW}, while the staircase steps are iteratively generated by perfect seed classes.  It seems likely that there are many different examples of this kind; our example was chosen with a view to minimizing the needed calculations.

\subsection{Outline of the construction}

We construct  staircases whose steps  have the same form as those in $\Hh_b: = \C P^2(1)\# \ov{\C P}^2(b)$ that were constructed and classified in \cite{BHM,MM,MMW}.\footnote
{
Here $\Hh_b$ is  the one-point blowup of $\C P^2$ in which the line has size $1$ and the exceptional divisor size $b\in (0,1)$.}
  Thus we consider sequences of perfect classes $\bE_k = \bigl(d_k; (\Tm_{kj}); W(p_k,q_k)\bigr)$, where the entries of the tuple $(d_k, (\Tm_{kj}), p_k, q_k)$ all satisfy a recursion of the form
\begin{align}\label{eq:recur} 
x_{k+1} = t x_k - x_{k-1},\quad k\ge 1.
\end{align}
Hence the staircase steps are determined by the two {\bf seed classes} $\bE_0, \bE_1$ and the {\bf recursion variable} $t$. 
In order for the resulting sequence of classes to be perfect,  the seeds, which themselves must be perfect, must be suitably compatible and an appropriate recursion variable $t$ must be chosen.  However, it turns out that the conditions developed in
\cite{MM,MMW} to deal with these issues easily generalize to the current situation: in fact, the only real difference between the current situation and that in the previous work is that instead of the two {\bf degree variables} $(d,m)$  needed to describe a class in $\Hh_b$ there are now a finite number of such variables, one for each element of the tuple $(b; (b_j))$ that defines the region $\Om = \Om(b; (b_j))$.

\begin{rmk}\rm The limiting regions $\Om(b;(b_j))$ for the staircases that we construct  have $b=1$ and $b_j$ equal to the limit of the ratios $\Tm_{jk}/d_k$; see \eqref{eq:domn}.  Hence we only construct staircases in  regions with a finite number of sides of rational slope. The most obvious approach to constructing staircases in a domain with infinite cut length would seem to be to consider a domain with at least one edge given by a line of irrational slope $\al$, and then to try to define obstructive classes that are partially recursively defined and partly defined via convergents
as in \S\ref{ss:ghost}.
But it is not at all clear how to do this.  
\end{rmk} 

\begin{definition}\label{def:adj} Two quasi-perfect classes $\bE_k = \bigl(d_k; (\Tm_{kj})_j; W(p_k,q_k)\bigr),\ k=0,1$ with $p_1/q_1 \le p_0/q_0$ are said to be
{\bf adjacent} if 
\begin{align}\label{eq:adj} 
d_0d_1 - \sum_j \Tm_{0j}\Tm_{1j}  = p_1q_0.
\end{align}
\end{definition}

The following lemma generalizes \cite[Lemma 3.1.4]{MM}.

\begin{lemma}\label{lem:recur} Suppose that the quasi-perfect classes $\bE_0,\bE_1$ with  $p_1/q_1 < p_0/q_0$ are  adjacent, and define
$t : = p_0q_1-p_1q_0$. Further, assume that $p_1>p_0, q_1>q_0$
and $t \geq 2$. Then the classes $\bE_k: =\bigl(d_k; (\Tm_{kj})_j; W(p_k,q_k)\bigr)$ with entries $x_k$ defined by the recursion \eqref{eq:recur} are all quasi-perfect. Moreover the sequence of centers $p_k/q_k$ decreases.
\end{lemma}

\begin{proof} The assumption that  $p_1>p_0, q_1>q_0$ 
and $t \geq 2$ guarantees that the sequences $(p_k), (q_k)$  are increasing. 
A positive tuple is quasi-perfect if its entries satisfy both the linear relation $c_1(\bE_k) = 1$ and the quadratic relation 
$\bE_k\cdot\bE_k = -1$ in \eqref{eq:Dioph}. The linear relation translates to the homogenous linear relation
$3d_k = \sum \Tm_{jk} + p_k + q_k$ and hence is preserved by any linear recursion.  
For quasi-perfect classes $(d,(m_j),p,q)$  the quadratic relation is $d^2 - \sum \Tm_j^2 - pq = -1$. Thus, to check the quadratic relation for $\bE_2$ we calculate:
\begin{align*}
&(td_1-d_0)^2-\sum_j(t\Tm_{1j}-\Tm_{0j})^2-(tp_1-p_0)(tq_1-q_0)\\
&\quad =
t^2\bigl(d_1^2-\sum_j\Tm_{1j}^2-p_1q_1\bigr)
+\bigl(d_0^2-\sum_j\Tm_{0j}^2-p_0q_0\bigr)\\
&\qquad
-2t\bigl(d_0d_1-\sum_j\Tm_{0j}\Tm_{1j}\bigr)
+t(p_1q_0+p_0q_1)\\
&\quad =
-t^2-1-2tp_1q_0+t(p_1q_0+p_0q_1)\\
&\quad =
-t^2-1+t(p_0q_1-p_1q_0)
 = -1
\end{align*}
where the second equality uses the quadratic relation and the adjacency condition, and the last equality uses that $t=p_0q_1-p_1q_0.$
Thus $\bE_2$ is quasi-perfect. 
Since $p_2=tp_1-p_0$ and $q_2=tq_1-q_0,$ we have
\[ p_1q_2-p_2q_1=p_0q_1-p_1q_0=t>0,\]
and as $p_2,q_2>0$, this gives $p_2/q_2<p_1/q_1.$ To see that $\bE_2$ is adjacent to $\bE_1$, we use that $d_2,\Tm_{2j}$ satisfy the recursion to obtain
\begin{align*}
d_1d_2-\sum_j\Tm_{1j}\Tm_{2j}
&=
t\bigl(d_1^2-\sum_j\Tm_{1j}^2\bigr)
-\bigl(d_0d_1-\sum_j\Tm_{0j}\Tm_{1j}\bigr)\\
&=
t(p_1q_1-1)-p_1q_0\\
&=
(tp_1-p_0)q_1
=
p_2q_1
\end{align*} 
where for the second equality we used the quadratic relation and adjacency of $\bE_0,\bE_1,$
and for the third the definition of $t$.  Since $p_1q_2-p_2q_1=p_0q_1-p_1q_0 = t$,  the lemma follows by induction.
\end{proof}

We will apply this lemma to two seed classes, 
 thereby obtaining a sequence of quasi-perfect classes $(\bE_k)_{k\ge 0}$.  
Just as in \cite{MM,MMW} these do form a sequence of obstructive classes for a suitable domain $X_\Om$ whose parameters are determined by the steps $(\bE_k)_{k\ge 0}$.\footnote{
In \cite{MM,MMW}, it was only necessary to choose the correct parameter $b$, but now, since we fix $b: = 1$,  we need to choose  a suitable tuple $(b_j)$.}

One  problem is that 
pre-staircases formed from these classes can rather easily be overshadowed as  in  \S\ref{ss:ghost}; indeed we will see in Lemma~\ref{lem:adjblock} below  that the obstruction from any  class that is adjacent to both $\bE_0$ and $\bE_1$ goes through the accumulation point of the staircase and hence may well overshadow the staircase, that is, give an obstruction larger than those provided by the steps.
  Thus we need to choose the initial steps rather carefully, so that there are no such classes.
\MS

\NI {\bf A new family of staircases}

We  consider the staircase family with steps
\begin{align} \label{eq:seedEx} \bE_0&=\bigl(2;1^{\times 2};W(3,1)\bigr)\\ \notag
 \bE_1 = \bE_1(n):  &=\bE(22+10n,9+4n)\\\notag
 &=\bigl(22+10n;13+6n,9+4n,4+2n^{\times 2},1^{\times (5+2n)};W(22+10n,9+4n)\bigr)\\\notag
 \bE_k = \bE_k(n): & = \bigl(d_k(n);\btm_k(n);W(p_k(n),q_k(n))\bigr) \ \quad \mbox{ where }\\\notag
 \bE_{k+1} = \bE_{k+1}(n) &= t_n \bE_k(n) - \bE_{k-1}(n),\quad t_n: = 5 + 2n.
\end{align}

When applying the recursion, we pad the shorter vector $\btm_0=(1,1)$ with zeros so that it has the same length as $\btm_1(n)$.
By Lemma~\ref{lem:recur}, these classes  are all quasi-perfect,  since the initial two terms $\bE_0, \bE_1$ are adjacent, $t_n = |3(9+4n) - (22+10n)| \geq 2$, $p_1(n)/q_1(n)<p_0(n)/q_0(n)$, and $p_1(n)>p_0(n),\ q_1(n)>q_0(n)$.
Note that when there is no cause for ambiguity, we will simplify notation by omitting the variable $n$ as in \eqref{eq:domn}.

We consider the domain 
\begin{align}\label{eq:domn} 
\Om_n: = \Omega_{\bm{b}(n)}:=\Om(1;\bm{b}(n)):=\Om(1;\lim_{k \to \infty}\frac{\btm_k(n)}{d_k(n)}) = \Om(1;\lim_{k \to \infty}\frac{\btm_k}{d_k}).
\end{align} 
Using Corollary~\ref{cor:dom} stated below, we compute $\lim_{k \to \infty} \btm_k/d_k$ to obtain: 
\begin{align}\label{eqn:bbn}
\bm{b}(n) &=: (b_1,b_2,b_3^{\times 2},b_4^{\times (5+2n)}) \\ \notag
 &=\bigl((2+n)\be_n+1/2,-(2+n)\be_n+1/2,(4+2n)\be_n^{\times 2},\be_n^{\times (5+2n)}\bigr)
 \end{align}
where \[\be_n:=\frac{1}{17+8n+\sqrt{(3+2n)(7+2n)}}.\]
Note that $\lim_{n \to \infty}  {\bf b}(n) =  (3/5,2/5,1/5^{\times 2},0^{\times (5+2n)}).$
We now compute the following:
\begin{align} \label{eq:ident}
 \Vol(\Om_n)&=\tfrac{1}{2}-\be_n^2(45+10n^2+42n) \;\to\; 2/5\\ \notag
  \Per(\Om_n)&=2-\be_n(13+6n) \;\to\; 7/5 \\ \notag
    z_\infty&=\frac{29+14n+3\sqrt{(3+2n)(7+2n)}}{13+6n+\sqrt{(3+2n)(7+2n)}}\; \to \; 5/2 \\ \notag
    V_{{\bf b}(n)}(z_\infty)&=\frac{2(17+8n+\sqrt{(3+2n)(7+2n)})}{13+6n+\sqrt{(3+2n)(7+2n)}}\; \to\; 5/2
\end{align}

\begin{figure}[h!]
\includegraphics{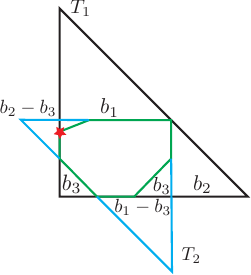}
  \caption{Here is one way to construct  a domain $\Om_n$  with weights $(1;b_1,b_2,b_3^{\times 2},b_4^{\times (5+2n)})$ and $7$ sides. This domain is geometrically Cremona equivalent to $\Omega_n'$ with weight sequence $(1-b_3;b_1-b_3,b_2-b_3,b_3,b_4^{\times (5+2n)})$. The domain $\Om_n$ (the inner green heptagon) can be seen from cutting the black triangle labeled $T_1$ of size $1$ or cutting the blue triangle labeled $T_2$ of size $1-b_3=b_1+b_2-b_3$. The unlabeled scalene triangle is made from $5+2n$ cuts of size $b_4$; its vertex labeled with a red star has determinant $2n+5$ as mentioned in Lemma~\ref{lem:cutOmegaN}.
}
\label{fig:3}
 \end{figure}

The goal of this section is to show that for sufficiently large $n$, $X_{\Omega_n}$ has an infinite staircase. 
The assumption of Theorem~\ref{thm:inftystair} is that the cut length of some $\Omega_n$ realizing the weights $(1;{\bf b}(n))$ tends to infinity with $n$. We define $\Omega_n:=\Omega_n(1;{\bf b}(n))$ to be the green convex domain illustrated in Figure~\ref{fig:3}. As we explain in Remark~\ref{rmk:cutOm}, there  are many different ways to construct a domain with weights 
$(1;{\bf b}(n))$. The cut lengths of these domains might differ, but as we explain in Remark~\ref{rmk:cutOm} however such a domain is formed, 
 the cut length of a domain with weight sequence $(1;{\bf b}(n))$ increases as $n$ increases due to Corollary ~\ref{cor:cutLength}.

\begin{lemma} \label{lem:cutOmegaN}
	The sequences given by the cut length $\{{\rm Cut}(\Omega_n)\}_{n \geq 0}$ and Cremona length $\{cr(\Omega_n)\}_{n \geq 0}$ are unbounded. 
\end{lemma}
\begin{proof}

We first compute the Cremona length. We can perform one Cremona move to $(1;\bm{b}(n))$ which results in 
\[\Omega'_n:=(1-b_3;b_1-b_3,b_2-b_3,0,b_3,b_4^{\times (5+2n)})\]
as $b_1+b_2=1$. 
After deleting the zero the vector is ordered, so the tuple is reduced, since
\[
1-b_3=(b_1-b_3)+(b_2-b_3)+b_3.
\]
Hence, the Cremona length is  $2n+8$. One construction of $\Omega'_n$ can be seen in Figure~\ref{fig:3}. If we consider the order of singularity of the vertex labeled with a red star in Figure~\ref{fig:3}, it is $\det\begin{pmatrix} 2n+5 & 0 \\ -1 & 1 \end{pmatrix}=2n+5.$ Hence, by Corollary ~\ref{cor:cutLength}, the sequence $\{{\rm Cut}(\Omega_n)\}_{n \geq 0}$ is unbounded. 
\end{proof}

\begin{rmk}\label{rmk:cutOm} \rm There are many different ways in which one could cut triangles out of $T(b)$ to construct a domain with weights $(1; \bm{b}_n)$.  For example, because $b_2 + 2b_3 + (5+2n)b_4< 1$, all the cuts after the first one could be put along the same edge.  In contrast to the realization pictured in Figure~\ref{fig:3}, this second choice  of realization is not affine equivalent to its Cremona reduction.  One could also  distribute the $2n+5$ cuts of size $b_4$   among the different edges of $\Om(1; b_1,b_2,b_3^{\times 2})$.  But, because these cuts are all of the same size and because a cut cannot be centered at a nonsmooth vertex, it is easy to check that if the domain $\Om(1; b_1,b_2, b_3^{\times 2}) $ is constructed to have $k$ vertices then however one adds the $2n+5$ cuts of size $b_4$ there has to be a vertex with order of singularity $\ge (2n+5)/k$. Thus, although these realizations might have different cut lengths, it follows from Corollary~\ref{cor:cutLength} that these cut lengths must tend to infinity with $n$.
\end{rmk}

The following proposition describes  conditions under which  a given sequence of classes $(\bE_k)_{k\ge 0}$  forms a staircase. 
Recall from  the beginning of \S\ref{ss:ghost} that given a sequence of obstructive classes $(\bE_k)_{k\ge 0}$ with centers $p_k/q_k$ converging to $z_\infty$, an
overshadowing class is a quasi-perfect class whose obstruction goes through the accumulation point $(z_\infty, V(z_\infty))$ and is larger  than the obstructions given by the $\bE_k$. Since we can always assume that such a class is live on one side of $z_\infty$,  
and since the capacity function is determined by exceptional classes, we may assume that such a class is in fact perfect.

\begin{prop} \label{prop:stair}
Let $(\bE_k: = \bE_k(n))_{k\ge 0}$ be a sequence of quasi-perfect classes generated recursively with $t_n \geq 2$ as above for some $n\ge 0$, and let $X_{\Omega_n}$ be the convex toric domain with negative weight expansion $
(1;\bm{b}(n))$ given in \eqref{eqn:bbn}. If the following conditions hold:
    \begin{itemize}
        \item[{\rm (i)}] for large enough $k,$ the functions $\mu_{\bE_k,\bm{b} (n)}(z)$ are obstructive in $X_{\Omega_n}$ at $z=p_k/q_k$,
        \item[{\rm (ii)}] the classes $\bE_k$ are perfect,
        \item[{\rm (iii)}] there is no overshadowing class, 
    \end{itemize}
    then  $X_{\Omega_n}$ has a staircase. 
\end{prop}

\begin{proof} We will see in Lemma~\ref{lem:unobs} that if $\bE_k$ is perfect, it  is live at $p_k/q_k$ for the domain
$\Om_{n,k}: = \Om(1,\bm{b}_k(n)) = (1;\frac{\widetilde{m}_{1,k}}{d_k},\hdots,\frac{\Tm_{s,k}}{d_k})$, where $s$ is the length of the vector $\bm{b}_k$, i.e. we have the strict inequality
\[ \mu_{\bE_k,\Om_{n,k}}(p_k/q_k)=c_{\Om_{n,k}}(p_k/q_k)>V_{\Om_{n,k}}(p_k/q_k).\]

If we also know that for sufficiently large $k$ (and fixed $n$), $\bE_k$ is obstructive at $p_k/q_k$ for the limiting domain $X_{\Om_n}$ where the weight sequence of $\Om_n$ is $\lim_{k \to \infty} (1;\bm{b}_k)$, it follows from Proposition~\ref{prop:accum1} that  $p_k/q_k$ converges to the accumulation point of $\Om_n$.  (Note that as $\btm_k,d_k$ satisfy the same recursion, 
the 
sequence $\bm{b}_k$
 converges, see Corollary~\ref{cor:dom0}.)  Moreover,  the accumulation point cannot be obstructed, for if it were, by continuity $\bE_k$ could not be live in $X_{\Om_{n,k}}$ at its break point $p_k/q_k$ 
 for all large $k$.
  If, for fixed n and for large $k$, the classes $\bE_k$ continue to be live in the limiting $X_{\Om_n}$, then they form the desired staircase. If this is not the case, either there is a single obstruction in $X_{\Om_n}$ that goes through the accumulation point and is larger than the obstructions from the $\bE_k = \bE_k(n), k\ge k_0$, or there is a sequence of classes 
 $\bE'_k, k\ge k_1,$ that are live for $\Om_n$ and with  obstructions  $\mu_{\bE'_k,\Om_n}$ that are maximal at $p_k/q_k$ and such that $\mu_{\bE'_k,\Om_n}(p_k/q_k) > \mu_{\bE_k,\Om_n}(p_k/q_k)$ for all large $k$.  But in the first case this single class is by definition an overshadowing class, while in the second, the classes $\bE_k'$ themselves form a staircase.
  \end{proof}

We now give the proof of Theorem~\ref{thm:inftystair}, which cites the work in the following three subsections. 
\begin{proof}[Proof of Theorem~\ref{thm:inftystair}]
	By Lemma~\ref{lem:cutOmegaN}, the domains $X_{\Omega_n} = \lim_k X_{\Omega_{n,k}}$ have increasing cut length. In the next three subsections, we check (in reverse order) that for sufficiently large $n$, the classes $\bE_k(n)$ for $X_{\Omega_n}$ satisfy the three conditions of Proposition~\ref{prop:stair}. Hence, we conclude $X_{\Omega_n}$ has a staircase for $n \geq n_0$. The proof of the first condition is given in Lemma~\ref{lem:nontrivial} and holds for all $n \geq 0$. 
	We reduce the proof that the classes $\bE_k(n)$ are perfect for all $k \geq 0$ and $n \geq 0$  to results that were already proved in \cite{BHM}. The details are in  Lemma~\ref{lem:classesPerfect}.  The last criterion is checked in Proposition~\ref{prop:OS}. This argument is rather tricky since, as explained in Remark~\ref{rmk:over}, there are similarly defined sequences of classes that are overshadowed, and we simplify it by requiring
	 that $n \geq n_0$ for some sufficiently large $n_0$.\footnote{We expect that there is no overshadowing class for all $n,$ but did not carry out the necessary computations.} 
	 The proof here is a combination of the arithmetic arguments in \cite[\S4.3]{MM} with some new estimates.
\end{proof} 

We end this section with a remark and some useful lemmas. The first three lemmas are variants of similar results proved for the Hirzebruch surface in \cite{BHM}, while the last quotes a result from \cite{MM} explaining how to compute the limit of a recursively defined sequence.

\begin{rmk} \label{rmk:obstr} Below we frequently use the fact that if $z$ is sufficiently close to the center $p/q$ of a quasi-perfect class $\bE = (d',(\Tm_j), p,q)$ in some domain $\Om=\Om(1,(b_j))$  then the corresponding obstruction  is 
\begin{align}\label{eqn:obstr}
\mu_{\bE, \Om}(z) = \begin{cases} \frac {qz}{d-\sum_j\Tm_j b_j}, & \mbox{ if } z\le p/q\\
\frac p{d-\sum_j\Tm_j b_j} & \mbox{ if } z\ge p/q.\end{cases}
\end{align}
One measure of \lq\lq sufficiently close\rq\rq\  is that there should be no point lying strictly between $p/q$ and $z$ whose continued fraction 
$[\ell_0; \ell_1,\dots,\ell_n]$ is shorter than that of $p/q$; see \cite{ball} or \cite{MMW}. Another useful fact is that,
 when $z$ is sufficiently close to the break point,  the obstruction from any exceptional class $\bE = (d, \btm, \bf{m})$ always has the form $\mu_{\Om, \bE}(z) = \frac{A + Cz}{ d-\sum \Tm_j b_j}$ for some integers $A,C\ge 0$; see \cite[Prop.2.3.2]{ball}.
\end{rmk}

\begin{lemma} \label{lem:unobs}
Let $\bE=(d;\widetilde{m}_1,\hdots,\widetilde{m}_r;W(p,q))$ be a perfect class in the domain $X_{\Om_B}$, where 
$B=(1;\frac{\Tm_1}{d},\hdots,\frac{\Tm_r}{d}).$
Then, $\mu_{\bE,\Om_B}(p/q)$ is live at $p/q$ for $X_{\Om_B}$ 
i.e. $c_{\Om_B}(p/q)=\mu_{\bE,\Om_B}(p/q)$.
\end{lemma}
\begin{proof}
    To check that $\mu_{\bE,\Om_B}(p/q)$ is live when $B=(1;\widetilde{m}_1/d,\hdots,\widetilde{m}_r/d),$ we first check that $\mu_{\bE,\Om_B}(p/q) > V_{\Om_B}(p/q).$
    We have that
    \begin{align*}
       \mu_{\bE,\Om_B}(p/q)=\frac{p}{d-\sum \Tm_j \frac{\Tm_j}{d}}=
       \frac{pd}{d^2-\sum \Tm_j^2}.
    \end{align*}   
    and
    \begin{align*}
        V_{\Om_B}(p/q)=\sqrt{\frac{p}{q(1-\sum \frac{\Tm_j^2}{d^2})}}=\sqrt{\frac{pd^2}{q(d^2-\sum \Tm_j^2)}}
    \end{align*}
    Hence, $\bE$ is obstructive at $p/q$ for the domain $\Om_B$ since
    \begin{align*}
        \mu_{\bE,\Om_B}(p/q) > V_{\Om_B}(p/q) & \iff
        \frac{pd}{d^2-\sum \Tm_j^2} > \sqrt{\frac{pd^2}{q(d^2-\sum \Tm_j^2)}} \\
     &   \iff
       pq> d^2-\sum \Tm_j^2=pq-1,
    \end{align*}
    where  we have used that $\bE$ is perfect. 
    Then, to check $\bE$ is live, we consider the obstruction from any other exceptional class $\bE':=(d';\widetilde{m}'_1,\hdots \widetilde{m}_r';\bm{m}')$. Let $\bm{m}:=W(p,q)$.  Then because\footnote
    {
    This inequality, which holds only for exceptional classes $\bE$,  is the key difference between perfect and quasi-perfect classes.} $\bE \cdot \bE' \geq 0,$ we have
    \[ dd'-\sum \widetilde{m}_j \widetilde{m}_j' \geq \sum m_j' m_j .\]  
    Therefore, for $a=p/q$ we have
    \begin{align*}
    \mu_{\bE',\Om_B}(a)=\frac{\bm{m}' \cdot w(a)}{d'-\sum \Tm_j'\frac{\Tm_j}{d}} 
 \leq \frac{d(dd'-\sum \Tm_j\Tm_j')}{q(dd'-\sum \Tm_j \Tm_j')}=\frac{d}{q}
    \end{align*}  
    where here we have used that $\bm{m}' \cdot w(a)=\frac{1}{q}\bm{m}' \cdot \bm{m}$.
On the other hand, because $w(a)\cdot w(a)=a=p/q$, we have
\[ \mu_{\bE,\Om_B}(a)=\frac{pd}{d^2-\sum \Tm_j^2} > \frac{d}{q}\]
again because $pq=d^2-\sum \Tm_j^2+1.$ 
Hence $\mu_{\bE,\Om_B}(p/q) > \mu_{\bE',\Om_B}(p/q)$.
\end{proof}

The next two lemmas will be used in the analysis of potential overshadowing classes, since such a 
class must be obstructive at its break point.

\begin{lemma} \label{lem:nontrivialInequ}
    If a Diophantine class $\bE=(d;\btm;\bm{m})$ is obstructive at its break point for the domain $B=(1;\bm{b})$, then 
     \[ \left(d^2-||\btm||^2+1\right)(1-||\bm{b}||^2) > \left(d- \btm \cdot \bm{b}\right)^2.\]
\end{lemma}
\begin{proof}
       By definition of $\mu_{\bE,\bm{b}}(z)$, we have
    \[ (d-\btm \cdot \bm b) \mu_{\bE,\bm{b}}(z)=\bm{m} \cdot \bw(z) \leq || \bm{m} ||\ || \bw(z) ||=\sqrt{z}\sqrt{d^2-||\btm||^2+1}.\]
    If $\mu_{\bE,\bm{b}}(z)$ is obstructive, then we have \[\sqrt{z}\frac{\sqrt{d^2-||\btm||^2+1}}{d-\btm \cdot b}\geq \mu_{\bE,\bm{b}}(z) > V_{\bm{b}}(z),\] which implies that
    \[\frac{\sqrt{d^2-||\btm||^2+1}}{d-\btm \cdot b}> \frac{1}{\sqrt{1-||\bm{b}||^2}} .\]
    Hence
    \[ \left(d^2-||\btm||^2+1\right)(1-||\bm{b}||^2)> \left(d- \btm \cdot \bm{b}\right)^2\]
    as claimed.
\end{proof}

\begin{lemma} \label{lem:sqrt2} Let $\bE=\bE(d; \btm; \bbm)$ be a Diophantine class that is obstructive at its  break point $z$  for $\Om = \Om(b; (b_j))$. Denote the tuple $(b_1,b_2,\dots)$ by  $\bm{b}$, and write $\btm = M\bm{b} + \bm{\eps}$ where $M\in \R$ and $\bm{\eps}\cdot \bm{b} = 0$. Then, we have $(M-d)^2 \le \frac 1{ \|\bm{b}\|^2} - 1$.
\end{lemma}

\begin{proof} As $\bE$ is obstructive, by Lemma~\ref{lem:nontrivialInequ} we know that 
$$
(d-\btm \cdot\bm{b})^2 \le (1-\|\bm{b} \|^2) (d^2 - \|\btm\|^2 + 1),
$$
Since $\btm\cdot\bm{b}= M\|\bm{b}\|^2$,  
 and $ \|\btm\|^2\ge M^2\|\bm{b}\|^2$, this implies that
$$
(d- M\|\bm{b}\|^2)^2 \le (1-\|\bm{b}\|^2) (d^2 - M^2\|\bm{b}\|^2 + 1),
$$
which expands to
$$
d^2 - 2dM \|\bm{b}\|^2 + M^2 \|\bm{b}\|^4  \le d^2 - M^2\|\bm{b}\|^2 + 1  -\|\bm{b}\|^2 d^2 + M^2\|\bm{b}\|^4 - \|\bm{b}\|^2.
 $$
This simplifies to
$$
 M^2\|\bm{b}\|^2   - 2dM \|\bm{b}\|^2 + \|\bm{b}\|^2 d^2 \le 1 -  \|\bm{b}\|^2,
 $$
 or equivalently 
 $
 (M-d)^2 \le \frac 1{ \|\bm{b}\|^2} - 1.
 $
\end{proof} 

This next lemma is  \cite[Lemma 3.1.4]{MM}; it is the basis of many calculations.
\begin{lemma} \label{lem:recur1}
    Let $x_k,k \geq 0$ be a sequence of integers that satisfy the recursion
    \[ x_{k+1}=t x_k-x_{k-1}, \quad t \geq 3,\]
    and let 
    \[ \la=\frac{t+\sqrt{\si}}{2}\] where $\si=t^2-4.$ Then there is a number $X \in Q[\sqrt{\si}]$ such that 
    \begin{equation} \label{eq:recurSln}
         x_k=X\la^k+\ov{X}\ \ov{\la}^k,
    \end{equation}
    where $\ov{a+b\sqrt{\si}}:=a-b\sqrt{\si}$, so that $\la \ \ov{\la}=1.$ Further, if we write $X=X'+X''\sqrt{\si},$ then
    \begin{equation} \label{eq:bigX}
        X'=\frac{x_0}{2}, \quad X''=\frac{2x_1-tx_0}{2\si}.
    \end{equation}
\end{lemma}

\begin{cor}\label{cor:dom0} Suppose that  the  sequences $(a_k), (b_k)$  are defined recursively as in Lemma~\ref{lem:recur1} using the same parameter $t$, and define $A =A'+A''\sqrt{\si}, B = B'+B''\sqrt{\si}$ as above. Then 
$$
\lim_k\frac{a_k}{b_k} = \frac AB, \quad \lim_k a_kB - A b_k = \ov A B - A\ov B.
$$
\end{cor}

\begin{cor}\label{cor:dom}  The domain $\Om_n$ defined in \eqref{eq:domn} has coefficients $\bm{b}(n) = (B_{n1},\dots,B_{nK})$,
where $B_{nj} = M_{nj}/D_n$. Here $M_{nj}$, respectively $D_n$, is the number $X$ defined as above for the recursive sequence $m_{jk}(n), k\ge 0$, respectively $d_{k}(n), k\ge 0$.
\end{cor}

\subsection{There is no overshadowing class}\label{ss:oversh}
The staircases constructed in \cite{MMW} are not overshadowed basically for arithmetic reasons:
we showed that, in the  situation  considered there, there are  two lines that always go through the limit point
$(z_\infty, V_\Om(z_\infty))$. This was enough to allow one to rule out the existence of a third line through this point of the form corresponding to an obstructive class.  In the current situation, such  arguments do not quite suffice, but they do give some useful information. Note that, as mentioned just before Proposition~\ref{prop:stair}, we may assume that the overshadowing class is perfect.

\begin{lemma} \label{lem:perLine} Let $X_\Omega$ be a
 convex toric domain with perimeter $\Per: = \Per(\Om)$ and volume $\Vol: = \Vol(\Om)$. Then
 the line $\frac{1+z}{\Per(\Om)}$ goes through the limit  point $(z_\infty,V_{\Omega}(z_\infty)).$
\end{lemma}
\begin{proof} This holds because
    \begin{align*}\frac{1+z_\infty}{\Per}=\sqrt{\frac{z_\infty}{\Vol}} 
    &\iff (1+z_\infty)^2 =z_\infty\frac{\Per^2}{\Vol} \\
    & \iff z_\infty^2-(\Per^2/\Vol-2)z_\infty+1 =0,
    \end{align*}
where   the last identity  holds by definition of $z_\infty$.
\end{proof}

The following result generalizes the notion of \lq\lq blocking class'' from \cite{MM,MMW}.

\begin{lemma}\label{lem:adjblock}  Let $(\bE_k)_{k\ge 0}$ be a recursively defined staircase in $X_\Om$ with accumulation point $z_\infty$, and suppose
that $\bE': = \bigl(d'; (\Tm_j'); W(p',q')\bigr)$ is a quasi-perfect class whose center $p'/q'$  is either larger or smaller than all the centers $p_k/q_k, k\ge 0$.  Then if the class $\bE'$ is adjacent to the first two steps $\bE_0, \bE_1$, the corresponding obstruction $\mu_{\Om,\bE'}$ goes through the limit point $(z_\infty, V_\Om(z_\infty))$.
\end{lemma}

\begin{proof} Suppose first that $p'/q' < p_k/q_k$ for all $k$, and denote by $D, B_j,P, Q$ the appropriate quantity $X$ defined in
Lemma~\ref{lem:recur1}  for each of the recursive sequences $x_k: = d_k,\widetilde{m}_{jk},p_k,q_k $. 
Because the adjacency relation \eqref{eq:adj} is bilinear,  we have
$$
D(d' - \sum \Tm_j'B_j) = p'Q.
$$
But by \eqref{eqn:obstr}, $\mu_{\Om,\bE'}(z_\infty) = \frac {p'}{d' - \sum \Tm_j'B_j} = D/Q$.  On the other hand, 
\begin{align*}
&z_\infty = P/Q,\quad  \Vol(\Om) = 1-\sum B_j^2,\;\;\mbox{and } \\
&\quad 
V_\Om(z_\infty) = \sqrt{\frac P{Q(1-\sum B_j^2)}} = D/Q,
\end{align*}
 where the last equality holds because the   Diophantine identity
$d_k^2 - \sum_j \Tm_j^2 = p_kq_k - 1$ implies that $D^2 (1-\sum B_j^2) = PQ$.

The proof of the  second case (with
 $p'/q' > p_k/q_k$ for all $k$) is a very similar argument  and is left to the reader.
\end{proof}

If $(\bE_k)_{k\ge 0}$ is a descending pre-staircase as in our situation, then a class $\bE'$ as above with $p'/q'< z_\infty$ is called an 
{\bf ascending blocking class}, while a similar class with  $p'/q'> z_\infty$ is a potential overshadowing class.

\begin{lemma}
\label{lem:os1}
    Let $\bE'=(d';\btm';W(p',q'))$ be an ascending perfect blocking class for a pre-staircase in $X_\Om$, and let $\bE=(d;\btm;{\bf m})$ be a (not necessarily perfect) descending overshadowing class with obstruction function $\frac{A+Cz}{\la}$. Then, 
    \begin{equation} \label{eq:OScondition}
         (C-A)\la'+\la p'=Cp' \Per
     \end{equation} 
    where $\la'=d'-\btm' \cdot \bm{b}, \la=d-\btm \cdot \bm{b}$, and $\Per = \Per(\Om).$
\end{lemma}
\begin{proof}
  At the accumulation point, we know from Lemmas~\ref{lem:perLine}, ~\ref{lem:adjblock} and Remark~\ref{rmk:obstr} that the following quantities are all equal to the volume at the accumulation point: 
    \[ \frac{1+z_\infty}{\Per}\;=\; \frac{p'}{\la'}\;=\;\frac{A+Cz_\infty}{\la}.\]
Solving the first equality for $z_\infty$ gives
\[
z_\infty=\frac{p'\Per-\lambda'}{\lambda'}.
\]
Then substituting this expression for $z_\infty$ into the second equality
gives
\[
\lambda p'
=
\lambda'(A+Cz_\infty)
=
A\lambda'+C(p'\Per-\lambda').
\]
Rearranging yields
\[
(C-A)\lambda'+\lambda p'=Cp'\Per.
\]
The lemma follows.
\end{proof}

Assume that $\bE=(d;\btm;\bm{m})$ is an obstructive class for domain $(1;b_1,\hdots,b_r)$ with  break point $a=p/q$. Let $\la_a=V_{\bm{b}}(a)$. As in \eqref{eq:eps}, write:
\[ (\btm,\bm{m})=\frac{d}{\la_a}(\la_a b_1,\hdots,\la_a b_r,a_1,\hdots,a_N)+\bm{\eps}\]
where $(a_1,\hdots,a_N)$ is the weight expansion of $a=p/q$.
As $\bE$ is obstructive, we can conclude that $\eps\cdot \eps<1$ by \eqref{eq:eps2}. 
This observation is the crucial reason why the following lemma holds.

\begin{lemma} \label{lem:btmD}
    Assume $\bE=(d;\btm;\bm{m})$ is an obstructive class for $(1;b_1,\hdots,b_r).$ Writing $\btm=(\widetilde{m}_1,\hdots,\widetilde{m}_r,\hdots)$. Then, 
    \[ \widetilde{m}_i \in \{\lfloor db_i \rfloor, \lceil db_i \rceil \}\]
    Further, let $J:=\{k,\hdots,k+s-1\}$ be a block of $s \geq 2$ consecutive integers for which the $b_i$ are constant for $i \in J$. Then, $\Tm_i$ is constant on $i \in J$ except for at most one entry.
\end{lemma}
\begin{proof}
 The proof, which exploits  the fact that $\eps\cdot \eps<1$,  is  the same as that in \cite[Lemma 2.1.7]{ball}, but with the vector $\btm$ replaced by $\bf{m}$.
    \end{proof}

With these preliminaries in place, we are now ready to prove the main result of this subsection.

\begin{prop}\label{prop:OS} For sufficiently large $n$, the pre-staircase  $\bigl(\bE_k\bigr) : = \bigl(\bE_k(n)\bigr)$  defined in \eqref{eq:seedEx} has no overshadowing class $\bE$.
\end{prop}
\begin{proof}
The proof is quite complex and involves several steps. Since as remarked before Proposition~\ref{prop:stair} a maximal overshadowing class is perfect, we will assume that $\bE = (d; \btm; \bf m)$ is a maximal, hence perfect, overshadowing class
with obstruction function $\frac{A + Cz}{\la}$, and then show that $A=0$ and that $C,\la$ must satisfy some contradictory conditions.
\MS

\NI {\bf Step 1:} {\it The class $\bE'=\bE(2,1)=(2;1,1,1;W(2,1))$  is an  ascending blocking class for all $n$, with
$\la' = 1-2\be_n(2+n)$.}
\MS

\NI
{\it Proof:}
The first claim holds because $\bE'$ 
 is adjacent to $\bE_0$ and $\bE_1(n)$  and $p'/q'= 2/1$ lies below the centers of all the classes as $\frac{5}{2}<p_k/q_k \leq 3$ for all $k$.
 The calculation of \[\la' = 2- (b_1+b_2+b_3)=1-2\be_n(2+n)\] follows from equations \eqref{eqn:bbn} and  \eqref{eq:ident}.
 \MS\MS
 
\NI {\bf Step 2:} {\it  The integers $A,C$ must satisfy the following identities:}
\begin{align} \label{eq:arith}
    3C+A&=2d-\widetilde{m}_1-\widetilde{m}_2 \\ \notag
    C(11+5n)+A(2+n)&=(2+n)\widetilde{m}_1-(2+n)\widetilde{m}_2+(4+2n)L+K,
\end{align}
where $L: = \widetilde{m}_3 + \widetilde{m}_4$ and $K:=\widetilde{m}_5+\hdots +\widetilde{m}_{9+2n}$.
\MS

 \NI
 {\it Proof:} The identity
 \eqref{eq:OScondition} in Lemma~{lem:os1} implies that 
 \begin{equation} \label{eq:firstArith}
    (C-A)(1-2\be_n(2+n))+2\la=2C(2-\be_n(13+6n)). 
\end{equation} 
Since $\la = d - \Tm\cdot\bm{b}$, we have
\begin{align*} \la&=d-((2+n)\be_n+\tfrac12)\widetilde{m}_1-(-(2+n)\be_n+\tfrac12)\widetilde{m}_2-(4+2n)\be_nL-\be_nK \\
&=d-1/2(\widetilde{m}_1+\widetilde{m}_2)-\be_n\left((2+n)\widetilde{m}_1-(2+n)\widetilde{m}_2+(4+2n)L+K\right).
\end{align*}

After substituting the above expression for $\lambda$ into \eqref{eq:firstArith},
we obtain an equation of the form $c_1 + c_2\be_n = 0$ where $c_1, c_2$ are integers. Because  $\beta_n$ is irrational,
we must have  $c_1 = c_2=0$, which immediately gives 
\eqref{eq:arith}.
\MS\MS

\NI {\bf Step 3:} {\it  Assuming $n \geq 3$, the following bounds hold:}
\begin{align} \label{eq:arithBounds}
     d-2& \leq 3C+A \leq d+2,\\
\label{eq:ClaAlabound}
   C/\la &\geq      \frac{5+2n+\sqrt{(3+2n)(7+2n)}}{10+4n}  \\ \label{eq:Alabound}
        0 \leq A/\la &\leq \frac{3}{2}-\frac{3\sqrt{(3+2n)(7+2n)}}{10+4n}\\
        \label{eq:CoverA}
          7 &\leq  \tfrac{1}{12}\bigl(5+2n+\sqrt{(3+2n)(7+2n)}\bigr) \leq C/A
     \end{align} 
     {\it where the last inequality assumes $A \neq 0.$}
\MS

\NI {\it Proof:} To obtain the bound \eqref{eq:arithBounds} for
 $3C+A$ we note that  
$ \widetilde{m}_i \in \{\lfloor db_i \rfloor, \lceil db_i \rceil \}$ by Lemma~\ref{lem:btmD}.  
Because $b_1+b_2=1,$
we find that  $ d-2 \leq \widetilde{m}_1+\widetilde{m}_2 \leq d+2$. Now apply
the first identity in 
\eqref{eq:arith}.
\MS

To bound $A/\la,C/\la$, we consider the lines $\ell_k(z):=s_k z+r_k$ going from the point $(z_\infty,V_{\bm{b}}(z_\infty))$ to the peak $(p_k/q_k,\mu_{\bE_k,\bm{b}}(p_k/q_k))$ of the $k$th obstruction.  
Define $\la_k=d_k-\btm_k \cdot \bm{b}$ and, in the notation of Lemma~\ref{lem:recur1}, set
$\Lambda=D-\widetilde{\bm{M}} \cdot \bm{b}$, where $\bm{\widetilde{M}}=(\widetilde{M}_1,\widetilde{M}_2,\widetilde{M}_3^{\times 2},\widetilde{M}_4^{\times 2n+5})$. Then 
 the line $\ell_k$ goes through the points $(p_k/q_k,p_k/\la_k)$ and $(P/Q,P/\Lambda),$ and so has  slope 
\[ s_k=\frac{
\frac{p_k}{\lambda_k}-\frac{P}{\Lambda}
}{
\frac{p_k}{q_k}-\frac{P}{Q}
}
=
\frac{
q_kQ\bigl(p_k\Lambda-P\lambda_k\bigr)
}{
\lambda_k\Lambda(p_kQ-q_kP).
}
\]
By Corollary~\ref{cor:dom0}, these slopes have limit
\begin{align}
    s_\infty&:=\frac{Q^2}{(D-\bm{\widetilde{M}}\cdot \bm{b})^2}\frac{\ov{P}(D-\bm{\widetilde{M}}\cdot \bm{b})-P(\ov{D}-\ov{\bm{\widetilde{M}}} \cdot \bm{b})
    }{\ov{P}Q-\ov{Q}P} \notag \\
    &=\frac{40n^3+258n^2+565n+422+(80+79n+20n^2)\sqrt{(3+2n)(7+2n)}}{(5+2n)(40n^2+158n+163)} \label{eq:sInf}.
\end{align}
where we used Lemma~\ref{lem:recur1} and 
\eqref{eq:seedEx} to directly compute $s_\infty$. 
If $s_\infty>C/\la,$ then there is no overshadowing class as it would not obstruct the peaks of $\mu_{\bE_k,\bm{b}}$ for large $k.$ Hence, we must have 
\[ s_\infty  \leq C/\la.\]
We now compare the line $s_\infty$ to the line $s_0.$ In particular, since $\bE_0 = (2;1,1; W(3,1))$ the line $\ell_0$ has equation
\[
y = \frac{3-V_{\bm{b}}(z_\infty)}{3-z_\infty}(z-z_\infty)+V_{\bm{b}}(z_\infty).\]
 By substituting the exact values listed in  \eqref{eq:ident}, we have that 
 \[ s_0= \frac{5+2n+\sqrt{(3+2n)(7+2n)}}{10+4n}.\] 
 By \eqref{eq:sInf}, for a fixed $n$, a direct computation gives $s_0 < s_\infty$, so $s_0 < s_\infty \leq C/\la$, which is seen in \eqref{eq:ClaAlabound}.
Further as the line $\ell_0(z)$ and the overshadowing class intersect at the point $(z_\infty,V_{\bm b}(z_\infty))$ and $s_0<C/\la$, we must also have 
\[ A/\la \leq r_0=\ V_{\bm{b}}(z_\infty)-z_\infty\frac{3-V_{\bm{b}}(z_\infty)}{3-z_\infty}.\] This is the upper bound for $A/\la$ in \eqref{eq:Alabound} obtained by substituting the exact values listed in  \eqref{eq:ident}. The lower bound holds because $A\ge 0$ by the subscaling property of capacity functions; see \cite[Lem.1.1.1]{ball}. (Here we use the fact that the capacity function in some interval $[z_\infty,z_\infty + \eps)$ is given by $\bE$ since this class is assumed to be live.) 
 Finally, \eqref{eq:CoverA} follows from  the bounds in \eqref{eq:ClaAlabound} and evaluating at $n = 3$.


\MS\MS

\NI {\bf Step 4:} {\it  If $\bE$ overshadows the $n$th staircase for some $n\ge 3$ then $d\le 18$ and $A = 0$.}
\MS

\NI {\it Proof:} As usual $d,\btm,\bm{b}$ depend on $n$, which we suppress for ease of notation except for $\bm{b}(n).$   We first give a lower bound for $\lambda=d-\widetilde{\bm m} \cdot \bm{b}(n)$. 
We have
     \[ ||\bm b(n)||^2=1-\Vol=\tfrac{1}{2}+\be_n^2(45+10n^2+42n).\]
     The quantity $||\bm b(n)||^2$ increases for $n \geq 1$, so we have
     \begin{equation} \label{eq:bBound}
     0.596\leq ||\bm b(1)||^2 \leq ||\bm b(n)||^2 \leq 0.6 
     \end{equation}
     As in Lemma~\ref{lem:sqrt2}, write
\[
        \widetilde{\bm m}=M\bm b(n)+\bm\eps,
        \qquad
        \bm\eps\cdot\bm b(n)=0
\]
where Lemma~\ref{lem:sqrt2} and \eqref{eq:bBound} gives 
\[
        |M-d|\leq \sqrt{\frac{1}{\|\bm b(n)\|^2}-1}\leq \sqrt{\frac{1}{\|\bm b(1)\|^2}-1}\leq  .823.
\]
Therefore
\begin{align*}
        \lambda
        &=
        d-\widetilde{\bm m}\cdot\bm b(n)\\
        &=
        d-M\|\bm b(n)\|^2\\
        &\geq
        d-(d+.823)\|\bm b(n)\|^2\\
        &\geq
        .4d-.4938
\end{align*}
where the last inequality uses the upper bound in \eqref{eq:bBound}.
Combining this with the lower bound for $C/\lambda$ from Step 3 gives
\[
        \frac{5+2n+\sqrt{(3+2n)(7+2n)}}{10+4n}
        \bigl(.4d-.4938\bigr) \leq C.
\]
Combining this inequality with \eqref{eq:arithBounds} and using $A\geq0$, we obtain
\[
        \frac{3(5+2n+\sqrt{(3+2n)(7+2n))}}{10+4n}
        \bigl(.4d-.4938\bigr)
        \leq d+2.
\]
The function
\[
        n\mapsto \frac{5+2n+\sqrt{(3+2n)(7+2n)}}{10+4n}
\]
is increasing in $n$. Therefore, for $n \geq 3,$ we have
\[
        \frac{3(11+\sqrt{117})}{22}\bigl(.4d-.4938\bigr)\leq d+2.
\]
In particular,
\[
        1.1899d-1.482\leq d+2.
\]
It follows that $d\leq18$ as $d$ is an integer. Finally, \eqref{eq:arithBounds} gives
\[
        3C+A\leq d+2\leq20.
\]
In particular, $C\leq6$ since $C,A$ are nonnegative integers.  On the other hand, \eqref{eq:CoverA} shows that if $A>0$, then
\[
        \frac{C}{A}\geq7.
\]
This would imply $C\geq7$, a contradiction.
Therefore $A=0$.

 \MS\MS     

   \NI {\bf Step 5} {\it Completion of the proof.}
   Because the degree $d$ of $\bE$ is at most $18$ there are only finitely many possibilities for $\bE$.
    Therefore, either there is no overshadowing class for sufficiently large $n$ or there is a class $\bE$ that overshadows the staircase $\Ss_n$ for arbitrarily large  $n$.  In the latter case, the coefficient of $n$ in the second equation in \eqref{eq:arith} must vanish, which, because $A = 0$, gives the  two linear relations,
    \begin{align}\label{eq:arith0}
    11C  &= 2\widetilde m_1 - 2 \widetilde m_2 + 4L + K,\\ \notag
    5C & = \widetilde m_1 -  \widetilde m_2 + 2L.
    \end{align}
    We claim that $K=0$ or $K=1$. Indeed, the final block of $\bm b(n)$ has length
$2n+5$ and all entries equal to $\beta_n$. By Lemma~\ref{lem:btmD}, the corresponding
entries of $\btm$ are constant except possibly for one entry. Since $d\leq18$ and
$\beta_n\to0$, we have $d\beta_n<1$ for all sufficiently large $n$. Hence every entry
in this block is either $0$ or $1$. The common value cannot be $1$, since the class
$\bE$ is fixed while the block length tends to infinity. Thus, the common value is $0$,
and at most one exceptional entry can be equal to $1$. Therefore $K=0$ or $K=1$.

Note that if $K=0,$ the right hand side of the first equation in \eqref{eq:arith0} is twice the right hand side of the second implying that $C=0$. Hence, $\bE$ cannot exist, since  $A$ and $C$ cannot both be zero. 


    It remains to rule out the case $K=1$. In this case, subtracting twice the second
equation in \eqref{eq:arith0} from the first gives $C=1$. Since $A=0$, the first
equation in \eqref{eq:arith} gives
\[
        \widetilde m_1+\widetilde m_2=2d-3.
\]
Combining this with the second equation in \eqref{eq:arith0}, we obtain
\begin{equation} \label{eq:Kis1}
    \widetilde m_1=d+1-L,\qquad
        \widetilde m_2=d+L-4.
\end{equation}
%
Since the exceptional class $\bE$ has nonnegative intersection with the exceptional class $L-E_1-E_2$, it follows from equation \eqref{eq:Kis1} that
\[
        d-\widetilde m_1-\widetilde m_2=3-d\geq0.
\]
Hence $d\leq3$.
Since by assumption the class $\bE$ gives the maximal obstruction, we can assume that the $\widetilde m_i$ are nonincreasing. 
 Since $\widetilde m_1,\widetilde m_2,L$ are all nonnegative integers, equation \eqref{eq:Kis1} implies that the only possibilities are
\[ (d;\widetilde m_1, \widetilde m_2,L)\in \{(2;1,0,2),(3;3,0,1),(3;2,1,2)\}.\]
The first and second are not possible as $\widetilde{m}_2=0$ implies $\widetilde{m}_3=\widetilde{m}_4=0$, so $L=0$ due to the nonincreasing assumption. For the third case, the ordering forces $\widetilde{m}_3=\widetilde{m}_4=1$ and since $K=1$, we must have $\widetilde{m}_5=1.$ The Diophantine equations eliminate this possibility as well, which completes the proof.
   \end{proof}

\begin{rmk}\label{rmk:over} \rm In view of Lemma~\ref{lem:adjblock}, when choosing the initial steps for potential staircases, one must be careful to avoid 
choosing classes that are all adjacent to the same class. Thus we picked $\bE(22,9)$ to be a component of the first step $\bE_1(n)$
because it is not adjacent to  the class $\bE_0': = (3;2,1,1,1; W(3,1))$. Notice that  $\bE_0'$ is  adjacent to $\bE_0$ and to $\bE(5,2) = (5;3,2,1,1,1; 2,2,1,1)$. But because it is not  adjacent to $\bE(22,9)$ it does not overshadow the staircases. 
 The class
$\bE(5,2)$ would also be a potentially overshadowing class  if it were adjacent to $\bE(22,9)$.
This is one reason why we chose to use the rather complicated class  $\bE(22,9)$ as a component of $\bE_1(n)$.
\end{rmk}

\subsection{The classes are perfect} \label{ss:perf}

We now show that the classes $\bE_k(n)$ are perfect, which is condition (ii) in Proposition~\ref{prop:stair}.
We use the following result, which is explained for example in \cite[Prop 1.2.12]{ball}. 
Also recall from Definition~\ref{def:Cremona} that a {\bf Cremona move} is the composite of a Cremona transformation with a permutation. 

\begin{lemma}
    An ordered integral class $\bE:=dL-\sum_{i=1}^N n_iE_i$ in $H_2(\CP^2 \# N\oCP^2)$ represents an exceptional divisor if and only if it may be reduced to $E_1$ by a repeated application of Cremona moves. 
\end{lemma}

It is convenient for our purposes here to keep track of  the initial order of the entries in the tuple $\bE_k(n)$. Therefore
we consider the Cremona transformation $c_{i,j,k}$ given by
\begin{align} c_{i,j,k}(dL-\sum_{i=1} n_iE_i)&=\widetilde{d}L-\sum_{i=1} \widetilde{n}_1E_i, \quad \text{with }  \begin{cases} \widetilde{d}= d+\delta_{ijk} \\ 
\widetilde{n}_\ell =n_\ell+\delta_{ijk} & \text{if $\ell=i,j,k$}\\
\widetilde{n}_\ell=n_\ell & \text{if $\ell \neq i,j,k$}
\end{cases}  \\ \notag
&\mbox{ where } \ \delta_{ijk}=d-n_i-n_j-n_k.
\end{align}
Further, we call 
 $\delta_{ijk}$  the \textbf{defect}. 
 Below, we write this transformation in coordinates as $$
 c_{i,j,k}(d;n_1,\hdots,n_N)=(\widetilde{d};\widetilde{n}_1,\hdots,\widetilde{n}_N),
 $$
  and the reordering operation reorders any of the $\widetilde{n}_\ell.$
We say two vectors are {\bf Cremona equivalent} if one can be obtained from the other via a series of Cremona moves. As Cremona moves are reversible, to verify some $\bE$ is exceptional, it suffices to show that $\bE$ is Cremona equivalent to some other exceptional $\bE'$.  

 In this section, we will show that for each $n \geq 0$ and $k \geq 2$, the class $\bE_k(n)$ is Cremona equivalent to a particular class that has been previously shown to reduce to $(0;-1).$ Hence, we can conclude $\bE_k(n)$ is perfect.

We first establish some properties relevant to the weight decomposition $W(p_k(n),q_k(n)).$
The paper \cite{BHM} studied a related set of classes 
\[\bB_k(n):=(\ov{d}_k(n);\ov{m}_k(n);W(\ov{p}_k(n),\ov{q}_k(n))) \]
for the one point blowup of $\C P^2$.  (Thus, here $\ov{m}_k(n)$ is a single number rather than a tuple.)  
After reindexing the family from \cite{BHM}, these
classes are defined recursively from initial seeds $$
\ov{p}_2/\ov{q}_2 = (2n+6)/1, \quad \ov{p}_3/\ov{q}_3 = (2n+7)/(2n+4)
$$ 
with the recursion variable $t = 2n+5$.
In \cite[Theorem 56]{BHM}, it was shown\footnote{The paper \cite{BHM} uses the notation $\Ss^U_{\ell,n}$ to refer to the family of classes $\bB_k(n)$, and it proved that these classes are perfect by reducing them via Cremona transformations. However, one could also deduce this from the results of \cite{McSie}, using the fact (proved in \cite{MMW}) that the centers of these curves correspond to nodal rays in suitable almost toric models for the  one point blowup of $\C P^2$.} that for $ k \geq 2$, the continued fractions of the centers of $\bB_k(n)$ are 
\begin{equation} \label{eq:CFBnk}
\ov{p}_k(n)/\ov{q}_k(n)=\begin{cases}
  [\{2n+7,2n+3\}^{\times \lfloor(k-2)/2\rfloor},2n+7,2n+4] & \text{if $k$ is odd} \\
[\{2n+7,2n+3\}^{\times\lfloor (k-2)/2\rfloor},2n+6] & \text{if $k$ is even},
    \end{cases}
\end{equation}
where the notation $[\{i,j\}^{\times k}]$ denotes the tuple $ [i,j,\dots,i,j]$ with $k$ repetitions and thus $2k$ entries.

\begin{lemma} \label{lem:matrix} Define \[ A_n:= \begin{pmatrix} 2 & 1 \\ 1 & 0 \end{pmatrix}^2 \begin{pmatrix} 2n+3 & 1 \\ 1& 0 \end{pmatrix} 
=\begin{pmatrix} 17+10n & 5 \\ 7+4n & 2 \end{pmatrix}\] 
Then for $k \geq 2$, the
 following identity holds:
  \[ A_n\begin{pmatrix} \ov{p}_k(n) \\ \ov{q}_k(n) \end{pmatrix}=\begin{pmatrix} p_k(n) \\ q_k(n) \end{pmatrix}\]
\end{lemma}
\begin{proof}
This identity holds when $k=2,3$ and hence holds for all $k \geq 2$ since  the quantities $p_k(n),q_k(n), \ov{p}_k(n), \ov{q}_k(n)$
evolve by the same linear recursion.
\end{proof}

\begin{cor}  \label{cor:CFEnk}
    The continued fraction expansion of the centers of $\bE_k(n)$ are as follows:
    \begin{align*} p_0(n)/q_0(n)=[3], \quad p_1(n)/q_1(n)=[2;2,2n+4], 
    \end{align*}
    and for $k \geq 2:$
    \[  p_k(n)/q_k(n)=\begin{cases} 
    [2;2,2n+3,\{2n+7,2n+3\}^{\times \lfloor(k-2)/2\rfloor},2n+6] & \text{if $k$ is even} \\
   [2;2,2n+3,\{2n+7,2n+3\}^{\times \lfloor (k-2)/2\rfloor},2n+7,2n+4] & \text{if $k$ is odd} \\
    \end{cases}
    \]
\end{cor}
\begin{proof}
    This can be seen from computing the continued fractions of $p_k/q_k$ for $k=0,1$ and then 
for $k \geq 2$ using the fact that $A_n$ is the composite of the three  linear transformations that take  the fraction $p'/q': = [\ell_0;\ell_1,\dots\ell_r]$
(thought of as a vector  $(p',q')$) to the fraction $p/q = [2;2,2n+3,\ell_0,\ell_1,\dots\ell_r]$.
\end{proof}

The next lemma collects some identities about $\bB_k(n)$ and $\bE_k(n).$ Here, we drop the $k$ and $n$ subscript assuming they are constant. 
\begin{lemma} \label{lem:cremonaIdent}
    The classes $\bB_k(n)$ satisfy the following identities:
      \begin{align*} 
      (2n+5)\ov{d}&=(2+n)\ov{p}+(3+n)\ov{q} \\
      (2n+5)\ov{m}&=(1+n)\ov{p}+(4+n)\ov{q}.
  \end{align*}
     The classes $\bE_k(n)$ satisfy the following identities
    \[ \btm=(d-q,q,(d-2q)^{\times 2},(5q-d-p)^{\times (2n+5)})\] and 
   \[d(5+2n)+2(2+n)p-2(11+5n)q=0.\]
\end{lemma}
\begin{proof}
    The identities can be checked for $\bB_k(n)$ and $\bE_k(n)$ with $k=0,1.$ As the identities are linear, they follow by induction since $\bB_k$ and $\bE_k$ are defined recursively. 
\end{proof}

We show below that $\bE_k(n)$ is Cremona equivalent to $\bB_k(n)$. To conclude that $\bE_k(n)$ is perfect, we use the following result from \cite{BHM}:
\begin{lemma}{\cite[Proposition 79]{BHM}} \label{lem:bBperfect}
    The classes $\bB_{k}(n)$ are exceptional classes. 
\end{lemma}

\begin{lemma} \label{lem:classesPerfect}
    For $k \geq 2$ and all $n$, the class $\bE_k(n)$ reduces to the class $\bB_k(n)$. In particular, the classes $\bE_k(n)$ are perfect. 
\end{lemma}
\begin{proof}
In this proof, we continually rearrange the classes $\bE,$ so we will only use a semicolon to distinguish the first element $d$ of the tuple $\bE$. Further, we will simplify notation by omitting the decorations by $n$ and $k$. 
Corollary~\ref{cor:CFEnk} and equation \eqref{eq:CFBnk} together explain the relation between  the continued fractions of $p/q$ and $\ov{p}/\ov{q}$, and imply  that 
\[ \bE=\bigl(d;\btm, q^{\times 2},(p-2q)^{\times 2}, (5q-2p)^{\times (2n+3)}, W(\ov{p},\ov{q})\bigr)
\]
since $\ov{p} = 5q-2p$ and $ \ov{q} = (p-2q) - (2n+3)(5q-2p)$.
We aim to reduce $\bE$ to the class
\[ \bigl(\ov{d};\ov{{m}};W(\ov{p},\ov{q})\bigr).\]
Since the last entries in these tuples coincide, it
suffices to show that $$
v: = \bigl(d;\btm,q^{\times 2},(p-2q)^{\times 2},(5q-2p)^{\times (2n+3)}\bigr)
$$ can be reduced to $(\ov{d};\ov{{m}}).$

By substituting for $\btm$ using the formula in  Lemma~\ref{lem:cremonaIdent}, we can rearrange $v$  to
\[ v_0:=\bigl(d;\ d-q,q^{\times 3},(d-2q)^{\times 2},(p-2q)^{\times 2},(5q-2p)^{\times (2n+3)},(5q-d-p)^{\times (2n+5)}\bigr).\]
We compute that 
\[ c_{4,6,7}c_{1,4,5}c_{1,2,3}(v_0)=\bigl(-p+3q;\ 0^{\times 3},5q-d-p,0,d-p,0,p-2q,(5q-2p)^{\times (2n+3)},(5q-d-p)^{\times (2n+5)}\bigr).\]
By rearranging this and dropping the zeros, we find that $v_0$ is Cremona equivalent to
\[ v_1:=\bigl(3q-p;\ p-2q,(5q-2p)^{\times (2n+3)},(5q-d-p)^{\times (2n+6)},d-p\bigr).\]
We now take a composition of $n+1$ Cremona moves 
\[c_{1,2n+2,2n+3} \hdots c_{1,4,5}c_{1,2,3}(v_1).\]
Each of these has defect $(3q-p)-(p-2q)-2(5q-2p)=2p-5q,$ and one can check that  the result of applying these $n+1$ moves to $v_1$ is:
\begin{align*} \bigl(3q-p+(n+1)(2p-5q);\ &p-2q+(n+1)(2p-5q), 0^{\times (2n+2)},\\ &\qquad (5q-2p)^{\times (2n+3)},(5q-d-p)^{\times (2n+6)},d-p\bigr).\end{align*}
Simplifying and rearranging, we get that $v_1$ is Cremona equivalent to
\[v_2:=\bigl((1+2n)p-(2+5n)q;\ (3+2n)p-(7+5n)q,5q-2p,(5q-d-p)^{\times (2n+6)},d-p\bigr).\]
We now apply $c_{1,2,3}$ to $c_2$ and rearrange/drop the zeros to get
\[ v_3:=\bigl(d+2(1+n)p-(7+5n)q;\ d+2(2+n)p-(12+5n)q,(5q-d-p)^{\times (2n+5)},(d-p)^{\times 2}\bigr).\]
We next apply the composition $c_{1,2n+2,2n+3},\hdots c_{1,2,3}$ of $n+1$ Cremona moves to $v_3$. These  each have a defect of $2d-5q$, resulting in: 
\begin{align*}v_4 = \bigl(d(3+2n)+(2+2n)p-2q(6+5n);\ & d(3+2n)+2(2+n)p-(17+10n)q, \\ 
&\ \  \ (d-p)^{\times (2n+2)},(5q-d-p)^{\times 3},(d-p)^{\times 2}\bigr)\end{align*}
We can now apply $c_{2n+6,2n+5,2n+4}$ (which acts on the three terms of size $5q-d-p$) and use Lemma~\ref{lem:cremonaIdent} to simplify and rearrange, to obtain   
\[ v_5:=\bigl(2d(2+n)+p(3+2n)-(17+10n)q;\ \
d(3+2n)+2(2+n)p-(17+10n)q, 
(d-p)^{\times (2n+4)},0^{\times 3}\bigr).
\]
We then take the composition of the $n+2$ Cremona moves $c_{1,2n+4,2n+5}\hdots c_{1,2,3}$, which  each have a defect of $p-d$, 
to get the pair
\[ \bigl(d(2+n)+(5+3n)p-(17+10n)q;\  d(1+n)+3(2+n)p-(17+10n)q\bigr).\]
Using Lemma~\ref{lem:cremonaIdent}~(ii), this simplifies to:
\[ v_6:=\bigr(-d(3+n)+(1+n)p+5q; -d(4+n)+(2+n)p+5q\bigl).\]

We finish the proof by checking that $v_6 = (\ov{d}, \ov{m})$; in other words
\[ \ov{d}=-d(3+n)+(1+n)p+5q \quad \text{and} \quad \ov{m}=-d(4+n)+(2+n)p+5q.\]
But Lemma~\ref{lem:matrix} implies that
$(17+10n)\ov{p}+5\ov{q}=p$ and $(7+4n)\ov{p}+2\ov{q}=q.$ Hence writing $v_6$ in terms of $\ov{p}$ and $\ov{q}$ (first using Lemma~\ref{lem:cremonaIdent} to write $d$ in terms of $p,q$), we get that 
\[ v_6=\left(\frac{(2+n)\ov{p}+(3+n)\ov{q}}{5+2n};\frac{(1+n)\ov{p}+(4+n)\ov{q}}{5+2n}\right)=(\ov{d},\ov{m}),\]
where the last equality follows from Lemma~\ref{lem:cremonaIdent}. This completes the proof.
\end{proof}

\subsection{The classes are obstructive}

We now check that condition (i) in Proposition~\ref{prop:stair} holds.
In view of Lemma~\ref{lem:unobs} we know that each step $\bE_n$ is live at its center for $\Om_B$ where 
 $B=(1;\frac{m_1}{d},\hdots,\frac{m_n}{d}),$ and need to check that it remains obstructive for the limiting domain.

\begin{lemma} \label{lem:nontrivial}
    For $k,n \geq 0$, the classes $\bE_k(n)$ are obstructive at their centers $p_k(n)/q_k(n)$ for $c_{\bm{b}(n)}.$
\end{lemma}

\begin{proof}
We must show that
\begin{equation} \label{eq:nontrivEx}
     \frac{p_k}{d_k-\btm_k \cdot \bm{b}(n)} > \sqrt{\frac{p_k}{q_k \cdot \Vol}}
\end{equation}
where we have suppressed the $n$ in $p_k(n), q_k(n), \btm_k(n)$ for ease of notation and $\Vol:=\Vol(\Omega_n)$. 
The inequality \eqref{eq:nontrivEx} is equivalent to
\[ \Vol > \frac{(d_k-\btm_k\cdot \bm b(n))^2}{p_kq_k}=:R_k. \]
To show that this holds, we will first show that for a fixed $n$, $R_k$ is strictly increasing in $k$ and then show that $\lim_{k\to\infty}R_k=\Vol .$

Using equation~\eqref{eqn:bbn} and Lemma~\ref{lem:cremonaIdent}, we compute
\begin{align*} \btm_k \cdot \bm{b}(n)&=\left(d_k-q_k,q_k,(d_k-2q_k)^{\times 2}, (5q_k-d_k-p_k)^{\times (2n+5)}\right) \cdot \bm{b}(n)  \\
&=\frac{d_k}{2}+\be_n(d_k(5+3n)-(5+2n)p_k+5q_k).
\end{align*}
Using this expression for $\btm_k \cdot \bm{b}(n)$ and the expression for $d_k$ from Lemma~\ref{lem:cremonaIdent}, we get
\begin{align} \label{eq:obsIdent}
     R_k&=\frac{\bigl(-(2+n)p_k+\be_n(10n^2+42n+45)(p_k-3q_k)+(11+5n)q_k\bigr)^2}{(5+2n)^2p_kq_k} \\ \notag
     &=\frac{1}{(5+2n)^2}\left(\sqrt\frac{p_k}{q_k}(-(2+n)+1/\be)n(1/2-\Vol))+\sqrt\frac{q_k}{p_k}((11+5n)-3/\be_n(1/2-\Vol))\right)^2,
\end{align}
where we simplified the right hand side 
using  $\be_n^2(10n^2+42n+45)=1/2-\vol$ from  \eqref{eq:ident}. 
Define $g_1(n)$ and $g_2(n)$ as follows \[ (5+2n)g_1(n)=-(2+n)+\frac1{\beta_n}\left(\frac12-\Vol\right) \quad \text{and} \quad  (5+2n)g_2(n)=11+5n-\frac3{\beta_n}\left(\frac12-\Vol\right),\]
so that
\[ R_k=\left(g_1(n)\sqrt{\frac{p_k}{q_k}}+g_2(n)\sqrt{\frac{q_k}{p_k}}\right)^2.\]
Note, that $g_2(n)+3g_1(n)=1.$ 
Thus, if $z=p_k/q_k$, then
\[
        R_k=f(z)^2,
        \qquad
        f(z):=g_1(n)\sqrt z+(1-3g_1(n))\sqrt{1/z}.
\]

By Lemma~\ref{lem:recur} and Corollary~\ref{cor:CFEnk}, the centers $p_k/q_k$ decrease with $k$, and they lie in the interval $[2,3]$. Hence it is enough to show that $f$ is strictly decreasing on $[2,3]$ for a fixed $n$. We have that $2f'(z)=\frac{g_1(n)}{\sqrt{z}}-\frac{1-3g_1(n)}{z^{3/2}}$, so $f'(z) <0$ whenever
\[   z  < \frac{1}{g_1(n)}-3. \]
As $z \leq 3,$ we need to verify that $g_1(n) \leq \frac{1}{6}.$ One checks directly from the formula for $g_1(n)$ that
\[0<g_1(n)\leq g_1(0)<\frac16.\]
Hence $f'(z)<0$ for all $z\in[2,3]$. We conclude that $R_k$ is strictly increasing in $k.$ 

It remains to compute the limit of $R_k$. By construction, we have
\[ \bm b(n)=\lim_{k\to\infty}\frac{\btm_k}{d_k}.\]
Hence, we have that 
\[ \lim_{k\to\infty}\frac{d_k-\btm_k\cdot\bm b(n)}{d_k}=
1-\|\bm b(n)\|^2=\Vol.\]
Additionally, since the classes $\bE_k(n)$ are quasi-perfect, they satisfy
\[d_k^2-\|\btm_k\|^2-p_kq_k=-1.\]
Dividing by $d_k^2$ gives
\[ \lim_{k\to\infty}\frac{p_kq_k}{d_k^2}= 
1-\|\bm b(n)\|^2=\Vol.\]
Hence,
\[ \lim_{k \to \infty} R_k=\lim_{k\to\infty}
        \frac{(d_k-\btm_k\cdot\bm b(n))^2}{p_kq_k}=\frac{\Vol^2}{\Vol}=\Vol\]
        as desired. 
Since $R_k$ is strictly increasing and converges to $\Vol$, we have $R_k<\Vol$ for every finite $k$.
\end{proof}

\end{document}